\documentclass[10pt]{amsart}
\usepackage{amsmath,amsfonts,amssymb,amsthm,epsfig,verbatim,ifthen,graphicx}
\usepackage[all]{xy}
\usepackage{tikz}
\usetikzlibrary{automata,positioning}
 \numberwithin{equation}{section}
\usepackage{enumitem}
\theoremstyle{definition} 
 \textwidth =6in \usepackage{amsmath}
\usepackage{amsfonts}
\usepackage{amstext}
\usepackage{amsbsy}
\usepackage{amsopn}
\usepackage{amsxtra}
\usepackage{upref}
\usepackage{amsthm}
\usepackage{amsmath}
\usepackage{amssymb}
\usepackage{epsfig,verbatim,ifthen,graphicx}
\newtheorem{prop}{Proposition}[section]
\newtheorem{lema}{Lemma}[section]
\newtheorem{teo}{Theorem}[section]
\newtheorem{eje}{Example}[section]
\newtheorem{coro}{Corollary}[section]
\newtheorem{claim}{Claim}[section]

\theoremstyle{definition}
\newtheorem{defi}{Definition}[section]
\newtheorem{rem}{Remark}[section]

\def\R{{\mathbb R}}

\def\N{{\mathbb N}}

\def\F{{\mathcal F}}

\def\G{{\mathcal G}}
\def\H{{\mathcal H}}

\title[Hidden Gibbs measures on shift spaces over countable alphabets]{Hidden Gibbs measures on shift spaces over countable alphabets}
\date{\today}

\begin{thanks}
{This research was partly developed within the Ph.D. thesis research of C.L.  G.I.\ was partially supported by CONICYT PIA ACT172001 and by Proyecto Fondecyt 1150058.
C.L.\  and Y.Y\ were partially supported by CONICYT PIA ACT172001, Proyecto Fondecyt 1151368 , and Grupo de investigaci\'{o}n GI 172208/C  at Universidad del del B\'{i}o-B\'{i}o. }
\end{thanks}

\subjclass[2000]{37D35, 37D25}
\keywords{Thermodynamic formalism, Gibbs measures, Hidden Gibbs measures, almost-additive sequences}

\author{Godofredo Iommi} \address{Facultad de Matem\'aticas,
Pontificia Universidad Cat\'olica de Chile (PUC), Avenida Vicu\~na Mackenna 4860, Santiago, Chile}
\email{giommi@mat.puc.cl}
\urladdr{http://www.mat.puc.cl/\textasciitilde giommi/}

\author{Camilo Lacalle}
\address{Grupo de investigaci\'{o}n en Sistemas Din\'{a}micos y Aplicaciones-GISDA, Departamento de Matem\'atica, Facultad de Ciencias, Universidad del B\'{i}o-B\'{i}o, Avenida Collao 1202, Casilla 5C
Concepci\'on, Chile.}
\email{clacalle7@gmail.com}

\author{Yuki Yayama}
\address{Grupo de investigaci\'{o}n en Sistemas Din\'{a}micos y Aplicaciones-GISDA, Departamento de Ciencias B\'{a}sicas,  Universidad del B\'{i}o-B\'{i}o, Avenida Andr\'{e}s Bello 720, Chill\'{a}n, Chile}
\email{yyayama@ubiobio.cl}

\begin{document}

\begin{abstract} 
We study the thermodynamic formalism for particular types of sub-additive sequences on a class of subshifts over countable alphabets. The subshifts we consider include factors of irreducible countable Markov shifts under certain conditions. We show the variational principle for topological pressure. We also study conditions for the existence and uniqueness of invariant ergodic Gibbs measures and the uniqueness of equilibrium states. As an application, we extend the theory of factors of (generalized) Gibbs measures on subshifts on finite alphabets to that on certain subshifts over countable alphabets.
\end{abstract}

\maketitle

\section{Introduction}

Thermodynamic formalism is an area of ergodic theory which addresses the problem of choosing relevant invariant measures among the, sometimes very large, set of invariant probabilities. This theory was brought  from statistical mechanics into dynamics in the early seventies by Ruelle and Sinai among others \cite{Ru_book, Sinai}.  The powerful formalism developed to study equilibrium of systems consisting of a large number of particles (e.g. gases) has been surprisingly efficient to describe certain dynamical systems that exhibit complicated behavior. The theory has been developed in several directions. Originally the dynamical system was assumed to be defined on a compact set and the observable was a continuous function. Both assumptions have been relaxed over the years.
For example, Gurevich \cite{gu1, gu2, gs},  Mauldin and Urba\'nski \cite{mu, mu2} and Sarig \cite{s1, s3, s3} have developed thermodynamic formalism in the non-compact setting of countable Markov shifts. Since there exists a wide range of relevant dynamical systems that can be coded with countable Markov shifts, this theory has had relevant applications. Other extension of thermodynamical formalism to non-compact settings was developed by Pesin and Pitskel \cite{pepi}. In that case, the system is not assumed to have any Markov structure but it has to be the restriction of a continuous map defined on a compact set. Also, the observables have to have continuous extensions (therefore observables are assumed to be bounded). In a different direction, the theory was extended to consider not only a single observable but instead a sequence of them. Certain additivity assumptions were required on the sequence in order for the ergodic theorems to hold. This circle of ideas was called non-additive thermodynamic formalism. It was originally formulated by Falconer \cite{f} with the purpose of applying it in the study of the dimension theory of non-conformal dynamical systems. Ever since, different additivity assumptions have been considered in the sequence. For example, Barreria  \cite{b1, b2, b3} developed the theory assuming a strong additivity assumption called almost-additivity. Mummert \cite{m} also obtained results in this direction. Cao, Feng and Huang \cite{cao} studied the case in which the sequence was only assumed to be sub-additive. More generally, Feng and Huang \cite{FH} extended the theory to handle asymptotically sub-additive sequences.  Over the last few years, thermodynamic formalism for non-compact dynamical systems and sequences of observables has been developed. Iommi and Yayama \cite{iy1, iy2} have studied  thermodynamic formalism for almost-additive sequences on (non-compact)  countable Markov shifts. Also, K{\"a}enm{\"a}ki and Reeve \cite{KR} studied the formalism for sequences of potentials under weaker additivity assumptions but for the full shift over a countable alphabet.

In this paper, we further develop the theory. We consider particular types of sub-additive sequences on a fairly general class of subshifts. 
We call this class the class of \emph{countable sofic shifts}, where a
countable sofic shift is defined as the image of a countable Markov
shift under a one-block factor map with an additional condition (see Section \ref{ss:fac}).
This class therefore  generalizes the concept of a sofic shift over a
finite alphabet. We stress that this dynamical system is non-Markov and
it is defined
on  a non-compact space. Even in the case of a single observable,
several of our results are new, to the best of our knowledge. The types
of sub-additive
sequences we consider are generalizations of  continuous functions with
tempered variation on subshifts satisfying the weak specification
property (see
Section \ref{potentials} for details). In Section \ref{back}, we propose
a definition of the topological pressure and compare it with the
Gurevich pressure. Then we prove the corresponding
variational principle
in Theorems \ref{vp1} and \ref{thmsofic} in Section \ref{vpvp}.  
In particular, 
Section \ref{sofic} studies a variational principle  for  sequences with tempered variation
defined on  \emph{finitely irreducible subshifts} (see Definition
\ref{fi}) which
preserve a certain finiteness property found in compact spaces.  In
Section \ref{vp:spc}, the variational principle  is also studied in the
case when the Bowen sequences (see Definition \ref{bowen}) are defined
on countable Markov shifts which are not necessarily finitely irreducible.
We see that  if the topological pressure of the sequence considered in
Section \ref{vp:spc} is finite, then the space on which it is defined is
finitely irreducible.
Hence, this type of sequence is  suitable for studying Gibbs measures.
In Section \ref{sectiongibbs}, we show under some assumptions the existence and uniqueness of
Gibbs measures on finitely irreducible countable sofic shifts, together with uniqueness of the
Gibbs equilibrium
states (see Theorem \ref{main2}).  Our results extend those in
\cite{KR}, encompassing more general classes of sequences and far more
general dynamical systems.

Differences with the work in \cite{iy1, iy2} are discussed in Section
\ref{differ}.  In particular, not every almost-additive sequence studied in
\cite{iy1} is in the class of  sequences we study here (see Example
\ref{dif}). This phenomenon is different from what  is observed in the
compact case,
in which every almost-additive sequence satisfies the assumptions we
consider. Examples of the kinds of sequences we study are  presented and
 compared with almost-additive sequences in Section
\ref{sectionexample}, and these are studied especially with the
variational principle in Section \ref{vpvp}. 

One of the main applications of the thermodynamic formalism studied in this article  is to develop the theory of factors of Gibbs measures on shift spaces over countable alphabets. 
An important question in the area is to determine under which conditions the (generalized) Gibbs property  is preserved under a one-block factor map. For  Gibbs measures for continuous functions on subshifts over finite alphabets, this problem has been studied widely, for example, by Chazotte and Ugalde \cite{CUO,CU}, Kempton  and Pollicott  \cite{PK},  Kempton \cite{k}, Piranio \cite{pm}, Jung \cite{ju},  Verbitskiy \cite{v} and  Yoo \cite{jy}. 
For generalized Gibbs measures for sequences on subshifts over finite alphabets, this type of question has been addressed by Barral and Feng \cite{bf}, Feng \cite{Fe4} and Yayama \cite{Y1, Y}, especially in connection with dimension problems on non-conformal repellers.
In Section \ref{hiddeng}, we address this question in the (non-compact and non Markov)  context of finitely irreducible countable sofic shifts. 
Applying the results of Sections \ref{vpvp} and \ref{sectiongibbs}, in
Theorem \ref{hgibbs}  we show that under certain conditions a factor of
a unique
invariant Gibbs measure for an almost-additive sequence on a finitely
irreducible countable sofic shift is 
a Gibbs measure for a type of sequence we study in Section
\ref{potentials}.  The most general form of the variational principle concerning factor maps in this paper is given in Theorem \ref{ftemperedv}. The results in Section \ref{hiddeng} generalize some results  of
\cite{Y}.
Finally, in Section \ref{applications}, applications are given to the study of some problems in dimension theory, in particular,  product of matrices and the singular value function.

\section{Background}\label{back}

\subsection{Subshifts on countable alphabets and specification properties}
This section is devoted to recall basic notions of symbolic dynamics. We discuss countable Markov shifts, factor maps and different specification properties in this setting. For more details we refer the reader to \cite{LM, BP}.
 Let $(t_{ij})_{\N  \times \N}$ be a transition matrix of zeros and ones (with no row and no column made entirely of zeros). 
 The associated \emph{(one-sided) countable Markov shift} $(\Sigma, \sigma)$ is the set
\[ \Sigma:= \left\{ (x_n)_{n \in \N} : t_{x_{n}, x_{n+1}}=1  \text{ for every } n \in \N \right\}, \]
together with the shift map $\sigma: \Sigma  \to \Sigma $ defined by  $\sigma(x)=x'$, for $x=(x_n)_{n=1}^{\infty}, x'=(x'_n)_{n=1}^{\infty}$ with $x'_n=x_{n+1}$ for all $n\in \N$. If for every $(i,j) \in \N^2$ the transition matrix satisfies $t_{ij}=1$, then we say that the corresponding countable Markov shift is the \emph{full shift on a countable alphabet}.  

An \emph{allowable word} of length $n \in \N$ for $\Sigma$ is a string $i_1 \dots i_{n}$ where $t_{i_{j,} i_{j+1}}=1$ for every $j \in \{1, \dots, n-1\}$.  For each $n\in\N$, denote by $B_n(\Sigma)$ the set of allowable words of length $n$ of $\Sigma$. For $i_{1} \dots i_{n} \in B_n(\Sigma)$, we define a cylinder set $[i_1 \dots i_{n}]$ of length $n$ by 
\begin{equation*}
[i_1 \dots i_{n}]= \left\{x \in \Sigma: x_j=i_j \text{ for } 1 \le j \le n \right\}.
\end{equation*}

We endow $\Sigma$ with the topology generated by cylinder sets. This is a metrizable space. The following metric generates the cylinder topology. Let $d$  on $\Sigma$ by $d(x,x')={1}/{2^{k}}$ if
$x_i={x'}_i$ for all $1\leq i\leq k$ and $x_{k+1}\neq {x'}_{k+1},$ $d(x,x')=1$ if $x_1\neq x'_1$, and $d(x,x')=0$ otherwise.  We stress that, in general, $\Sigma$ is a non-compact  space. 

We can drop the Markov structure and define subshifts on countable alphabets.   Let $X$ be a closed subset of the full shift $\Sigma$. If $X$ is $\sigma$-invariant, that is $\sigma(X)\subseteq X$, then we say that  $(X, \sigma\vert_X)$ is a \emph{subshift} and we write $\sigma_X$ instead of $\sigma\vert_X$.  In particular, if $X$ is not a subset of the full shift on a finite alphabet, then we say that $(X, \sigma_X)$ is a \emph{subshift on a countable alphabet}.
We also  write  $(X, \sigma)$ for simplicity. The set $X$ is endowed with the topology induced by $\Sigma$. In this context the set of allowable words of length $n$ of $X$ is defined by
\begin{equation*}
B_n(X):= \left\{	i_{1} \dots i_{n} \in B_n(\Sigma) : [i_1 \dots i_{n}] \cap X \neq \emptyset		\right\}.
\end{equation*}
For an allowable word $w=i_{1} \dots i_{n}$ we denote by $|w|$ its length, $|i_{1} \dots i_{n}|=n$.  Given a subshift $(X, \sigma)$ on a countable alphabet, we now define the language of $X$.  The  word of length $n=0$ of $X$ is called the empty word and it is denoted by $\varepsilon$.
The language of $X$ is the set $B(X)=\bigcup_{n=0}^{\infty}B_n(X)$, i.e., the union of all allowable words of $X$ and the empty word $\varepsilon$.

We now define several notions of specification that generalize the one first introduced by Bowen \cite{bo} with the purpose of proving that there exits a unique measure of maximal entropy for a large class of compact subshifts.  Our definitions are given in terms of the language of $X$.

\begin{defi} \label{def:irr}
We say that a subshift $(X,\sigma)$ on a countable alphabet is \emph{irreducible} if for any allowable words $u, v \in B(X)$, there exists an allowable word $w\in B(X)$  such that $uwv\in B(X)$. 
\end{defi}
\begin{defi} \label{def:spec}
We say that a subshift $(X,\sigma)$ on a countable alphabet  has the \emph{weak specification property} if there exists $p\in\N$ such that for any allowable words 
$u, v \in B(X)$, there exist  $0\leq k\leq p$ and $w\in B_{k}(X)$ such that $uwv\in B(X)$. If in addition, $k=p$ for any $u$ and $v$, then $X$ has the \emph{strong specification property}. We call such $p$ a weak (strong, respectively) specification number. 
\end{defi}

\begin{defi} \label{fi}
A subshift $(X,  \sigma)$ is \emph{finitely irreducible} if there exist $p \in \N$ and a finite subset $W_1 \subset \bigcup_{n=0}^{p}B_n(X)$ such that  for every $u, v \in B(X)$,  there exists $w \in W_1$ such that
$uwv \in B(X)$.
\end{defi}

\begin{defi} \label{def:fp}
A subshift $(X,  \sigma)$ is \emph{finitely primitive} if there exist $p \in \N$ and a finite subset $W_1 \subset B_p(X)$ such that  for every $u, v \in B(X)$,  there exists $w \in W_1$ such that
$uwv \in B(X)$.
\end{defi}

\begin{rem} 
Note that the weak specification property does not imply topologically mixing. However, if  $(\Sigma, \sigma)$ is a topologically mixing subshift of finite type defined on a finite alphabet with the weak specification property, then it has the strong specification property (see \cite[Lemma 3.2]{j}). The class of general shifts on finite alphabets with the weak specification property include irreducible sofic shifts (see \cite{j} and Definition \ref{def:sofic}). \end{rem}

As it is clear from the definition, the notion of finitely primitive (see Definition \ref{def:fp}) is essentially the same as that of specification introduced by Bowen \cite{bo} in a non-compact symbolic setting. There is a closely related class of countable Markov shifts studied by Sarig \cite{s3}.

\begin{defi} \label{BIP}
A countable Markov shift  $(\Sigma, \sigma)$ is said to satisfy  the \emph{big images and preimages property (BIP property)} if 
there exists $\{ b_{1} , b_{2}, \dots, b_{n} \}$ in the alphabet $S$ such that for every
$ a \in S$ there exist $i,j \in \{1, \dots, n\}$ such that  $t_{b_{i}a}t_{ab_{j}}=1$.
\end{defi}

\begin{rem} If the countable Markov shift $(\Sigma, \sigma)$ satisfies the BIP property, then for every symbol in the alphabet, say $a$, there exist $b_i, b_j \in \{ b_{1} , b_{2}, \dots, b_{n} \}$ such that $b_ia$ and $ab_j$ are allowable words. Note, however, that a system with the BIP property can have more than one transitive component. Indeed, if $\Sigma$ is the disjoint union of two full shifts on countable alphabets, then it satisfies the BIP property and it has two transitive components. 
\end{rem}

Nevertheless, as noted by Sarig \cite[p.1752]{s3} and by Mauldin and Urba\'nski \cite{mu2}, under the following dynamical assumption both notions coincide. A countable Markov shift is \emph{topologically mixing}, i.e., for each pair $x,y\in \N$, there exists $N \in \N$ such that for every $n > N$ there is an allowable  word $i_1\dots i_{n}\in B_n(\Sigma)$ such that $i_1=x, i_{n}=y$.

\begin{lema}
If $(\Sigma, \sigma)$ is a  topologically mixing countable Markov shift  with the BIP property, then it is  finitely primitive.
\end{lema}

\begin{proof}
Let $a, c \in \mathcal{A}$ be two symbols of the alphabet. Since $(\Sigma, \sigma)$ is BIP, there exist $b_i, b_j \in \{ b_{1} , b_{2}, \dots, b_{n} \}$ in the alphabet, such that 
\begin{equation*}
ab_i \quad , \quad b_j c
\end{equation*}
are allowable words. Since $(\Sigma, \sigma)$ is topologically mixing, for each pair $b_l, b_r \in \{ b_{1} , b_{2}, \dots, b_{n} \},$  there exists $N_{l,r} \in \N$ such that for every  $k > N_{l,r}$ there is a word $w_{l,r}^k \in B_k(\Sigma)$ such that $b_l w_{l,r}^k b_r \in B_{k+2}(\Sigma)$. Let $N:= \max \{ N_{l,r} : l, r \in\{1, \dots ,n \}\} +1$ and consider the set
\begin{equation*}
\mathcal{F}:= \left\{b_j w_{j,i}^{N} b_i : i,j 	 \in\{1, \dots ,n \}		\right\}.
\end{equation*}
Then, for any pair of allowable words $u \in B_l(\Sigma), v \in B_m(\Sigma)$ there exists $b_j w_{j,i}^{N} b_i \in \mathcal{F}$ such that $u b_j w_{j,i}^{N} b_i v$ is an allowable word. The result now follows since every word in $\mathcal{F}$ has length $N+2$.
\end{proof}

\begin{rem}
Note that if $(\Sigma, \sigma)$ satisfies the strong specification property then it is topologically mixing and 
has infinite entropy. On the other hand, if $(\Sigma, \sigma)$ satisfies the weak specification property then it is irreducible and has infinite entropy (see Section \ref{vpvp}). 
\end{rem}

\subsection{Pressure for a class of sequences of continuous functions}\label{potentials}
In this section, we provide two definitions of pressure of sequences of continuous functions defined on  non-compact subshifts. 
We prove that under fairly general assumptions both coincide. 
Let $(X, \sigma)$ be a subshift on a countable alphabet. For each $n\in \N$, let  $f_{n}: X \to \R^{+}$ be a continuous function and $\mathcal{F}= \{ \log f_n \}_{n=1}^{\infty}$  a sequence of continuous functions on $X$.  In order to develop thermodynamic formalism and to be able to apply ergodic theorems, additivity assumptions are required on the sequences.
%
 
\begin{defi} \label{aaa}
A sequence $\mathcal{F}= \{ \log f_n \}_{n=1}^{\infty}$ of continuous functions on $X$ is called \emph{almost-additive} if there exists a constant $C\geq0$ such that for every $n,m\in \N, x\in X$, $\F$ satisfies
\begin{equation} \label{A2}
 f_{n+m}(x) \leq  f_n(x) f_{m}(\sigma^n x) e^{C}
  \end{equation}
and 
\begin{equation} \label{A1}
f_n(x) f_{m}(\sigma^n x) e^{-C} \leq f_{n+m}(x).  
\end{equation}
\end{defi}
In particular, $\F$ is called  \emph{sub-additive} if $\F$ satisfies  (\ref{A2}) with $C=0$ and $\F$ is \emph{additive} if $\F$ satisfies (\ref{A2}) and (\ref{A1}) with $C=0$. Note that we have  (\ref{A2}) if and only if the sequence $\F+C=\{\log (e^Cf_n)\}_{n=1}^{\infty}$ is sub-additive.
We also assume the following regularity condition. 
\begin{defi}\label{bowen}
A sequence $\mathcal{F}= \{ \log f_n \}_{n=1}^{\infty}$ of continuous functions on $X$ is called a \emph{Bowen}
sequence if there exists $M \in \R^{+}$ such that
\begin{equation}\label{bowenbound}
 \sup \{ M_n : n \in \NÊ\} \leq M,
\end{equation}
where
\[M_n= \sup \left\{ \frac{f_n(x)}{f_n(y)} : x,y  \in X, x_i=y_i \textrm{ for } 1 \leq i \leq n\right\}.\]
More generally, if $M_n<\infty$ for all $n\in \N$ and $\lim_{n \rightarrow \infty}(1/n)\log M_n=0$, then we say that $\mathcal{F}$ has  \emph{tempered variation}. Without loss of generality, we assume $M_{n}\leq M_{n+1}$ for all $n\in \N$. \end {defi}

\begin{rem} \label{rem:con_bow}
Definition \ref{bowen} extends a notion introduced by Walters \cite{w3} when developing thermodynamic formalism. We say that a continuous function  $f:X \to \R$  belongs to  the \emph{Bowen class} if the sequence $\{\log e^{S_n(f)}\}_{n=1}^{\infty}$, where $(S_nf)(x)=f(x)+f(\sigma(x))+\dots +f(\sigma^{n-1}(x))$ for each $x\in X$  is a Bowen sequence. The Bowen class contains the functions of summable variations  and the Bowen sequences are a generalization of functions in the Bowen class (see \cite{b2, iy1}).
\end{rem} 

We now list several assumptions we will use throughout the paper. These are hypothesis on both the system $(X, \sigma)$ and the sequence $\F$. 
\begin{enumerate}[label=(C\arabic*)]
\item \label{a0} The sequence $\F+C$ is sub-additive for some $C\geq 0$. 
\item There exist $p\in\N$ and $D>0$ such that  given any $u \in B_n(X), v \in B_m(X)$, $n, m\in\N$, there exists $w \in B_k(X), 0\leq k\leq p$ such that \label {a1} 
\begin{equation*}
\sup \{f_{n+m+k}(x):x \in [uwv]\} \geq D \sup \{f_n(x):x\in [u]\} 
\sup \{f_m(x):x \in [v]\}.
\end{equation*}
\item There exists a finite  set $W \subset \bigcup_{k=0}^p B_k(X)$ consisting of elements $w$ for which the property \ref{a1} holds. \label{a3}
\item$Z_1(\mathcal{F}):= \sum_{i \in \N } \sup \{f_{1}(x): x\in [i]\} < \infty$ \label{a2}. 
\end{enumerate}

In addition, we consider in Section \ref{sofic} sequences satisfying the following weaker condition.

\begin{enumerate}[label=(D\arabic*)]
\setcounter{enumi}{1}
\item \label{a4}  There exist $p\in\N$ and a positive sequence $\{D_{n,m}\}_{(n,m)\in \N \times \N}$ such that  given any $u \in B_n(X), v \in B_m(X)$, $n, m\in\N$, there exists $w \in B_k(X), 0\leq k\leq p$ such that 
$$\sup \{f_{n+m+k}(x):x \in [uwv]\} \geq D_{n,m} \sup \{f_n(x):x\in [u]\} 
\sup \{f_m(x):x \in [v]\},$$
where $\lim_{n\rightarrow\infty}(1/n)\log D_{n,m}=\lim_{m\rightarrow\infty}(1/m)\log D_{n,m}=0$.  Without loss of generality, we assume that $D_{n,m}\geq D_{n,m+1}$ and $D_{n,m}\geq D_{n+1,m}$.
\item There exists a finite  set $W \subset \bigcup_{k=0}^p B_k(X)$  consisting of elements $w$ for which the property \ref {a4} holds.  \label{a5}
\end{enumerate}
%

 If a sequence $\F$ on $X$ satisfies \ref{a1}  (\ref{a4}, respectively) with $w\in B_{p}(X)$ for all $w$, then we say that $\F$ on $X$ satisfies \ref{a1} (\ref{a4}, respectively) with the strong specification.

\begin{rem}\label{simplification}
Given a pair  $u\in B_n(X),v\in B_m(X)$, $n,m\in\N$, if \ref{a1} holds when $w=\varepsilon$, then we obtain that  $uv$ is an allowable word and  $\sup \{f_{n+m}(x):x \in [uv]\} \geq D \sup \{f_n(x):x\in [u]\}  \sup \{f_m(x):x \in [v]\}$. In particular, it is easy to see that if $(X,\sigma)$ is a subshift on a countable alphabet  and $\F$ is a Bowen sequence on $X$ satisfying \ref{a0} and \ref{a1}, then $W=\{\varepsilon\}$ in \ref{a3} if and only if $(X, \sigma)$ is the full shift on a countable alphabet and $\F$ is almost-additive on the full shift. The case when $\F$ is an almost-additive Bowen sequence on the full shift  has been studied in \cite{iy1}. 
\end{rem}
 \begin{rem}
Note that if conditions \ref{a1} or \ref{a4} are satisfied then $(X, \sigma)$ has the weak specification property. Moreover, if conditions \ref{a3} or \ref{a5} are satisfied then $(X, \sigma)$ is finitely irreducible.
 \end{rem}

%
%
%
We can now give the definitions of pressure.

\begin{defi}\label{defpressure}
Let $(X, \sigma)$ be an irreducible subshift on a countable alphabet and  $\F=\{\log f_n\}_{n=1}^{\infty}$ a sequence of continuous functions on $X$ with tempered variation  
satisfying \ref{a0}. Define $Z_n(\F)$ by
\begin{equation*}
Z_n(\F):=\sum_{i_1\dots i_n \in B_n(X)} \sup \left\{f_{n}(x): x\in [i_1\dots i_n]\right\}
\end {equation*}
and the \emph{topological pressure} of $\F$  by 
\begin{equation}\label{defp}
P(\F): =\limsup_{n\rightarrow\infty} \frac{1}{n}\log Z_n(\F),
\end{equation}
if $\limsup_{n\rightarrow\infty} (1/n)\log Z_n(\F)$ exists, including possibly $\infty$ and $-\infty$. 
\end{defi}

It is clear that  if $Z_1(\F)<\infty$ then sub-additivity of the sequence $\F+C$ implies that 
$P(\F)=\lim_{n\rightarrow\infty} (1/n)\log Z_n(\F)$ and $-\infty\leq P(\F)<\infty$. 
 We will see in Section  \ref{vpvp} that  
if $Z_1(\F)=\infty$, under certain additional assumptions on $(X, \sigma)$ and $\F$, we obtain $P(\F)=\infty$. 
The variational principal is studied for such sequences $\F$ in Section \ref{vpvp}. 

\begin{rem}
The topological pressure in Definition \ref{defpressure} is a natural extension of the classical definition of pressure for compact  subshifts. 
This definition was later extended by Mauldin and Urba\'nski \cite{mu} for countable Markov shifts satisfying the finitely irreducible condition.  
This notion of pressure was also extended for sequences of regular functions defined on subshifts of finite type by Barreira \cite{b1,b2,b3}, Falconer \cite{f}, Feng \cite{Fe1, Fe2, Fe3} and Cao, Feng and Huang \cite{cao} among others. 
Actually, assumption \ref{a1} was introduced by Feng \cite{Fe3} while studying thermodynamic formalism for potentials related to product of matrices and appeared also in the study of dimension of non-conformal repellers \cite{Fe4, Y1}. 
Moreover,  when $(X,\sigma)$  is a subshift on a finite alphabet,  Feng \cite{Fe4} studied thermodynamic formalism for the class  of sequences which satisfies \ref{a0} and  \ref{a1} (see Theorem \ref{feng}). Note that in this case \ref{a3} and \ref{a2} are automatically satisfied by compactness. K{\"a}enm{\"a}ki and Reeve \cite{KR} extended the work of Feng \cite{Fe3, Fe4}  to the full shift on a countable alphabet. They studied thermodynamic formalism for sequences of potentials defined on the full shift satisfying what they called \emph{quasi multiplicative} property. 
This  assumption on the sequences used in \cite{KR} is equivalent to assume conditions \ref{a0}, \ref{a1} with
$w\in \bigcup_{k=1}^p B_k(X)$, and \ref {a3} with $W \subset \bigcup_{k=1}^p B_k(X)$ on a Bowen sequence on the full shift. In Section \ref{differ}, we discuss the differences between almost-additivity and conditions \ref{a1} and \ref{a4}.
\end{rem}

Next we define the Gurevich pressure. Throughout  the paper, we identify the set of a countable alphabet with $\N$.

\begin{defi}\label{gurevich}
Let $(X, \sigma)$ be an irreducible subshift  on a countable alphabet and  $\F=\{\log f_n\}_{n=1}^{\infty}$ a sequence of continuous functions on $X$ with tempered variation satisfying \ref{a0} 
and \ref{a4}. 
For $a\in \N$, define 
\begin{equation*}
Z_n(\F, a):=\sum_{x:\sigma^{n}x=x}f_n(x)\chi_{[a]}(x),
\end{equation*}
where  $\chi_{[a]}(x)$ is a characteristic function of the cylinder $[a]$. 
The $\emph{Gurevich pressure}$ of $\F$ on $X$, denoted by $P_{G}(\F)$, is defined by
\begin{equation}\label{gurev}
P_{G}(\F):=\limsup_{n\rightarrow \infty} \frac{1}{n} \log Z_n(\F, a),
\end{equation}
if   $\limsup_{n\rightarrow \infty} (1/{n}) \log Z_n(\F, a)$ is independent of $a\in \N$.
\end{defi}
In Proposition \ref{gdefi}, we will study the definition of Gurevich pressure $P_{G}(\F)$
when $(X,\sigma)$ is a countable Markov shift and $Z_1(\F)<\infty$. If $Z_1(\F)=\infty$, under  certain assumptions on  $(X, \sigma)$ and $\F$, we obtain $P(\F)=P_{G}(\F)=\infty$ (see Section \ref{vpvp}). The definition is also studied in Section 
 \ref{sofic} when  $(X, \sigma)$ is a finitely irreducible countable sofic shift.

\begin{rem}
The Gurevich entropy was first introduced by Gurevich for countable Markov shifts. This notion  was later extended by Sarig \cite{s1} where he defines the Gurevich pressure of regular potentials defined on topologically mixing countable Markov shifts.  In \cite[Section 1]{ffy},
the definition was extended to a certain type of irreducible countable Markov shift.  It was shown by Dougall and Sharp in \cite[Section 3]{ds} that the definition could be extended to topological transitive shifts on countable alphabets for regular potentials. In all these cases, it was shown that the definition does not depend on the symbol $a$ chosen. The Gurevich pressure was defined and studied for almost-additive sequences  on topologically mixing countable Markov shifts by Iommi and Yayama \cite{iy1}. 
We stress that the definition given here extends both the class of sequences of potentials and the class of shifts (satisfying the weak specification)
previously considered in the literature.
\end{rem}

It was shown by Mauldin and Urba\'nski \cite{mu2} and by Sarig \cite{s3} that when restricted to topologically mixing countable Markov shifts satisfying the BIP property for a regular potential, Definitions \ref{defpressure} and \ref{gurevich} coincide. The next result extends this observation to countable Markov shifts satisfying the weak specification  property and to sequences of functions satisfying mild additivity assumptions.

\begin{prop}\label{gdefi}
Let $(X, \sigma)$ be a countable Markov shift 
and
$\F=\{\log f_n\}_{n=1}^{\infty}$ a sequence on  $X$ with  tempered variation  satisfying
\ref{a0} and \ref{a4}.  If $P(\F)<\infty$,  then
\begin{equation}\label{firstdef}
P(\F)= P_{G}(\F).
\end{equation}
If $\F$  satisfies \ref{a4}  with the strong specification, then  $\limsup$ in (\ref{gurev}) can be replaced by 
$\lim$.
\end{prop}

\begin{proof}
  First we observe that  $P(\F)<\infty$ if and only if $Z_1(\F)<\infty$ (see Proposition \ref{cha}). 
Let $a \in \N$ be fixed and  $c_n:=x_1\dots x_n\in B_n(X)$. By assumption 
\ref{a4} there exist  allowable words $w_1, w_2$ with $0\leq \vert w_1 \vert, \vert w_2\vert \leq p$, such that $aw_1x_1\dots x_nw_2a$ is an allowable word of length $n+2+\vert w_1 \vert + \vert w_2\vert$ satisfying 
\begin{equation*}\label{keyforsub}
\begin{split}
&\sup \{f_{n+2+\vert w_1 \vert + \vert w_2\vert}(x): x\in [aw_1c_nw_2a]\} \\&
\geq 
D_{1,n}D_{1+p+n, 1}\sup \{f_n(x): x\in [c_n]\}(\sup \{f_{1}(x): x\in [a]\})^2.
\end{split}
\end{equation*}
Since $\F$ has tempered variation, for any $x \in [aw_1c_nw_2a]$ we have that
\begin{equation*}
\begin{split}
&\sup\{f_{n+2+\vert w_1 \vert + \vert w_2\vert}(x): x\in [aw_1c_nw_2a]\} \\
&\leq M_{n+2p+2}f_{n+2+\vert w_1 \vert + \vert w_2\vert}(x) \leq M_{n+2p+2}f_{n+1+\vert w_1 \vert + \vert w_2\vert}(x)\sup \{f_{1}(x): x\in [a]\}e^C.
\end{split}
\end{equation*}
Since $\bar{x}=(aw_1c_nw_2)^{\infty}=(aw_1c_nw_2aw_1c_nw_2aw_1c_nw_2\dots)$ is a periodic point with  period $n+\vert w_1 \vert + \vert w_2\vert +1$, 
we obtain

\begin{equation*}
\begin{split}
f_{n+\vert w_1 \vert +\vert w_2 \vert+1}(\bar{x})\geq \frac{D_{1,n}D_{1+p+n, 1}e^{-C}}{M_{n+2p+2}} \sup \{f_{n}(x): x\in [c_n]\}\sup \{f_{1}(x): x\in [a]\}.
\end{split}
\end{equation*}
Note that since $\F$ has tempered variation we have that $\sup \{f_{1}(x): x\in [a]\}$ is bounded. 
Setting $d_n=(D_{1,n}D_{1+p+n, 1}\sup \{f_{1}(x): x\in [a]\})/(e^C M_{n+2p+2})$ and summing over all allowable words  
$c_n=x_1\dots x_n\in B_n(X)$, we obtain 
\begin{equation}\label{ii1}
\sum_{i=n+1}^{n+2p+1}Z_i(\F, a)\geq d_n Z_n(\F)>0.
\end{equation}
Hence, there exists $n+1\leq i_n\leq n+2p+1$ such that $Z_{i_n}(\F, a)\geq (d_n Z_n(\F))/(2p+1)$.
Therefore,
\begin{equation*}
\frac{1}{i_n}\log Z_{i_n}(\F, a)\geq \frac{1}{n+2p+2} \left(\log \frac{1}{2p+1}+\log d_n +\log Z_n(\F) \right).
\end{equation*}
Thus 
\begin{equation}\label{eneq}
\limsup_{n\rightarrow\infty}\frac{1}{i_n}\log Z_{i_n}(\F, a)\geq P(\F).
\end{equation}
Since $Z_{i_n}(\F, a)\leq Z_{i_n}(\F)$ for all $i_n$ and  $a$ is arbitrary, (\ref{eneq}) implies (\ref{firstdef}).

Next we show the second part.
If $\F$ satisfies  \ref{a4}  with the strong specification, we obtain for all $n\in\N$
\begin{equation*}
Z_{n+2p+1}(\F, a)\geq d_n Z_n(\F)>0.
\end{equation*}
Thus similar arguments above imply that 
$$ \limsup_{n\rightarrow\infty}\frac{1}{n}\log Z_n (\F,a)=\liminf_{n\rightarrow\infty}\frac{1}{n}\log Z_n (\F,a).$$
In particular one can take a limit instead of a limsup in the definition of Gurevich pressure.
 \end{proof}


\begin{rem}\label{scase} 
 In Section \ref{vpvp}, we  obtain (\ref{firstdef}) when $Z_1(\F)=\infty$ under certain assumptions on $(X, \sigma)$ and $\F$. 
In Section \ref{sofic},  for a sequence $\F$ on a finitely irreducible countable sofic shift we establish  conditions ensuring  $P(\F)=P_G(\F)$.
\end{rem}

\begin{rem}[Entropy]
A particular case of the definitions considered in Section \ref{potentials} is when the sequence $\F=\{\log f_n\}_{n=1}^{\infty}$ is such that for every $n \in \N$ we have that $\log f_n=0$. In this case we denote $\F=0$. 
The numbers $P(0)$ and $P_G(0)$ are called the \emph{entropy} and the \emph{Gurevich entropy} respectively. It is well known that for a compact irreducible sofic shift  (see Definition \ref{def:sofic}) both notions coincide (see \cite[Theorem 4.3.6]{LM}). 
However, even for topologically mixing countable Markov shifts these two notions can be different, we can have $P_G(0) < P(0)$. In Proposition \ref{gdefi}, fairly general conditions are obtained so that we can still have an equality $P(0)=P_G(0)$ in the non-compact setting. We use the following notation $P(0)=h(\sigma)$ and $P_G(0)=h_G(\sigma)$.
\end{rem}

\subsection{Factor maps} \label{ss:fac}
The goal of this section is to study certain subshifts which are images of countable Markov shifts under factor maps.  
The following class of maps will play an important role in this article.

\begin{defi} \label{def:factor}
Let $(X, \sigma_X)$ and $(Y, \sigma_Y)$ be subshifts on countable alphabets. 
A \emph{one-block code}  is a map $\pi:X\rightarrow Y$ for which there exists a function, denoted again by $\pi$, $\pi:B_1(X)\rightarrow B_1(Y)$ such that 
$(\pi(x))_i=\pi(x_i)$ for all $i\in\N$. For $u=x_1 \dots x_k \in B_k(X)$, $k\in\N$, we denote $\pi(x_1)\dots \pi(x_k)\in B_k(Y)$ by $\pi(u)$. 
A map
$\pi:X\rightarrow Y$ is a \emph{factor map} if it is continuous, surjective and satisfies $\pi \circ \sigma_X = \sigma_Y \circ \pi$. 
For a one-block factor map $\pi:X\rightarrow Y$ where $X$ is an irreducible countable Markov shift,  let $v \in B_k(Y)$. We denote by $\pi^{-1}(v)$ 
the set of allowable words $u$ of length $k$ of $X$ such that $\pi(u)=v$ and
by $\vert  \pi^{-1}(v) \vert $ the cardinality of the set. 
Throughout the paper, we only consider one-block factor maps $\pi: X\rightarrow Y$ such that $\vert  \pi^{-1}(i) \vert <\infty$ for any $i\in\N$. 
Hence for each $k\in \N$,
$v\in B_k(Y)$,  we have
$\vert  \pi^{-1}(v) \vert <\infty$. 

\end{defi}
Next we show that, in some cases, the image of a countable Markov shift under a one-block code is a subshift. In general, the image of a shift space on a countable alphabet under a sliding block code is not closed and hence it is not a subshift (see \cite{LM}). 

\begin{lema} \label{closed}
Let $(X, \sigma_X)$ be a subshift on a countable alphabet and $(\Sigma, \sigma)$ the full shift on a countable  alphabet. 
Let $\pi:X\rightarrow \Sigma$ be a one-block code 
such that $\vert  \pi^{-1}(i) \vert <\infty$ for each $i\in\N$. Let $Y:=\pi(X)$. Then $(Y, \sigma_Y)$ is a subshift  on a countable alphabet.\end{lema} 
\begin{proof}
It is easy to see that $Y$ is invariant and we show that $Y$ is closed. For $m\in\N$, let $y^{(m)}=\{y_n^{(m)}\}_{n=1}^{\infty}\in Y$. Let  $\{y^{(m)}\}_{m=1}^{\infty}$ be a sequence in $Y$ converging to $y=\{y_i\}_{i=1}^{\infty}$.  We show that $y\in Y$. Since $Y$ is the image of $X$ under $\pi$, for each 
$m\in \N$, we can pick an $x^{(m)}\in X$ such that  $\pi(x^{(m)})=y^{(m)}$
and let  $x^{(m)}=\{x_n^{(m)}\}_{n=1}^{\infty}$. Fix $l\in\N$. Since $\{y^{(m)}\}_{m=1}^{\infty}$ converges to $y\in Y$, there exists $M\in \N$ such that $d(y^{(m)}, y)<1/2^l$ for all $m\geq M$. Then  we have  $y_i^{(m)}=y_i$ for all $m\geq M$, $1\leq i\leq l+1. $
Note that $ \pi^{-1}(y_i^{(M)})$ is a finite set for each $1\leq i\leq l+1$.
Consider  the sequence $\{x^{(m)}\}_{m=M}^{\infty}$. 
Then we have $x_{i}^{(m)}\in \pi^{-1}(y_i^{(M)})$ for $1\leq i\leq l+1, m\geq M$. Since there are finitely many symbols in $ \pi^{-1}(y_1^{(M)})$, there exists $x_1^{*} \in \pi^{-1}(y_1^{(M)})$ such that $x_1^{*}$  is the initial symbol of $x^{(m)}$, for infinitely many $m\geq M$. Now we extract a subsequence 
$\{x^{1, n}\}_{n=1}^{\infty}$ of sequences with the initial symbol $x_1^{*}$ from $\{x^{(m)}\}_{m=M}^{\infty}$. Define $\{x^{0, n}\}_{n=1}^{\infty}:=\{x^{(m)}\}_{m=M}^{\infty}$.
Repeating this process, for each $ 1\leq i\leq l+1$, there exist $x_i^{*}\in \pi^{-1}(y_i^{(M)})$ and 
a sequence $\{x^{i, n}\}_{n=1}^{\infty}$ of sequences with the $i$ th symbol $x_i^{*}$ such that $\{x^{i, n}\}_{n=1}^{\infty}$ is a subsequence of $\{x^{i-1, n}\}_{n=1}^{\infty}$. 
We define  $x_i^{*}$ for $i=l+i$,$i\geq 2$ similarly.
Given $l+1$, there exists $M_1$ such that $d(y^{(m)}, y)<1/2^{l+1}$ for all $m\geq M_1$. Then  we have  $y_i^{(m)}=y_i$ for all $m\geq M_1$, $1\leq i\leq l+2.$
Consider the sequence $\{z^{l+1, n}\}_{n=1}^{\infty}:=\{x^{l+1, n}\}_{n=1}^{\infty} \cap \{x^{m}\}_{m=M_1}^{\infty}$.   Since there are finitely many symbols in $ \pi^{-1}(y_{l+2}^{(M_1)})$, there exists $x_{l+2}^{*}\in \pi^{-1}(y_{l+2}^{(M_1)})$ such that $x_{l+2}^{*}$  is the $(2+l)$ th symbol of  $\{z^{l+1, n}\}_{n=1}^{\infty}$ for infinitely many $n$. Now we extract a subsequence $\{x^{l+2, n}\}_{n=1}^{\infty}$ of  sequences with the $(2+l)$ th symbol $x_{l+2}^{*}$ from $\{z^{l+1, n}\}_{n=1}^{\infty}$.
Since $\vert  \pi^{-1}(k) \vert <\infty$ for each $k\in\N$, by repeating this process, for each $i\geq 2$ there exist
$x_{l+i}^{*}\in \pi^{-1} (y_{l+i})$ and  a sequence $\{x^{l+i, n}\}_{n=1}^{\infty}$ with the $i$ th symbol $x_i^{*}$, each of which is a subsequence of $\{x^{l+i-1, n}\}_{n=1}^{\infty}$.
Define $x^{*}=\{x_i^{*}\}_{i=1}^{\infty}$.
By a diagonal argument, it is clear that the sequence $\{x^{l+i, i}\}_{i=1}^{\infty}$
converges to $x^{*}$. Since  $X$ is closed, we obtain that $x^{*}\in X$. Then $\pi (x^{*})=\{\pi(x_i^{*})\}_{i=1}^{\infty}=\{y_i\}_{i=1}^{\infty}=y$. Hence $Y$ is closed.

\end{proof}

In Lemma \ref{closed}, if $X$ is a countable Markov shift, then $\pi: X\rightarrow Y$ is a one-block factor map. Hence
we find a  class of subshifts which generalize countable Markov shifts.  Recall that if $(X, \sigma)$ is a finite state Markov shift, then the image of $X$ under a one-block 
factor map is a sofic shift \cite{LM, BP}. In the following, we generalize this definition to the case when $(X, \sigma)$ is a countable Markov shift.

\begin{defi} \label{def:sofic}
A \emph{countable sofic shift} is a subshift on a countable alphabet which is the image of a countable Markov shift under a one-block factor map $\pi$ such that $\vert \pi^{-1}(i)\vert<\infty$ for each $i\in \N$. In particular, an \emph{irreducible countable sofic shift} is the image of an irreducible countable Markov shift.

\end{defi}

\begin{rem}\label{remsofic}
Note that an irreducible subshift is defined in Definition \ref{def:irr}. In Definition \ref{def:sofic}, in order for $Y$ to be an irreducible countable sofic shift,  we  additionally assume that  it is an image of an irreducible countable Markov shift.   
\end{rem}

It is well known that if $X$ and $Y$ are subshifts on finite alphabets such that there exists a factor map $\pi:X \rightarrow Y$, then $h(X) \geq h(Y)$. In the non-compact case, this is in general not true (see the discussion in   \cite[Section 13.9]{LM} ). However, the next lemma shows that under suitable assumptions this property still holds.


\begin{lema}
Let  $(X, \sigma_X)$ and  $(Y, \sigma_Y)$ be 
topologically mixing countable Markov shifts and $\pi: X\to Y$ a one-block factor map such that 
$|\pi^{-1}(n)|< \infty$  for every $n \in \N$. Then $h(\sigma_X) \geq h(\sigma_Y)$.
\end{lema}

\begin{proof}
Recall that the Gurevich entropy  satisfies the following approximation property by compact sets \cite{gu1, gu2} 
\begin{align*}
h_G(\sigma_X)&= \sup \{ h(\sigma_X |_{K}) : K \subset X \text{ compact and invariant} \} \\
&=\sup \{ h(\sigma_X |_{\Sigma_K}) : \Sigma_K \subset X \text{ topologically mixing finite Markov shift} \}. 
\end{align*}
Since for every $n \in \N$ we have that   $|\pi^{-1}(n)|< \infty$, for every $\Sigma_K \subset Y$  topologically  mixing finite Markov shift we have that $\pi^{-1}(\Sigma_K)$ is a compact subshift of $X$. Therefore, by \cite[Proposition 4.16]{kit} we have that 
\begin{equation*}
h_G(\sigma_X |_{ \pi^{-1}(\Sigma_K)}) \geq h_G(\sigma_Y |_{ \Sigma_K}).
\end{equation*}
The result now follows.
\end{proof}



\section{Examples}\label{sectionexample}
In this section,  we study the types of sequences introduced in \ref{potentials} and present some examples.

\subsection{Differences between the \ref{a1} condition and almost-additivity} \label{differ}
This section is devoted to study the relations and differences between the additivity assumptions we have considered. That is, we establish relations between almost-additivity and conditions \ref{a1} and \ref{a4} introduced in Section \ref{potentials}. The results depend upon the combinatorial structure of the shifts. 
\begin{rem} \label{r:fin}
If $(X, \sigma)$ is an irreducible Markov shift defined on a finite alphabet (compact), then any almost-additive Bowen sequence on $X$ satisfies \ref{a1}. 
\end{rem}

Next lemma shows that the result in Remark \ref{r:fin} also holds for a finitely irreducible subshift on a countable alphabet. Even more, under weaker regularity assumptions it is possible to prove that an almost-additive sequence satisfies condition \ref{a4}.

\begin{lema}\label{d}
Let $(X, \sigma)$ be a finitely irreducible subshift on a countable alphabet and $\mathcal{G}=\{\log g_n\}_{n=1}^{\infty}$ an almost-additive sequence on $X$ with tempered variation. Then $\mathcal{G}$  satisfies \ref{a0}, \ref{a4} and \ref{a5}. 
If $\mathcal{G}$ is an almost-additive Bowen sequence on $X$, then it satisfies \ref{a0},  \ref{a1} and \ref{a3}. 
\end{lema}

\begin{proof}
Since $(X, \sigma)$ is a finitely irreducible subshift on a countable alphabet, there exist $p \in\N$ and a finite set $W_1 \subset \bigcup_{i=0}^{p} B_i(X)$ such that for any  $n,m\in\N$  and  $u\in B_{n}(X), v \in B_{m}(X)$ there exists $ w \in W_1$ such that $uwv$ is an allowable word. 
 Since $W_1$ is a finite set and $\mathcal{G}$ has tempered variation, there exists $Q_1>0$ such that 
 \begin{equation*}
\sup_{w\in W_1, \vert w\vert \geq 1} \left\{g_{\vert w \vert}(y): y \in [w] \right\}>Q_1.
\end{equation*}
 For $n\in\N$, let $M_n$ is defined as in Definition \ref{bowen}. Let $x\in [uwv]$, where $\vert w\vert=k\geq 1$. Then
\begin{equation}\label{p2}
 g_{n+m+k}(x)\geq e^{-C}g_{n}(x)g_{k}(\sigma^{n}x)g_{m}(\sigma^{k+n}x) \geq \frac{e^{-C}Q_1}{M_p}g_{n}(x)g_{m}(\sigma^{k+n}x).
 \end{equation}
 Now consider a pair $u\in B_n(X), v\in B_m(X)$ such that $uv$ is an allowable word. If $x\in [uv]$, then  we obtain $g_{n+m}(x)\geq e^{-C}g_{n}(x)g_{m}(\sigma^{n}x)$. Let $Q=\min \{Q_1, 1\}$.  Then \ref{a4} holds in particular for $p$ equal to the same $p$ that appears in the specification property and we obtain the result by setting $D_{n,m}=(e^{-C}Q)/(M_p M_nM_m)$ in  \ref{a4} and $W=W_1$ in \ref{a5}.  If the sequence $\mathcal{G}$ is an almost-additive Bowen sequence, the same argument   
 replacing $M_p,  M_n$ and  $M_m$ by $M$ yields the desired result.
 \end{proof}
 

\begin{lema}\label{variation}
Let $(X, \sigma)$ be a subshift on a countable alphabet,  $\G=\{\log g_n\}_{n=1}^{\infty}$ an almost-additive sequence on $X$ with tempered variation, and 
 $\F=\{\log f_n\}_{n=1}^{\infty}$ a sequence on $X$ satisfying \ref{a0},\ref{a4} and \ref{a5}.  Define $\H:=\{\log (f_n/g_n) \}_{n=1}^{\infty}$.
 Then $\H$ satisfies  \ref{a0},\ref{a4} and \ref{a5}.  
 \end{lema} 
\begin{proof}
The proof is straightforward. We use the similar approach as in Lemma \ref{d}.
\end{proof}

 \begin{eje}\label{basic1}\textnormal{\textbf{A continuous function on a finitely irreducible subshift  with tempered variation.}}
In this example, we show that the the formalism developed in this article generalizes results concerning continuous potentials satisfying mild regularity assumptions. Let $f$ be a continuous function  defined on a finitely irreducible subshift $X$. Denote by 
\begin{equation*}
A_n:=\sup \left\{\vert (S_nf)(x)-(S_nf)(y)\vert: x_i=y_i, 1\leq i\leq n \right\}.
\end{equation*}
We say that $f$ has \emph{tempered variation} if  $A_n<\infty$ for all $n\in \N$ and $\lim_{n\rightarrow\infty}(1/n) A_n=0$.
We remark that sometimes (see for example \cite{ffy}) the definition of tempered variation is given without the finiteness assumption $A_n < \infty$. We stress that in this paper we always do assume finiteness. 

Let $f$ be a continuous function on a finitely  irreducible subshift $X$ with tempered variation. Following the procedure described in Remark \ref{rem:con_bow}, for each $n\in \N$, define $f_n(x)=e^{(S_nf)(x)}$ and $\F=\{\log f_n\}_{n=1}^{\infty}$. The sequence $\F$ is additive. Moreover,  by  Lemma \ref{d}, $\F$ satisfies 
\ref{a4} and \ref{a5}.
\end{eje}

\begin{eje} \textnormal{\textbf{An almost-additive sequence on a countable Markov shift which does not satisfy \ref{a1}.}}\label{dif}
Let $A$ be a transition matrix on a countable alphabet defined by 
\begin{equation*}
   \mathbf{A} = \left(
     \begin{array}{cccccccc}
       1 & 0 & 0 & 1 & 0 & 0  & 1 & \ldots \\
       0 & 1 & 1 & 1 & 0 & 0  & 0 & \ldots \\
       0 & 1 & 1 & 1 & 0 & 0  & 0  & \dots \\
       1 & 1 & 1 & 1 & 0 & 0  & 0 & \dots \\
        0 & 0 & 0 & 0 & 1 & 1  & 1 &\dots \\
       0 & 0 & 0 & 0 & 1 & 1  & 1 &\dots \\
       1 & 0 & 0 & 0 & 1 & 1  & 1  & \dots \\
     \vdots & \vdots & \vdots& \vdots & \vdots & \vdots  & \vdots & \ddots         
      \end{array} \right)
 \end{equation*}
 and consider the countable Markov shift $(X, \sigma)$  determined by $A$ (see Figure \ref{fig0}).  Let $\{\lambda_n\}_{n=1}^{\infty}$ be a sequence of real numbers such that 
$\lambda_n \in (0,1)$ and $\sum_{j=1}^{\infty} \lambda_j < \infty$. Let   $\{\log c_n\}_{n=1}^{\infty}$ be an almost-additive  sequence of real numbers, that is, there exists a constant $C>0$ such that  
\[ e^{-C}c_n c_m \leq  c_{n+m} \leq  e^Cc_n c_m.\]
For $n\in \N$, define $g_n:\Sigma \rightarrow \R$ by  
\begin{equation*}
 g_n(x)=  c_n \lambda_{i_1} \lambda_{i_1} \cdots \lambda_{i_{n}}, \text{ for } x \in[ i_1 \dots i_n],
\end{equation*}
and let $\mathcal{G}=\{ \log g_n \}_{n=1}^{\infty}$. These sequences have been studied in \cite[Example 1]{iy1} when defined on the full shift.

 \begin{tikzpicture}[shorten >=1pt,node distance=1.5cm,on grid,auto]
   \node[state] (0) {$1$};
   \node[state] (1) [right=of 0] {$2$};
   \node[state] (2) [right=of 1] {$3$};
   \node[state] (3) [right=of 2] {$4$};
   \node[state] (4) [right=of 3] {$5$} ;
  \node[state] (5) [right=of 4] {$6$} ;
 \node [state] (6) [right=of 5] {$7$} ;
\node (7) [right=of 6] {$\cdots$} ;
\node (8) [right=of 6] {} ;
\node (9) [right=of 6] {} ;


   \path[->]
    (0) edge   [loop below]   (0) 
    (0) edge  [bend left]   (3) 
    (0) edge  [bend left]   (6)
    (0) edge  [bend left]   (9)
    (1) edge                 (2)
    (1) edge [loop below]    (1)
    (1) edge  [bend left]   (3)
    
    (2) edge [loop below]    (2)
    (2)  edge    (1)
    (2)  edge    (3)
    
    (3)  edge    [bend right]     (0)
     (3) edge [loop below]    (3)
     (3)  edge                 (2)
    (3)  edge    [bend right]     (1)
    
    (4) edge                 (5)
    (4) edge [loop below]    (4)
   (4) edge  [bend left]   (6)

    (5) edge [loop below]    (5)
    (5)  edge    (4)
    (5)  edge    (6)
    
    (6)  edge    [bend right]     (0)
     (6) edge [loop below]    (6)
     (6)  edge                 (5)
    (6)  edge    [bend right]     (4)

     (9)  edge    [bend right]     (0)
      ;
\end{tikzpicture}
\begin{figure}[h]
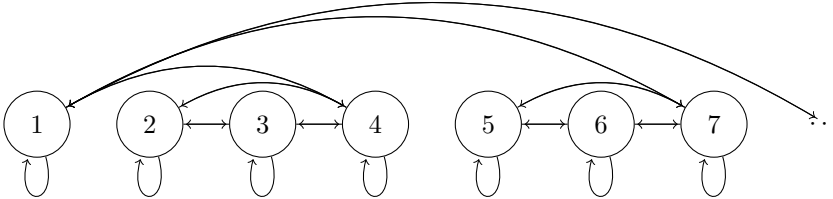

{} \caption{The graph defining  $X$ in Example \ref{dif}}
\label{fig0}
\end{figure}
\end{eje}

\begin{lema}
The sequence $\mathcal{G}=\{\log g_n\}_{n=1}^{\infty}$ defined on $X$ is an almost-additive Bowen sequence. However, it does not satisfy \ref{a1}. 
\end{lema}
\begin{proof}
It is clear that   $\mathcal{G}$ is an almost-additive Bowen sequence. Observe that $(X, \sigma)$ is topologically mixing and that $3$ is a strong specification number and, moreover, $X$ is not finitely irreducible. 
 \begin{claim}\label{ourp}
Let $(X, \sigma)$ be a subshift on a countable alphabet and  $\mathcal{F}=\{\log f_n\}_{n=1}^{\infty}$  a Bowen sequence on $X$ satisfying  \ref{a0} and \ref{a1}.  
Let $w\in \bigcup_{i=1}^p B_i(X)$ be an allowable word from \ref{a1}.
Then there exists $C'>0$ such that for any $w$ of length $k$, we have  $\sup\{f_k(x):x\in [w]\}\geq C'$.
\end{claim}

\begin{proof}
Since \ref{a1} is satisfied, given  $u\in B_n(X), v\in B_m(X)$, there exist $0\leq k\leq p$ and $w=w_1\dots w_k\in B_k(X)$ with the property 
 \begin{equation}\label{p1}
 \sup\{f_{n+m+k}(x):x\in [uwv]\}\geq D \sup\{f_n(x):x\in [u]\} \sup\{f_m(x):x\in [v]\}.
 \end{equation}
We consider only $uwv$ with the length $k$ of $w\geq 1$. 
For any  $x\in [uwv]$, it is a consequence of (\ref{p1}),  \ref{a0}  and the Bowen property of $\F$ that,
 \begin{equation*}
 Me^{2C}f_{n}(x)f_{k}(\sigma^{n}x)f_{m}(\sigma^{k+n}x)\geq D f_n(x)f_m(\sigma^{k+n}x).
 \end{equation*}
 Hence
 \begin{equation*}
 \sup\{f_k(x):x\in [w]\}\geq \frac{D}{Me^{2C}}=C'.
 \end{equation*}
\end{proof}

Assume by way of contradiction that the sequence  $\mathcal{G}$ satisfies \ref{a1} for some $p \in \N$. Consider the symbol $3$ and $3n$ for some $n\in\N$. To connect $3$ and $3n$, the symbol $3n+1$ must be passed through. Suppose  $w=w_1\dots w_k$ is a word of length $k\leq p$ such that $3w(3n)$ is allowable and satisfies \ref{a1}. Then $3n+1$ must appear in some $w_i, 1\leq i\leq k$. Clearly $k\geq 1$. Since $ \lambda_{j}$ is bounded above by some constant $C''>1$ for all $j\in\N$, we obtain 
$$\sup\{g_k(x):x\in [w]\}
\leq \max_{1\leq k\leq p}\{c_p\}{C''}^{p-1}\lambda_{3n+1}.$$
Applying Claim \ref{ourp}, $\lambda_{3n+1}$ is bounded below by a constant for all $n\in\N$. However  by the construction of $ \lambda_{j}$, $\lim_{n\rightarrow \infty} \lambda_{3n+1}=0.$ This contradiction proves the lemma.
\end{proof}

\begin{eje} \textnormal{\textbf{A sequence satisfying  \ref{a0},\ref{a1} and \ref{a3}.}}\label{exfactor}
In this example, we will make use of the notion of factor map (see Section \ref{ss:fac}).  
Let $(X, \sigma_X)$, $(Y,\sigma_Y)$ be subshifts on countable alphabets, and $\pi:X\rightarrow Y$ be a one-block factor map such that $\vert \pi^{-1}(i)\vert<\infty$, for every $i\in \N$. Define $\phi_n:Y\rightarrow \R$ by $\phi_n(y)=\log \vert \pi^{-1}(y_1\dots y_n)\vert$ and $\Phi=\{\log \phi_n\}_{n=1}^{\infty}$. Then $\Phi$ is a Bowen sequence. In the next lemma, we prove that under suitable assumptions on $X$ and $Y$ the sequence $\Phi$ 
satisfies \ref{a0}, \ref{a1} and \ref{a3}. Let $\varepsilon_X$ and $\varepsilon_Y$ be the empty words of $X$ and $Y$ respectively. By convention, let
$\pi(\varepsilon_X)=\varepsilon_Y$.

\begin{lema}\label{imp}
Let $(X, \sigma_X)$ be a countable Markov shift, $(Y, \sigma_Y)$ a subshift on a countable alphabet, and $\pi:X\rightarrow Y$ a one-block factor map such that $\vert \pi^{-1}(i)\vert<\infty$ for each $i\in \N$. 
If $X$ is finitely irreducible, then $\Phi=\{\log \phi_n\}_{n=1}^{\infty}$ is a Bowen sequence on $Y$ satisfying \ref{a0},\ref{a1} and \ref{a3}. If $W_1$ is a finite set from Definition \ref{fi}, then let $\pi(W_1)=\{\pi(w): w \in W_1\}$.
For any $u\in B_n(Y), v\in B_m(Y), n, m\in\N$, there exists $w'\in \pi(W_1)$ such that 
$\vert \pi^{-1}(uw'v)\vert \geq  (1/{\vert W_1\vert})\vert \pi^{-1}(u)\vert \vert \pi^{-1}(v)\vert.$
\end{lema}
\begin{proof}
See  \cite[Lemma 5.7]{Fe4} in which the above result was studied for the case when $X$ is an irreducible subshift on finite alphabets. 
This implies the result.
\end{proof}


\begin{rem}
The case when $X$ is not finitely irreducible is studied in Example \ref{nofinite} in which $\Phi$ on $Y$ does not satisfy \ref{a3}.
We also remark that in general $\Phi$ is not almost-additive (see \cite{Y1, Y}).\end{rem}

\end{eje}

\subsection{Examples of sequences on irreducible countable sofic shifts }\label{examplespotentilas}
We provide a wide range of examples of sequences of functions satisfying (or not) different additivity properties. Some of these examples can only occur in non-compact settings and show some of the new phenomena that have to be considered in the countable alphabet situation. The examples  in this section come from a construction in the theory of factor maps and  will also appear in the following sections when we study the variational principle.   Let 
$\Phi$ be the sequence of functions as in Example \ref{exfactor}. 


%
%

\begin{eje}\textnormal{\textbf{A sequence on a finitely irreducible countable Markov shift satisfying \ref{a0}, \ref{a1} and \ref{a3}.}} \label{e1}
In this example, we construct a sequence of functions  which satisfies \ref{a0},\ref{a1} and \ref{a3}, but fails to be almost-additive. Let $(X, \sigma)$ be a countable Markov shift  determined by the transitions given by Figure \ref{fig1}. 

Let $\pi: \N\rightarrow \N$ be the function defined by $\pi(-i+n(n+1)/2)=n$, $i=0, \dots, n-1$ for $n \in \N$ and  $\Sigma$ be the full shift on a countable alphabet. Define $\pi: X\rightarrow \Sigma$ by 
$(\pi(x))_i=\pi(x_i)$ for all $i\in\N$ and 
denote $\pi(X)$ by $Y$. Then the map $\pi:X \rightarrow Y$ is a one-block factor map. 
Note that since $\vert \pi^{-1}(i)\vert=i$ for $i\in \N$ we have that $\vert \pi^{-1}(i)\vert$ is not uniformly  bounded. We stress that this property cannot occur  when $X$ is a finite state Markov shift.  $X$ has a strong specification number equal to $2$, just by considering  $W=\{12, 22\}$. Thus, the countable Markov shift $Y$ also has a strong specification number $2$.  

We first observe that $\Phi$ is not almost-additive on $Y$. Let $A$ be the transition matrix for $X$. It was shown in \cite [Example 5.6]{Y1} that $\Phi$ is not almost-additive on 
$\pi (X_{A\vert_{\{1,2,3\} \times \{1,2,3\}}})$.  
Let $k\geq 3$ be fixed and define 
\begin{equation*}
\psi_n(y)= \frac{\phi_n(y)}{(\vert \pi^{-1}(y_1)\vert \cdots \vert \pi^{-1}(y_n)\vert)^{k}},
\end{equation*}
and  $\Psi =\{\log \psi_n\}_{n=1}^{\infty}$. The sequence $\Psi$ is not almost-additive but it is sub-additive. 
By Lemma \ref{imp}, condition \ref{a1} holds with $p=2$. For $u\in B_n(Y)$ and $v\in B_m(Y)$, there exists a word $w \in \{\pi(12), \pi(22)\}$ of length $2$ such that 
$$\sup \{\psi_{n+m+2}(y): y\in [uwv]\}\geq \frac{1}{2^{2k+1}}\cdot \sup \{\psi_{n}(y): y\in [u]\}\sup \{\psi_{m}(y): y\in [v]\}.$$
Hence $\Psi$ is a Bowen sequence on $Y$ satisfying  \ref{a0}, \ref{a1} and \ref{a3}.

\begin{tikzpicture}[shorten >=1pt,node distance=1.5cm,on grid,auto]
   \node[state] (0) {$1$};
   \node[state] (1) [right=of 0] {$2$};
   \node[state] (2) [right=of 1] {$3$};
   \node[state] (3) [right=of 2] {$4$};
   \node[state] (4) [right=of 3] {$5$} ;
   \node  (5) [right=of 4] {$\cdots$} ;

   \path[->]
    (0) edge     (1)
    (1) edge                 (0)
    (1) edge   (2)
    (1)  edge [bend left]    (3)
    (1)  edge [bend left]    (4)
    (1) edge [loop below]   (1)
    (1) edge  (0)
    (1) edge [bend left]  (5)
    (2) edge [loop below]    (2)
    (2) edge       [bend right]          (0)
    (3) edge [bend right]    (1)
    (4) edge [bend right]   (1)
    (5) edge [bend right]   (1) ;
\end{tikzpicture}
\begin{figure}[h]
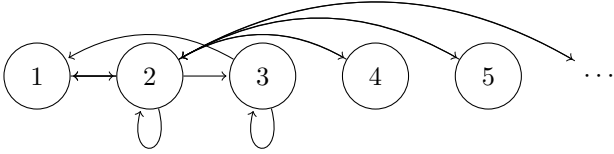

{} \caption{The graph defining $X$  in Example \ref{e1}}
\label{fig1}
\end{figure}
\end{eje}

\begin{eje}\textnormal{\textbf{A sequence on a finitely irreducible countable sofic shift satisfying \ref{a0}, \ref{a1} and \ref{a3}.}}\label{e11}
We study a general case of Example \ref{e1}. 
Let $(X, \sigma_X)$ be a finitely irreducible countable Markov shift, $(Y, \sigma_Y)$ a subshift, and $\pi: X\rightarrow Y$ be a one-block factor map.
Thus, $(Y, \sigma_Y)$ is a finitely irreducible countable sofic shift. 
Suppose there exist $C_1, C_2>0, k\geq 1 $ such that  for every $i \in \N$ we have
\begin{equation*} 
C_1i^k\leq \vert \pi^{-1}(i)\vert \leq C_2i^k.
\end{equation*}
For $y\in [y_1\dots y_n]$,  define
\begin{equation*}
\psi_n(y):= \frac{\phi_n(y)}{(\vert \pi^{-1}(y_1)\vert \cdots \vert \pi^{-1}(y_n)\vert)^{k+2}}
\end{equation*}
and $\Psi =\{\log \psi_n\}_{n=1}^{\infty}$. By Lemmas  \ref{variation} and \ref{imp}, the sequence  $\Psi$ is a sub-additive Bowen sequence on $Y$ satisfying   \ref{a0}, \ref{a1} and \ref{a3}. 
\end{eje}

\begin{eje}  \textnormal{\textbf {A sequence on a finitely irreducible countable sofic shift satisfying \ref{a0}, \ref{a1} and \ref{a3}.}}\label{exgeneral}
Let $(X, \sigma_X)$ be a finitely irreducible countable Markov shift, $(Y, \sigma_Y)$ a subshift, and $\pi: X\rightarrow Y$ be a one-block factor map such that $\vert  \pi^{-1}(i) \vert <\infty$ for any $i\in\N$. 
Thus, $(Y, \sigma_Y)$ is a finitely irreducible countable sofic shift. 
Let $K>0$. For $y\in [y_1\dots y_n]$, we define
$$\psi_n(y)= \frac{\phi_n(y)}{K^{y_1+\dots +y_n}}$$
and let $\Psi =\{\log \psi_n\}_{n=1}^{\infty}$. Then $\Psi$ is a sub-additive Bowen sequence on $Y$ satisfying   \ref{a0}, \ref{a1} and \ref{a3}. 
\end{eje}

\begin{eje}  \textnormal{\textbf {A sequence on a finitely irreducible countable sofic shift satisfying \ref{a0}, \ref{a1} and \ref{a3}.}}\label{exhiddeng}
Let $\G$ be defined as in Theorem \ref{hgibbs} in Section \ref{hiddeng}. Then $\G$ is a Bowen sequence defined on a finitely irreducible countable sofic shift satisfying \ref{a0},  \ref{a1} and \ref{a3}.
\end{eje}

\begin{eje} \textnormal{\textbf {A sequence satisfying \ref{a0} and \ref{a1} but not  \ref{a3}.}}\label{nofinite}
 Here we present an example of a sequence which satisfies \ref{a0} and \ref{a1}, but fails to be almost-additive and for which the finiteness condition  \ref{a3} does not hold.
We consider  a factor map defined on 
a countable Markov shift which is not finitely irreducible.  Let $(X, \sigma)$ be the countable Markov shift determined by the transitions given by Figure \ref{fig2}. We partition the alphabet defining $X$ in the following way:
$F_1=\{1\}, F_2=\{2, 3\}, F_3=\{4, 5, 6\},\dots$, in general $F_n$ consists of $n$ symbols, such that the subshift of $X$ restricted to the  symbols of $F_n$ is the full shift on $n$ symbols.
Let $\pi: \N\rightarrow \N$ be the function defined by $\pi(a)=n$ if $a\in F_{n}, n\in \N$ and let $\Sigma$ be the full shift on a countable alphabet. Define
$\pi: X\rightarrow \Sigma$  by $(\pi(x))_i=\pi(x_i)$ for all $i\in\N$. Let $Y=\pi(X)$. Then $Y$ is a countable Markov shift and
$\pi: X\rightarrow Y$ is a one-block factor map.
Note that $X$ is not finitely irreducible and that  $3$ is a specification number for $X$. On the other hand,  $Y$ is finitely primitive with a specification number $1$. Noting that  $\vert \pi^{-1}(i)\vert=i$ and $\vert \pi^{-1}(i1)\vert=1$, $\vert \pi^{-1}(i1)\vert/(\vert \pi^{-1}(i)\vert \vert \pi^{-1}(1)\vert)$ 
is not bounded below by a constant. Therefore the sequence $\Phi= \{\log \phi_n \}_{n=1}^{\infty}$ is not almost-additive, however it is  sub-additive by construction. 
Let $u=u_1\dots u_n\in B_n(Y)$ and  $v=v_1\dots v_m\in B_m(Y).$ We claim that 
\begin{equation}\label{ineqsp}
\vert \pi^{-1}(uu_n1v_1v)\vert \geq \vert \pi^{-1}(u)\vert \vert \pi^{-1}(v)\vert
\end{equation}
and hence \ref{a1} is satisfied. To see this,  consider a preimage $\bar u={\bar u}_1\dots {\bar u}_n$ of $u$ and $\bar v= {\bar v}_1\dots {\bar v}_m$ of  $v$.  Then ${\bar u}_n\in F_s$ and  ${\bar v}_1\in F_t$ for some $s, t\in\N$. Assume $s\neq 1$ and $t\neq 1$. Define $a_s\in F_s$ and $a_t\in F_t$ such that 
 $1a_s1$ and $1a_t1$ are allowable words. Then  $\bar u a_s1a_t \bar v$   is an allowable word of $X$ and $\pi(\bar u a_s
 1a_t \bar v)=uu_n1v_1v$. Similar arguments when $s=1$ or $t=1$ yield the same result. The claim now follows, indeed  
 $\Phi$ is a sub-additive Bowen sequence on $Y$ satisfying \ref{a1} with the strong specification. However, \ref{a3} is not satisfied. If we let  $W$ be the set consisting of all possible $u_n1v_1$ in (\ref{ineqsp}), then $W =\left\{ i1j: i, j \in \N \right\}$.  
 
 We observe that for any $p\in \N$ \ref{a3} is not satisfied.
 Suppose $\F$ satisfies \ref{a1} and \ref{a3} and let  $W$ be a finite set as in \ref{a3}. Clearly $W\neq \{\varepsilon\}$. 
 Observe that such a finite set $W$ consists of allowable words $w$ of the following four types. 
 If $w=w_1 \dots w_k$, for $1\leq k\leq p$, then $w_1=1$ and $w_k\neq 1$ (which we call Type 1), $w_1\neq 1$ and $w_k= 1$(Type 2), 
  $w_1=1$ and $w_k= 1$ (Type 3), or  $w_1\neq 1$ and $w_k\neq 1$ (Type 4).   
  Let $w$ be an allowable word of Type 1. Then for any allowable  words $u, v$ in $Y$ such that $uwv$ is allowable, we obtain
 $\vert \pi^{-1}(uwv)\vert\leq \vert \pi^{-1}(u1)\vert\vert \pi^{-1}(w)\vert \vert \pi^{-1}(v)\vert.$
 Let $i\in\N$. If we take $u=i$ then  $\vert \pi^{-1}(i)\vert=i$ and 
 $\vert \pi^{-1}(i1)\vert=1$. Therefore,  \ref{a1}  implies that $\vert \pi^{-1}(w)\vert /i\geq D$.  
  Hence
 there exist $N_1 \in \N$  such that if $i\geq N_1$ then  for any pair $i, v$, \ref{a1} does not holds with $iwv$ where $w$ is of Type 1. By making similar arguments for $w$ of Type 2, 3 and 4, there exists a pair $i, j\in \N$ such that \ref{a1} does hold by using a $w$ from a finite set $W$.

. \begin{tikzpicture}[shorten >=1pt,node distance=1.5cm,on grid,auto]
   \node[state] (0) {$1$};
   \node[state] (1) [right=of 0] {$2$};
   \node[state] (2) [right=of 1] {$3$};
   \node[state] (3) [right=of 2] {$4$};
   \node[state] (4) [right=of 3] {$5$} ;
  \node[state] (5) [right=of 4] {$6$} ;
 \node (6) [right=of 5] {$\cdots$} ;

   \path[->]
    (0) edge     (1) 
    (0) edge      [loop below]  
    (0) edge  [bend left]   (3) 
    (0) edge  [bend left]   (6)
    (1) edge                 (0)
    (1) edge   (2)
    (1) edge [loop below]    (1)
    (2) edge [loop below]    (2)
    (2)  edge    (1)
    (3)  edge    [bend right]     (0)
     (3) edge [loop below]    (3)
     (3)  edge                 (4)
     (3) edge [bend left]    (5)
     (4) edge    (3)
      (4) edge    (5)
     (4) edge [loop below]    (4)
    (5) edge  (4)
    (5) edge [bend right]   (3) 
    (5) edge [loop below]  (5)
    (6) edge       [bend right]          (0)
     
 ;
\end{tikzpicture}
\begin{figure}[h]
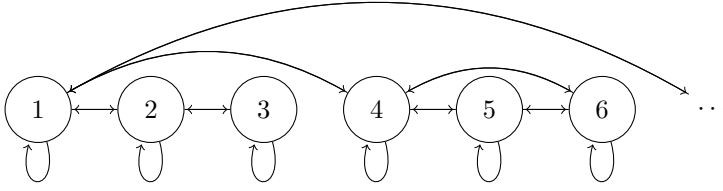

{} \caption{The graph defining $X$ in Example \ref{nofinite}}
\label{fig2}
\end{figure}
\end{eje}

\section{Variational Principle}\label{vpvp}
A fundamental result in thermodynamic formalism is the variational principle. It establishes a relation between the pressure (which is defined by means of the  topological structure of the system) and the sum of the metric entropy and the integral with respect to an invariant measure (which is defined by means of the Borel structure of the system). The relation in the variational principle for a sequence of functions $\F=\{\log f_n\}_{n=1}^{\infty}$ is the following
\begin{equation*}
P(\F)=\sup_{\mu\in M(X, \sigma)} \left\{h_{\mu}(\sigma)+\lim_{n\rightarrow  \infty}\frac{1}{n}\int \log f_n d\mu: \lim_{n\rightarrow  \infty}\frac{1}{n}\int \log f_n d\mu>-\infty \right\},
\end{equation*}
where $M(X, \sigma)$ denotes the space of $\sigma-$invariant Borel probability measures.  The goal of this section is to establish the variational principle for the types of sequences introduced in Section \ref{potentials}. 
\subsection{Variational principle for a countable Markov shift with the the weak specification property without the finiteness condition \ref{a3}}  \label{vp:spc}

The purpose of this section is to prove the variational principle for the Bowen sequences defined on 
countable Markov shifts  satisfying \ref{a0} and \ref{a1}. We do not assume the finiteness condition \ref{a3}.
Hence in the proof of the approximation property (Proposition \ref{approx1}), this  condition is not assumed.  However, we see in Lemma \ref{cha1} that if the pressure is finite, then the type of sequence we consider in this section is defined on a space with the finiteness condition. 

The following is an important technical remark (see \cite{mu2}). Since $X$ is an irreducible countable Markov shift,  by rearranging the set $\N$ of the symbols of $X$, there exists a transition matrix $A$ for $X$ and an increasing sequence $\{k_{n}\}_{n=1}^{\infty}$ such that the matrix $A\vert_{ \{1, \dots, k_{n}\}\times \{1, \dots,k_{n}\}}$ is irreducible.
Define $A_{k_n}:=A\vert_{ \{1, \dots, k_{n}\}\times \{1, \dots, k_{n}\}}$. We will assume the following property on the sequence of functions $\F=\{\log f_n\}_{n=1}^{\infty}$, where $f_{n}: X\to \R^{+}$ are continuous functions.
\smallbreak
\smallbreak
\begin{enumerate}[label=(P\arabic*)]  
\item \label{P1}   There exist  an increasing sequence $\{l_{n}\}_{n=1}^{\infty}$ and constants $D_1, p_1>0$ such that for each $l_n$ the matrix $A_{l_n}$ is irreducible and $\F\vert_{X_{ A_{l_n}}  }$ satisfies \ref{a1} with constants $D_{l_n}$ and  $p_{l_n}\in \N$ such that  $D_{l_n}\geq D_1$,  and  $p_{l_n}\leq p_1$ for every $n \in \N$.
\end{enumerate}

\begin{lema}\label{small}
If $\F=\{\log f_n\}_{n=1}^{\infty}$ is a Bowen sequence on an irreducible countable Markov shift $X$ satisfying \ref{P1}, then $\F$ satisfies \ref{a1} and  $X$ satisfies the weak specification property. 
\end{lema}
\begin{proof}
Let $u\in B_n(X)$ and $v\in B_m(X)$ for $n, m\in\N$. Then there exists $N\in\N$ such that $u,v$ are allowable words of  $X_{A_{l_N}}$. Call  $Y:=X_{A_{l_N}}$. Then
the Bowen property and  \ref{P1} imply that there exists  $w\in B_{k}(Y)$, $0\leq k\leq p_{l_N}\leq p_1$   such that 
\begin{equation*}
\begin{split}
&\sup\{f_{n+m+k}(x):x\in [uwv]\}\geq \sup\{f_{n+m+k}\vert_{Y} (x):x\in [uwv]\}\\
&\geq D_{l_N}\sup\{f_{n}\vert_{ Y} (x):x\in [u]\}\sup\{f_{m}\vert_{Y }(x):x\in [v]\}
\geq \frac{D_{1}}{M^2}\sup\{f_{n} (x):x\in [u]\}\sup\{f_{m}(x):x\in [v]\}.
\end{split}
\end{equation*}
\end{proof}

In particular, a Bowen sequence on a finitely irreducible Markov shift satisfying \ref{a1} and \ref{a3}  is a sequence satisfying \ref{P1} (see Corollary \ref {cor:5}).
In the following propositions and lemmas, we continue to use the notation from \ref{P1} and Section \ref{potentials}.  Let $a\in \N$ be a symbol of a countable alphabet. 
For a compact $\sigma$-invariant subset $Y$ of $X$, define 
$Z_n(\mathcal{F}\vert_{Y}, a)=\sum_{y:\sigma^n(y)=y,y_1=a}f_n\vert_Y(y)$.

We first  show that with the assumption \ref{P1} the topological pressure $P(\F)$ and the Gurevich pressure $P_G(\F)$ takes $\infty $ when $Z_1(\F)=\infty$.
\begin{lema}\label{gdefinf}
Let $(X, \sigma)$ be an irreducible  countable Markov shift and
$\F=\{\log f_n\}_{n=1}^{\infty}$ a sequence on $X$ with tempered variation satisfying \ref{a0} and \ref{P1}.  If $Z_1(\F)=\infty$,  then $P(\F)=\infty$.  Thus, equation (\ref{firstdef}) holds.
\end{lema}
 \begin{proof}
Let $f'_n(x)=e^Cf_n(x)$ for all $x\in X$ and $\mathcal{\F}'=\{\log f'_n\}_{n=1}^{\infty}$. Then the sequence  $\mathcal{F}'$ is sub-additive and $P(\F)=P(\F')$.
Note by Proposition \ref{gdefi} that we obtain $P_{G}(\F'\vert_{X_{l_n}}, a)=P(\F'\vert_{X_{l_n}})$ for each $l_n$. 
Since $g$ has tempered variation, if $Z_1(\F)=\infty$, then given $L>0$, there exists $N\in \N$ such that $Z_1(\F\vert_{X_{l_N}})>L$ and thus  $Z_1(\F'\vert_{X_{l_N}})>Le^C$.
Let $Y:=X_{l_{N}}$. Then \ref{P1} implies that for each $n\in \N$ there exists  $0\leq i_n\leq p_{1}(n-1)$ such that
 \begin{equation}\label{approxbelow}
Z_{n+i_n}(\F'\vert_Y)\geq (\frac{D'_1}{p_1+1})^{n-1}(Z_1(\F'\vert_Y))^n,
 \end{equation}
 where $D'_1=D_1/e^C$.
 Since we have  $P(\F')\geq  \limsup_{n\rightarrow \infty}(1/(n+i_n))\log Z_{n+i_n}(\F'\vert_Y)=P(\F'\vert_Y)$, we obtain  that 
 $P(\F')\geq P(\F'\vert_Y) \geq d +(1/(p_1+1))\log (Le^C)$ where $d=(1/(p_1+1))\log (D'_1/(p_1+1))$.  
 Letting $L\rightarrow\infty$, we obtain $P(\F)=P(\F')=\infty$.
 To see that (\ref{firstdef}) holds,  we apply  Proposition \ref{gdefi}. Since $P_{G}(\F\vert_Y, a)=P_{G}(\F'\vert_Y, a)=P(\F'\vert_{Y})$ and $P_{G}(\F, a)\geq P_{G}(\F\vert_{Y}, a)$, the result follows by letting $L\rightarrow\infty$. 
 \end{proof}

The next result provides an approximation property by compact invariant sets for the pressure, without the finiteness condition \ref{a3}. We prove this by using the Gurevich pressure.

\begin{prop}\label{approx1}
Let $(X,\sigma)$ be an irreducible 
countable Markov shift.
If $\F=\{\log f_n\}_{n=1}^{\infty}$ is a Bowen sequence on $X$ satisfying \ref{a0} and  \ref{P1}, then
\begin{equation*}
P(\F)=\sup_{l_n, n\in \N} \left\{P(\F\vert_{X_{A_{l_n}}}) \right\},
\end{equation*}
and $P(\F)\neq -\infty.$
\end{prop}
\begin{proof}
We use similar arguments used in  \cite[Proposition 3.1]{iy1}. 
First assume $Z_1(\F)<\infty$ and so $P(\F)<\infty$. Assume also that $-\infty<P(\F)$(we show that $P(\F)\neq -\infty$).
Define $f'_n(x)$ and $\mathcal{F}'$ as in the proof of Lemma \ref{gdefinf}.
Then the sequence  $\mathcal{F}'$ is a sub-additive Bowen sequence and $P(\F)=P(\mathcal{F}')$.  
Let $a\in \N$ be a symbol of the countable  alphabet of $X$.
Then $P(\mathcal{F}')=\lim_{n\rightarrow \infty}B_n$, where $B_n=\sup_{k\geq n}(1/k)\log Z_k( \mathcal{F}', a)$ and 
$B_n<\infty$ for all $n\in \N$. 
Let $\epsilon>0$ and fix $m\in \N$ such that 
\begin{equation}\label{cond}
\frac{p_1}{m}<\epsilon \quad \textnormal{ and }  \quad \frac{1}{m} \left\vert \log \frac{M(p_1+1)}{D_1}  \right\vert< \epsilon.
\end{equation}
Then there exists $q\in\N, q \geq m$ such that
\begin{equation*}
B_m-\epsilon<\frac{1}{q}\log Z_{q}( \mathcal{F}', a)\leq B_m.
\end{equation*}
Thus, 
\begin{equation}\label{eq0}
P(\mathcal{F}')\leq B_m <\frac{1}{q}\log Z_{q}( \mathcal{F}', a)+\epsilon.
\end{equation}
Since $q \geq m$, (\ref{cond}) holds by replacing  $m$ by $q$. 
Then $Z_{q}(\mathcal{F}', a)=\lim_{n\rightarrow \infty}Z_q(\mathcal{F}'\vert_{X_{A_{l_n}}}, a).$
Note that $\{Z_q(\mathcal{F}'\vert_{X_{A_{l_n}}}, a)\}_{n=1}^{\infty}$ increases to $Z_{q}(\mathcal{F}', a)$ monotonically.
Thus there exists $n_1 \in \N$ such that 
\begin{equation}\label{eq1}
\frac{1}{q}\log Z_{q}(\mathcal{F}', a)<\frac{1}{q}  \log Z_{q}(\mathcal{F}'\vert_{X_{A_{l_{n_1}}}}, a)+\epsilon.
\end{equation}
By \ref{P1}, $ \mathcal{F}'\vert_{X_{A_{l_{n_1}}}}$ satisfies \ref{a1}.
Let $Y=X_{A_{l_{n_1}}}$. Then  $Y$ has the weak  specification property with a specification number $p_Y\leq p_1$.

\begin{lema}\label{capro}
For $n, m \in \N$, there exists $0\leq i_{n,m}\leq p_Y$ such that 
\begin{equation}
\frac{(p_Y+1)M}{D_1}Z_{i_{n,m}+n+m}(\mathcal{F}'\vert_ Y, a)\geq Z_n(\mathcal{F}'\vert_Y, a)Z_m(\mathcal{F}'\vert_Y, a).
\end{equation}
\end{lema}

\begin{proof}
Let $n, m\in\N$ be fixed. Take $x, y \in Y$ such that $\sigma^{n}x=x$ and  $\sigma^{m}y=y$, where $x_1=y_1=a$.  Let
$x=(ax_2\dots x_{n-1})^{\infty}$ and $y=(ay_2\dots y_{m-1})^{\infty}$. By \ref{P1}, there exist $0\leq k\leq p_Y$ and an allowable word $b_1\dots b_k$ in $Y$  such that $ax_2\dots x_{n-1}b_1\dots b_kay_2\dots y_{m-1}$ is allowable in $Y$  satisfying \ref{a1}. Thus
$z=(ax_2\dots x_{n-1}b_1\dots b_kay_2\dots y_{m-1})^{\infty}\in Y$ and $\sigma^{n+m+k}z=z.$

\begin{equation*}
\begin{split}
Mf'_{n+m+k}\vert_Y(z) &\geq 
\sup\{f'_{n+m+k}\vert_Y(x) :x\in [ax_2\dots x_{n-1}b_1\dots b_kay_2\dots y_{m-1}]\}\\
&\geq 
D_1\sup\{f'_{n}\vert_Y(x) :x\in [ax_2\dots x_{n-1}]\}  \sup\{f'_{m}\vert_Y(x) :x\in [ay_2\dots y_{m-1}]\} \\
&\geq D_1f'_n\vert_Y(x)f'_m\vert_Y(y).
\end{split}
\end{equation*}
Therefore
\begin{equation*}
M\sum_{k=0}^{p_Y}Z_{n+m+k}(\mathcal{F}'\vert_Y,a) \geq D_1Z_{n}(\mathcal{F}'\vert_Y,a)Z_{m}(\mathcal{F}'\vert_Y,a). 
\end{equation*}
There exists $0\leq i_{n,m}\leq p_Y$ such that 
\begin{equation*}
\frac{(p_Y+1)M}{D_1}Z_{n+m+i_{n,m}}(\mathcal{F}'\vert_Y,a) \geq Z_{n}(\mathcal{F}'\vert_Y,a)Z_{m}(\mathcal{F}'\vert_Y,a). 
\end{equation*}
\end{proof}

Setting $m=n=q$ in Lemma \ref{capro}, there exists $0\leq  i_{1} \leq p_Y$ such that
$ ((p_Y+1)M/D_1)Z_{2q+i_{1}}(\mathcal{F}'\vert_Y,a) \geq (Z_{q}(\mathcal{F}'\vert_Y,a))^2.$
Applying the lemma $(k-1)$ times, there exist $0\leq i_{1},\dots, i_{k-1} \leq p_Y$ such that
\begin{equation} \label{aviodinf}
\left(\frac{(p_Y+1)M}{D_1} \right)^{k-1}Z_{kq+i_{1}+\cdots+i_{k-1}}(\mathcal{F}'\vert_Y,a) \geq (Z_{q}(\mathcal{F}'\vert_ Y,a))^k.
\end{equation}
Now let $a_q=\log Z_q( \mathcal{F}'\vert_Y,a)$.
Then 
\begin{equation}
\begin{split}
\frac{a_q}{q}=&\frac{\log (Z_q( \mathcal{F}'\vert_Y,a))^k}{kq}
\leq \frac{(k-1)\log  \left(\frac{(p_Y+1)M}{D_1} \right)+\log Z_{kq+i_{1}+\cdots+i_{k-1}}(\mathcal{F}'\vert_Y,a)}{kq}. \label{key-}
\end{split}
\end{equation}
Since $p_Y\leq p_1$, letting $k\rightarrow\infty$
\begin{equation*}
\begin{split}
\frac{a_q}{q}&\leq \frac{\log \left(\frac{(p_1+1)M}{D_1}\right)}{q}+\left(1+\frac{p_1}{q}\right)\limsup_{k\rightarrow \infty}\frac{1}{kq+i_{1}+\cdots+i_{k-1}}{\log Z_{kq+i_{1}+\cdots+i_{k-1}}(\mathcal{F}'\vert_Y,a)}\\
&\leq \epsilon +(1+\epsilon)P(\mathcal{F}'\vert_Y) \leq \epsilon(P(\mathcal{F}')+1)+P(\mathcal{F}'\vert_Y).
\end{split}
\end{equation*}
Hence, using (\ref{eq0}) and (\ref{eq1}),
we obtain
\begin{equation*}
P(\mathcal{F}')\leq 2\epsilon +\epsilon (P(\mathcal{F}')+1)+P(\mathcal{F}'\vert_Y).
\end{equation*}
Next assume that $Z_1(\F)=\infty$. Then  the proof of Lemma \ref{gdefinf} shows the result.
To show $P(\mathcal{F})\neq-\infty$, observe that (\ref{aviodinf}) and (\ref{key-})  are valid for any $m, l_n\in \N$ if we replace $q$ and $Y$ by $m$ and $X_{l_n}$ respectively. Hence  letting $k\rightarrow \infty$
in (\ref{key-}) implies that $P(\mathcal{F}')\neq-\infty$.
\end{proof}

\begin{prop}\label{cha}
Let $(X,\sigma)$ be a subshift on a countable alphabet.
If $\F$ a sequence on $X$ with tempered variation satisfying \ref{a0} and \ref{a4}, then $P(\F)<\infty$ if and only if $Z_1(\F)<\infty$. 
\end{prop}

\begin{proof}
We claim that $P(\F)<\infty$ if and only if $Z_n(\F)<\infty$ for all $n\in \N$. This gives the result by noting that $Z_1(\F)<\infty$ if and only if $Z_n(\F)<\infty$ for all $n\in \N$.
It is obvious that if $Z_n(\F)<\infty$ for all $n\in \N$, then 
$P(\F)<\infty$.
If $P(\F)<\infty$, then there exists $N\in \N$ such that 
$Z_n(\F)< \infty$ for all $n\geq N$. Let $u_{N}\in B_{N}(X)$ and $v_1\in B_{1}(X)$. Then by \ref{a4} there exist $p\in \N$ and $w \in B_{k}(X), 0\leq k\leq p$, such that 
$u_{N}wv_1$ is allowable and 
\begin{equation*}
\sup\{f_{N+k+1}(x) : x\in [u_{N}wv_1]\}\geq 
D_{N, 1}\sup\{f_{N}(x) :x\in [u_N]\}\sup\{f_{1}(x) :x\in [v_1]\}.
\end{equation*}
Hence 
\begin{equation*}
\sum_{i=0}^{p}Z_{N+i+1}(\F) \geq D_{N,1}Z_{N}(\F) Z_1(\F).
\end{equation*}
Since $Z_{N}(\F) $ is bounded below by a constant, 
we obtain that  $Z_1(\F) <\infty$ and hence $Z_n(\F)<\infty$ for all $n\in \N$. 
\end{proof} 
\begin{rem}
See \cite[Proposition 1.6]{mu2}  for a result related to Proposition \ref{cha}. 
\end{rem}

\begin{lema}\label{cha1}
Let $\F$ be a Bowen sequence on a subshift $X$ on a countable alphabet  satisfying \ref{a0} and \ref{a1}. If $\F$  fails to have \ref{a3}, then $P(\F)=\infty$.
\end{lema}
\begin{proof}

Assume $P(\F)<\infty$. By Claim  \ref{ourp}, $\sum_{i=1}^{p} Z_i(\F)=\infty$. This is a contradiction to Proposition \ref{cha}. 
\end{proof}
\begin{rem}
Note that by Lemma \ref{cha1} if $P(\F)\neq \infty$ then a Bowen sequence $\F$ satisfying \ref{a0} and \ref{a1} is defined on a finitely irreducible subshift. This motivates us to study a Gibbs measure for $\F$ (see \cite{mu2, s3}). 
\end{rem}

The main goal of this section is to obtain the variational principle and the results of the rest of this section will also be applied in Section \ref{sofic}.

\begin{prop}\label{upper}
Let $(X,\sigma)$ be a subshift on a countable alphabet and 
$\mathcal{F}=\{\log f_n\}_{n=1}^{\infty}$ be a sequence of continuous functions on $X$ with tempered variation  satisfying \ref{a0} and \ref{a4}. If $P(\F)<\infty$, then for any  $\mu \in  M(X, \sigma)$ such that $\lim_{n\rightarrow  \infty}
(1/n)\int \log f_n d\mu>-\infty$ we have
\begin{equation}\label{finitecase}
h_{\mu}(\sigma)+\lim_{n\rightarrow  \infty}\frac{1}{n}\int \log f_n d\mu\leq P(\F).
\end{equation}
\end{prop}
\begin{rem}
Assumptions of Proposition \ref{upper} imply that  $\lim_{n\rightarrow  \infty}
(1/n)\int \log f_n d\mu$ exists, and possibly $-\infty$ (see the proof below). Note that \ref{a4} implies  that $P(\F)\neq -\infty$. 
\end{rem}


\begin{proof}[Proof of Proposition \ref{upper}]
We follow the proof of \cite[Theorem 1.4]{mu2}. We have to slightly modify  the proof in order to take into account of the sub-additive sequence $\F':=\{\log (e^Cf_n) \}_{n=1}^{\infty}$. 
Since $P(\F)<\infty$,  Proposition \ref{cha}  implies  $Z_1(\F)<\infty$ and thus $\sup f_1<\infty$. Hence we obtain that $\int (\log e^Cf_1)^{+}d\mu<\infty$.  Applying sub-additive ergodic theorem to $\F'$, we obtain that 
$\lim_{n\rightarrow  \infty}(1/n)\int \log f_n d\mu$ exists for any $\mu\in M(X,\sigma)$.
Note by Proposition \ref{cha} that   $0< Z_n(\F) <\infty$ for each $n\in \N$. 
Using the sub-additivity of $\mathcal{F}'$, it follows  that for every $n,m\in\N$
\begin{equation*}
\frac{1}{nm} \int \log f_{nm} d\mu \leq \frac{1}{n} \int \log f_{n} d\mu+\frac{C}{n}.
\end{equation*}
Thus
\begin{equation*}
-\infty <\limsup_{m\rightarrow\infty}\frac{1}{nm} \int \log f_{nm} d\mu \leq \frac{1}{n} \int \log f_{n} d\mu +\frac{C}{n},
\end{equation*}
and for each $n\in\N$
\begin{equation*}
\sum_{w\in B_n(X)}\mu([w])\log \left(\sup\{f_n(x):x\in [w]\} \right)\geq \int \log f_n d\mu >-\infty.
\end{equation*}

For $n\geq 1$, letting $h(x)=-x\log x$, we have
\begin{equation*}
\begin{split}
&-\sum_{w_n\in B_n(X)}\mu([w])\log \mu([w])+\int \log f_n d\mu\\
&\leq \sum_{w\in B_n(X)}\mu([w]) \left(\log (\sup \{f_n(x):x\in [w]\})-\log \mu[w] \right)\\
&= Z_n(\F)\sum_{w\in B_n(X)}\frac {\sup \{f_n(x):x\in [w]\}}{Z_n(\F)}  h \left (\frac{\mu([w])} {\sup\{ f_n(x):x\in [w]\}}\right)\\
&\leq Z_n(\F)h\left(\sum_{w\in B_n(X)}\frac{\mu([w])}{Z_n(\F)} \right)\leq Z_n(\F)h\left(Z_n(\F)^{-1}\right)=\log Z_n(\F), 
\end{split}
\end{equation*}
where in the third inequality we use the concavity of $h$.
Therefore, for every $n\geq 1$ we have that $-\sum_{w\in B_n(X)}\mu([w])\log \mu([w])<\infty$. In particular, if we let $\alpha=\{[u]: u \in B_1(X)\}$, then $\alpha$ is a generator  for $\sigma$.
Hence
\begin{equation*}
h_{\mu}(\sigma)+\lim_{n\rightarrow \infty} \frac{1}{n}\int \log f_n d\mu \leq \lim_{n\rightarrow \infty} \left(-\frac{1}{n} \sum_{w\in B_n(X)}\mu([w])\log \mu([w]) +\frac{1}{n} \int \log (e^Cf_n) d\mu \right)\leq P(\F).
\end{equation*}
\end{proof}
\begin{lema}\label{upper1}
Let $(X,\sigma)$ and $\mathcal{F}=\{\log f_n\}_{n=1}^{\infty}$ be defined as in Proposition \ref{upper}.
If $P(\F)=\infty$, then for any  $\mu \in  M(X, \sigma)$ such that  $\limsup_{n\rightarrow  \infty}
(1/n)\int \log f_n d\mu>-\infty$,
\begin{equation}\label{inftycase}
h_{\mu}(\sigma)+\limsup_{n\rightarrow  \infty}\frac{1}{n}\int \log f_n d\mu\leq P(\F).
\end{equation}
If $\sup f_1< \infty$, then $\limsup$ can be replaced by $\lim$.
\end{lema}
\begin{proof}
The result  is obvious.
\end {proof}

To show the variational principle, we need the following variational principle for sequences on subshifts on finite alphabets (see \cite{cao}).
\begin{teo}\label{vpcompact} \cite{cao} 
Let $(X,\sigma)$ be a subshift  on a finite alphabet. If
$\F=\{\log f_n\}_{n=1}^{\infty}$ is a sequence on $X$ with tempered variation satisfying \ref{a0}, then 
\begin{equation*}
P(\F)=\sup_{\mu\in M(X, \sigma)} \left\{h_{\mu}(\sigma)+\lim_{n\rightarrow  \infty}\frac{1}{n}\int \log f_n d\mu \right\},
\end{equation*}
where $P(\F)$ is defined in Definition \ref{defpressure}. Then $P(\F)=-\infty$ if and only if $\lim_{n\rightarrow  \infty}(1/n)\int \log f_n d\mu=-\infty$ for all 
$\mu\in M(X, \sigma)$.
\end{teo}
In Theorem \ref{vpcompact} an equilibrium measure for $\F$ (see Definition \ref{dgibbs}) always exists.
\begin{teo}\label{vp1}
Let $(X,\sigma)$ be  an irreducible
countable Markov shift and 
$\F=\{\log f_n\}_{n=1}^{\infty}$ be a Bowen sequence on $X$ satisfying \ref{a0} and \ref{P1}.
If $P(\F)<\infty$, then
\begin{equation}\label{firstvp}
P_{G}(\F)=P(\F)=\sup_{\mu\in M(X, \sigma)} \left\{h_{\mu}(\sigma)+\lim_{n\rightarrow  \infty}\frac{1}{n}\int \log f_n d\mu: \lim_{n\rightarrow  \infty}\frac{1}{n}\int \log f_n d\mu>-\infty \right\}.
\end{equation}
In particular, if
$\F$  satisfies  \ref{a1}  with the strong specification, then $\limsup$ in  (\ref{gurev}) of the definition  $P_{G}(\F)$ can be replaced by lim. 
If $P(\F)=\infty$, then 
\begin{equation}\label{vpinf}
P_{G}(\F)=P(\F)=\sup_{\mu\in M(X, \sigma)} \left\{h_{\mu}(\sigma)+\limsup_{n \rightarrow  \infty}\frac{1}{n}\int \log f_n d\mu: \limsup_{n\rightarrow  \infty}\frac{1}{n}\int \log f_n d\mu>-\infty \right\}.
\end{equation}
In particular, if  $\sup f_1<\infty$,  equation (\ref{firstvp}) holds for the case when $P(\F)=\infty$.
\end{teo}

\begin{proof}
First assume that $P(\F)< \infty$. Let $\epsilon>0$. Applying Proposition \ref{approx1}, there exists a finite state Markov shift $Y$ such that $P(\F)-P(\F\vert_Y)<\epsilon$.
Let $m$ be an equilibrium measure for $\F\vert_Y$. 
Since $m\in M(X, \sigma)$ and $\lim_{n\rightarrow  \infty}(1/n)\int \log f_n dm >-\infty$,  we obtain
\begin{align*}
&h_{m}(\sigma_{Y})+\lim_{n\rightarrow  \infty}\frac{1}{n}\int \log f_n dm \\
&\leq 
\sup_{\mu\in M(X, \sigma)} \left\{h_{\mu}(\sigma)+\lim_{n\rightarrow  \infty}\frac{1}{n}\int \log f_n d\mu: \lim_{n\rightarrow  \infty}\frac{1}{n}\int \log f_n d\mu>-\infty \right\}.
\end{align*}
Thus
\begin{equation*}
\begin{split}
P(\F)-\epsilon \leq P(\F\vert_{Y})
&\leq 
\sup_{\mu\in M(X, \sigma)} \left\{h_{\mu}(\sigma)+\lim_{n\rightarrow  \infty}\frac{1}{n}\int \log f_n d\mu: \lim_{n\rightarrow  \infty} \frac{1}{n}\int \log f_n d\mu>-\infty \right\} 
\leq P(\F).
\end{split}
\end{equation*}
Hence  we obtain the result.
Equation (\ref{vpinf}) holds for $P(\F)=\infty$ by similar arguments using Lemma \ref{upper1}. The last statement is obvious.
\end{proof}
\begin{coro} \label{cor:5}
Let $(X,\sigma)$ be a 
countable Markov shift. If $\F$ a Bowen sequence on $X$ satisfying \ref{a0}, \ref{a1} and \ref{a3}, then Propositions \ref{approx1} and \ref{upper}, Lemma \ref{upper1}
and Theorem \ref{vp1} hold. 
\end{coro}

\begin{proof}
It suffices to show that  \ref{P1} is satisfied.  Let $W$ be a finite set from \ref{a3}. 
Since $X$ is an irreducible countable Markov shift,
let $A_{l_n}$ be defined as in the beginning of this section.
Take $l_q$ large enough 
so that $\{1, \dots, l_{q}\}$ contains all the symbols that appear in $W-\{\varepsilon\}$.
Then, for  $n\geq q$, $\F\vert_{X_{A_{l_n}}}$ satisfies \ref {a1} replacing $D$ by $D/M$. 
\end{proof}

\begin{rem}
$(X,\sigma)$ in Corollary \ref{cor:5} is finitely irreducible by \ref{a1} and  \ref{a3}. The case when $X$ is the full shift on a countable alphabet has been studied by \cite{KR}.
\end{rem}

\begin{eje}\label{exampleinfty}
In Example \ref{nofinite},
 the sequence $\Phi$ defined on the countable Markov shift $Y$ satisfies  \ref{a0}.  Here we show that \ref{P1} holds. 
Let $X_n$ be the subshift  of $X$ on the symbols $\{F_1, \dots, F_n\}$. Let $Y_n=\pi (X_n)$.  Then $Y_n$ is an irreducible finite Markov shift on $\{1, \dots, n\}$.
For $n\geq 3$, each $\Phi\vert_{Y_{n}}$  satisfies \ref{a1} with $p=3$ and $D=1$. Hence \ref{P1} is satisfied. Thus Proposition \ref {approx1} and Theorem \ref{vp1} hold. 
Since \ref{a3} is not satisfied, by Lemma \ref{cha1}, $P(\Phi)=P_G(\Phi)=\infty$ and equation (\ref{vpinf}) holds.
\end{eje}

\begin{eje}\label{fie1}
In Example  \ref{e1}, the sequence 
$\Psi =\{\log \psi_n\}_{n=1}^{\infty}$ defined on the countable Markov shift $Y$ satisfies \ref{a0}, \ref {a1} and \ref {a3}. 
 Hence Proposition \ref {approx1} and Theorem \ref{vp1} hold.  Since $\Psi$ satisfies  \ref{a1}  with the strong specification, $P(\Psi)=P_G(\Psi)=\lim_{n\rightarrow \infty}(1/n)\log Z_n(\Psi, a)$
for all $a \in \N$.
Since $k\geq 3$, we obtain $Z_1(\Psi)=\sum_{i\in\N}(1/\vert \pi^{-1}(i)\vert^{k-1})\leq \sum_{i\in\N}(1/i^{k-1})<\infty$. Therefore, $P(\Psi)<\infty$ and 
equation (\ref{firstvp}) holds.
\end{eje}

\subsection{Variational principle for finitely irreducible countable sofic shifts} \label{sofic}
In this section, we prove the variational principle for sequences $\F$ with tempered variation (see Definition \ref{bowen}) on finitely irreducible countable sofic shifts (see Definition \ref{def:sofic}). Therefore  the space $X$ 
is not a Markov shift and it has the finiteness property. The regularity condition on $\F$ is weaker than what was assumed in Section \ref{vp:spc}.  Our approach here is based on the proof of  \cite[Theorem 1.2] {mu2}.  


Let  $(X, \sigma)$ be an irreducible countable sofic shift. Then by Definition \ref{def:sofic} there exist an irreducible countable Markov shift  $(\bar{X}, \sigma_{\bar{X}})$ and a one-block factor map $\pi:\bar{X} \rightarrow X$ such that $\vert \pi^{-1}(i)\vert<\infty$ for each $i\in \N$. Rearranging the set $\N$, there is a transition matrix $A$ for $\bar{X}$ and an increasing sequence 
$\{l_{n}\}_{n=1}^{\infty}$ such that the matrix $A_{l_n}=A\vert_{\{1, \dots, l_{n}\}\times \{1, \dots, l_{n}\}}$ is irreducible.
For each $n\in\N$, let  $S_{l_n}=\{\pi(i): 1\leq i \leq \l_n\}$.
Then $(\pi(\bar{X}_{A_{l_n}}), \sigma_{\pi(\bar{X}_{A_{l_n}})})$ is a sofic shift on the set $S_{l_n}$ of finitely many symbols.
Clearly,  $\pi(\bar{X}_{A_{l_n}})\subseteq \pi(\bar{X}_{A_{l_{n+1}}}) \subset X$ and $\N=\cup_{n\in\N} S_{l_n}$.
We note that we can extract a subsequence $\{l_{n_j}\}_{j=1}^{\infty}$ such that $\pi(\bar{X}_{A_{l_{n_j}}})\subset \pi(\bar{X}_{A_{l_{n_{j+1}}}}) \subset X$ for all $n_j, j\in \N$.

We continue to use the notation above throughout this section. The following lemma is important and will be also applied in Section \ref{sectiongibbs}.
\begin{lema}\label{forsofic}
Let  $(X, \sigma)$ be an  irreducible countable sofic shift and  $\F=\{\log f_n\}_{n=1}^{\infty}$  a sequence on $X$ with tempered variation  satisfying \ref{a4} and \ref{a5}. 
Let $p$ be defined as in \ref {a4} and $W$ be defined as in \ref{a5}.
  Then there exists $q\in \N$ such that for each $k\geq q$ there exists an irreducible subshift $(X_{l_k}, \sigma_{X_{l_k}})$ on the set $S_{l_k}$ of finitely many symbols such that  $\pi(\bar{X}_{A_{l_k}})\subseteq X_{l_k} \subset X$.  Moreover, for any $n, m\in\N, k\geq q, u\in B_{n}(X_{l_k}), v\in B_{m}(X_{l_k})$, there exists $w\in W$ such that $uwv$ is an allowable word of $X_{l_k}$ and 
\begin{equation}\label{keyforsofic}
\sup\{f_{n+m+\vert w\vert }\vert_{X_{l_k}}(x) :x\in [uwv]\} \geq \frac{D_{n,m}}{M_{n+m+p}}\sup\{f_{n}\vert_{X_{l_k}}(x) :x\in [u]\}\sup\{f_{m}\vert _{X_{l_k}}(x) :x\in [v]\},
\end{equation}
where $M_n$ is defined as in Definition \ref{bowen}.
\end{lema}
\begin{rem}\label{casebowen} If $\F$ is a Bowen sequence, (\ref{keyforsofic}) implies that  \ref{a1} holds for $\F\vert_{X_{l_k}}, k\geq q$, replacing $D$ in \ref{a1} by $D/M$.
\end{rem}
\begin{proof}
Since $(X, \sigma)$ is an irreducible countable sofic shift, there exist an irreducible countable Markov shift  $(\bar{X}, \sigma_{\bar{X}})$ and a one-block factor map $\pi:\bar{X} \rightarrow X$ such that $\vert \pi^{-1}(i)\vert<\infty$ for each $i\in \N$. Since $W$ is a finite set, only finitely many  symbols appear in  $W$. We first consider the case when $W$ contains a nonempty allowable word.  
Call $S_W$ the set of symbols that appear in $W-\{\varepsilon\}$. 
Let $\pi^{-1}(S_W)$ be the set of preimages of the symbols of $S_W$ in $\bar{X}$. Then  $\pi^{-1}(S_W)$ is a finite set because $\vert \pi^{-1}(i)\vert<\infty$ for each $i\in \N$.

Now consider  a transition matrix $A$ for $\bar{X}$ and an increasing sequence 
$\{l_{k}\}_{k=1}^{\infty}$ such that the matrix $A_{l_k}=A\vert_{\{1, \dots, l_{k}\}\times \{1, \dots, l_{k}\}}$ is irreducible for each $l_k$. Then there exists $q\in\N$ such that  $\pi^{-1}(S_W)\subset \{1, \dots, l_{k}\}$ for all $k\geq q$. Thus, for $k\geq q$ the subshift $(\pi(\bar{X}_{A_{l_k}}), \sigma_{\pi(\bar{X}_{A_{l_k}})})$ is a sofic shift on the set $S_{l_k}$ of finitely many symbols that contains  $S_W$.
For a fixed $k\geq q$, consider  the set $\pi^{-1}(S_{l_k})$ of the preimages of the set $S_{l_k}$ and call it $P$. Then $P$ contains $\{1, \dots, l_{k}\}$ and it is a finite set.   Let $\bar{Y}_P\subset X$ be the finite state Markov shift on the symbols of $P$ and
define $Y=\pi (\bar{Y}_P)$. Then $Y$ is a subshift on the set of $S_{l_k}$ of finitely many symbols which contains $S_W$. Observe that $ \pi(\bar{X}_{A_{l_k}})  \subseteq Y \subset X$.

We observe that $Y$ is irreducible. Fix $n, m\in \N$.  Let $u=u_{1}\dots u_{n}\in B_{n}(Y)$ and $v=v_1\dots v_{m}\in B_{m}(Y)$.
Since these are allowable words of $X$, there exists $w =w_1\dots w_l\in W$, $0\leq l\leq p$, such that $uwv$ is allowable in $X$ and  \ref{a4} holds.
Since $uwv$ is allowable  in $X$, there exists $\bar{u}_1\dots \bar{u}_{n}\bar{w}_1\dots \bar{w}_l\bar{v}_1\dots \bar{v}_{m}\in B_{n+m+l}(\bar{X})$ such that
$\pi ( \bar{u}_1\dots \bar{u}_{n}\bar{w}_1\dots \bar{w}_l\bar{v}_1\dots \bar{v}_{m})=uwv$. Since all the symbols that appear in the preimages of $u, v, w$ are in the set $P$, we obtain that $ \bar{u}_1\dots \bar{u}_{n}\bar{w}_1\dots \bar{w}_l\bar{v}_1\dots \bar{v}_{m} \in B_{n+m+l}(\bar {Y}_P)$. Therefore, $uwv$ is allowable in $Y$ and $Y$ is irreducible. Using the property of tempered variation, 
\begin{align*}
\sup\{f_{n+m+\vert w\vert}\vert_{Y}(y) :y\in [uwv]\} 
&\geq \frac{1}{M_{n+m+p}}\sup\{f_{n+m+\vert w\vert}(x) :x\in [uwv]\} \\
&\geq \frac{D_{n,m}}{M_{n+m+p}} \sup\{f_{n}(x) :x\in [u]\}\sup\{f_{m}(x) :x\in [v]\}\\
&\geq  
 \frac{D_{n,m}}{M_{n+m+p}} \sup\{f_{n}\vert_{Y}(y) :y\in [u]\}\sup\{f_{m}\vert _{Y}(y) :y\in [v]\}.
\end{align*}
For each $k\geq q$, we can construct a such $Y$. 
Setting $Y=X_{l_k}$, we obtain the results. 
If $W=\{\varepsilon\}$, we make a similar argument.  

\end{proof}
Under the setting of Lemma \ref{forsofic} , we define the topological pressure $P(\F)$ as in Definition \ref{defpressure}. By Proposition \ref{cha} we have $Z_1(\F)<\infty$ if and only if $P(\F)<\infty $. We note that if $Z_1(\F)=\infty$, then $P(\F)=\infty$ and the proof is given in that of Theorem \ref{thmsofic}.

\begin{teo}\label{thmsofic}
Let  $(X, \sigma)$ be an irreducible countable sofic shift. If $\F=\{\log f_n\}_{n=1}^{\infty}$ is a sequence on $X$ with tempered variation satisfying \ref{a0}, \ref{a4} and \ref{a5},
then 
\begin{align}\label{aprox2}
P(\F)&=\sup_{\substack{l_n\\ n\geq q}} \{P(\F\vert_{X_{l_n}})\}\\
&=\sup\{P(\F\vert_{Y}): Y\subset X \text{ is an irreducible sofic shift on a finite alphabet} \}, \label{approsofic}
\end{align}
where $X_{l_n}, q$ are defined as in Lemma \ref{forsofic}, and $P(\F)\neq -\infty.$
The variational principle holds.
If $P(\F)<\infty$, then
\begin{equation} \label{vp:sofic}
P(\F)=\sup_{\mu\in M(X, \sigma)} \left\{h_{\mu}(\sigma)+\lim_{n\rightarrow  \infty}\frac{1}{n}\int \log f_n d\mu: \lim_{n\rightarrow  \infty}\frac{1}{n}\int \log f_n d\mu>-\infty \right\}.
\end{equation}
If $P(\F)=\infty$, then 
\begin{equation}\label{vp:soficinf}
P(\F)=\sup_{\mu\in M(X, \sigma)} \left\{h_{\mu}(\sigma)+\limsup_{n \rightarrow  \infty}\frac{1}{n}\int \log f_n d\mu: \limsup_{n\rightarrow  \infty}\frac{1}{n}\int \log f_n d\mu>-\infty \right\}.
\end{equation}
 \end{teo}

\begin{rem}
Condition \ref{a5} implies that $(X, \sigma)$ is a finitely irreducible countable sofic shift.  If $\sup f_1<\infty$, then 
(\ref{vp:sofic}) also holds for the case when $P(\F)=\infty$. 
\end{rem}

\begin{proof}
We first consider the case when $Z_1(\F)<\infty$. Then $P(\F)<\infty$ by Proposition \ref{cha}.
Note that there exist an irreducible countable Markov shift $(\bar {X}, \sigma_{\bar{X}})$ and a one-block factor map 
$\pi:\bar X \rightarrow X$ such that  $\vert \pi^{-1}(i)\vert<\infty$ for each $i\in \N$.
We show first (\ref{aprox2}) using a  modification of the proof of \cite[Theorem 1.2]{mu2}. As in the proof of Proposition \ref{approx1}
let $f'(x)= e^Cf_n(x)$ and $\F'=\{\log f'_n\}_{n=1}^{\infty}$. 
Then $\F'$ is sub-additive and $P( \F)=P( \F')$. Let $M_n$ be defined for $\F$ as in Definition \ref{bowen}.

Let $\epsilon>0.$ Fix $m\in \N$ such that  $(1/m)\log M_m< \epsilon$,  $(1/(m+p))\vert \log ({D_{m,m}}/{e^C})\vert <\epsilon $ and $1-\epsilon< (m/(m+p))$. Note that 
$Z_m(\F')<\infty$.

We apply Lemma \ref{forsofic} and consider $X_{l_k}$ where $k\geq q$. Then  for each $n\in\N$, we have 
\begin{equation}
Z_n(\F'\vert _{X_{l_k}})=\sum_{w\in B_n(X_{l_k})}\sup\{f'_n\vert_{X_{l_k}}(x):x\in [w]\}.
\end{equation}
Since $w\in B_m(X_{l_k})$ implies that $w\in B_m(X)$,
let $S_{l_k}(\F'):=\sum_{w\in B_m(X_{l_k})}\sup\{f'_m(x):x\in [w]\}$.
Noting that for each $x_1\dots x_m\in B_m(X)$, there exists $i\in\N$  such that 
$x_1\dots x_m \in B_m(X_{l_i})$, 
\begin{equation}
Z_m(\F')=\lim_{i\rightarrow \infty}S_{l_i}(\F') ,
\end{equation}
where $\{S_{l_i}(\F')\}_{i=1}^{\infty} $is monotone increasing.
Hence, for every  $\epsilon>0$, there exists $k_1>  q$ such that
\begin{equation}
\frac{1}{m}\log Z_m(\F')-\frac{1}{m}\log S_{l_{k_1}}(\F')  < \epsilon.
\end{equation}
Since $\F$ has tempered variation,  we have that $M_mZ_m(\F'\vert _{X_{l_k}})\geq S_{l_k}(\F')$.
Since $\F'$ is sub-additive, we obtain 
\begin{equation}\label{first}
\frac{1}{m}\log  Z_m(\F'\vert _{X_{l_{k_1}}}) \geq 
\frac{1}{m}\log Z_m(\F')-\epsilon-\frac{\log M_m}{m} \geq P(\F') -2\epsilon.
\end{equation}

Now, for $0\leq i\leq n$, $n\in\N$, let $u_i\in B_m(X_{l_{k_1}})$.   Since $\F$ satisfies \ref{a4} and  \ref{a5}, letting  $W$ be a finite set from \ref{a5}, there exist
 $w_1, \dots, w_{n-1}$ in $W$ such that 
$u_1w_1\dots w_{n-1}u_n$ is an allowable word of length $nm+ \vert w_1 \vert +\dots +\vert w_{n-1}\vert$ of $X$, such that 
\begin{equation}\label{csofic}
\begin{split}
\sup\{f'_{nm+\vert w_1 \vert +\dots +\vert w_{n-1}\vert }(x): x\in  [u_1w_1\dots w_{n-1}u_n]\}
&\geq
\left(\frac{D_{m,m}}{e^C} \right)^{n-1}\prod _{i=1}^{n} \sup\{f'_{m}(x) :x \in [u_i]\}.
\end{split}
\end{equation}
By the construction of $X_{l_k}$, $k\geq q$, in the proof of Lemma \ref{forsofic}, we note that $u_1w_1\dots w_{n-1}u_n$ is an allowable word of $X_{l_{k_1}}$. Therefore, 

\begin{equation}
\begin{split}
&M_{nm+p(n-1)}\sup\{f'_{nm+\vert w_1 \vert +\dots +\vert w_{n-1}\vert }\vert_{X_{l_{k_1}}} (x): x\in  [u_1w_1\dots w_{n-1}u_n]\}
\label{control}\\
&\geq 
\sup\{f'_{nm+\vert w_1 \vert +\dots +\vert w_{n-1}\vert }(x): x\in  [u_1w_1\dots w_{n-1}u_n]\}\\
&\geq 
\left(\frac{D_{m,m}}{e^C}\right)^{n-1}\prod _{i=1}^{n} \sup\{f'_{m}(x) :x \in [u_i]\}
\geq 
 \left(\frac{D_{m,m}}{e^C}\right)^{n-1}  \prod _{i=1}^{n} \sup\{f'_{m}\vert_{X_{l_{k_1}}}(x) :x \in [u_i]\}.
\end{split}
\end{equation}
Summing over all allowable words  $u_i\in B_m(X_{l_{k_1}})$, $0\leq i\leq n$, we obtain 
\begin{equation*}
\begin{split}
\sum _{ 0\leq t\leq p(n-1)} Z_{nm+t} (\F'\vert _{X_{l_{k_1}}})
&\geq 
 \left(\frac{D_{m,m}}{e^C}\right)^{n-1}\cdot \frac{1}{M_{nm+p(n-1)}}(Z_{m} (\F'\vert _{X_{l_{k_1}}}))^n.
\end{split}
\end{equation*}
Hence, there exists $0\leq i_{n,m}\leq p(n-1)$ such that 
\begin{equation*}
\begin{split}
Z_{nm+i_{n,m}} (\F'\vert _{X_{l_{k_1}}})
&\geq \left(\frac{D_{m,m}}{e^C}\right)^{n-1}\cdot \frac{1}{M_{nm+p(n-1)}} \cdot \frac{1}{p(n-1)+1}\cdot 
\left(Z_{m} (\F'\vert _{X_{l_{k_1}}}) \right)^n.
\end{split}
\end{equation*}
Thus
\begin{equation*}
\begin{split}
&\frac{1}{nm+i_{n,m}}\log (Z_{nm+i_{n,m}} (\F'\vert _{X_{l_{k_1}}}))\\
&\geq \frac{1}{nm+p(n-1)} \log ((\frac{D_{m,m}}{e^C})^{n-1}\cdot \frac{1}{M_{nm+p(n-1)}}  \cdot \frac{1}{p(n-1)+1}) +  \frac{n}{nm+p(n-1)} \log Z_{m} (\F'\vert _{X_{l_{k_1}}}).
\end{split}
\end{equation*}
Letting $n\rightarrow \infty$ and using (\ref{first}) we have,
\begin{equation}
\begin{split}
\limsup_{n\rightarrow\infty} \frac{1}{nm+i_{n,m}}\log (Z_{nm+i_{n,m}} (\F'\vert _{X_{l_{k_1}}}))&\geq  \frac{1}{m+p} \log \frac{D_{m,m}}{e^C}+  \frac{m}{m+p} \cdot \frac{1}{m}\log Z_{m} (\F'\vert _{X_{l_{k_1}}}) \label{forcontrol}\\
&\geq -2\epsilon-\epsilon P(\F')+ 2\epsilon^2+P(\F').\\
\end{split}
\end{equation}
Therefore,  we obtain (\ref{aprox2}). 

Next assume $Z_1(\F)=\infty$. We first show that $P(\F)=\infty$. Given $L>0$, there exists $X_{l_{s}}, s\geq  q$ such that $Z_1(\F\vert_{X_{l_s}})>L$. Then $Z_1(\F'\vert_{X_{l_s}})>Le^C$. Let $Y:=X_{l_s}$.  Then 
for each $n\in\N$ there exists $0\leq i_{n,1}\leq p(n-1)$ such that 
\begin{equation}\label{keyforinfty}
\begin{split}
&\frac{1}{n+i_{n,1}}\log (Z_{n+i_{n,1}} (\F'\vert _Y))\\
&\geq \frac{1}{n+p(n-1)} \log ((\frac{D_{1,1}}{e^C})^{n-1}\cdot \frac{1}{M_{n+p(n-1)}}  \cdot \frac{1}{p(n-1)+1}) +  \frac{n}{n+p(n-1)} \log Z_{1} (\F'\vert _Y).
\end{split}
\end{equation}
A similar argument as in the proof of Lemma \ref{gdefinf} implies $P(\F)=\infty$.   The approximation property (\ref{aprox2}) is obvious from (\ref{keyforinfty}).
 
Since Propositions \ref{cha}, \ref{upper} and Lemma \ref{upper1} hold, the same proof (using the approximation property (\ref{aprox2})) as in the proof of Theorem \ref{vp1} yields the variational principle, equations \eqref{vp:sofic} and \eqref{vp:soficinf}. It is easy to see that (\ref{approsofic}) holds by Lemma \ref{forsofic} and its proof.
\end{proof}
In the following, we study a condition for which $P(\F)$=$P_{G}(\F)$, when $\G$ is defined on a countable sofic shift. 
\begin{prop}\label{defaasofic}
Let $(X, \sigma)$ be a finitely irreducible countable sofic shift. If $\mathcal{G}$ is an almost-additive sequence on $X$ with tempered variation,  then $P(\mathcal{G})=P_{G}(\mathcal{G}).$
In particular, if $X$ is a factor of a finitely primitive countable Markov shift and $P(\G)<\infty$, then $\limsup$ in (\ref{gurev}) can be replaced by $\lim$.
\end{prop}
\begin{proof}
First assume $Z_1(\G)<\infty$. Thus $P(\G)<\infty$.
Since $X$ is a finitely irreducible countable sofic shift, let $\bar{X}$ and  $\pi: \bar{X}\rightarrow 
X$ be as in the proof of Lemma \ref{forsofic}. 
Let $p\in\N$ and a finite set $W_1$ be defined for $X$ as in Definition \ref{fi}. We consider the case when $W_1\neq \{\varepsilon\}$. Let $x_1\dots x_n\in B_n(X)$ and $a\in \N$ be a symbol in $X$.
Then there exist allowable words $w_1, w_2$ in $W_1$ of length $0\leq k_1, k_2\leq p$ respectively such that $aw_1x_1\dots x_nw_2a \in B_{n+2+k_1+k_2}(X)$. Therefore, there exist 
$\bar{x}_1\dots \bar{x}_n\in \pi^{-1}(x_1\dots x_n)$,
$a_1, a_2\in \pi^{-1}(a)$, $\bar{w}_1 \in \pi^{-1}(w_1)$ and $\bar{w}_2\in \pi^{-1}(w_2)$ such that 
$a_1\bar{w}_1\bar{x}_1\dots \bar{x}_n\bar{w}_2a_2\in  B_{n+k_1+k_2+2}(\bar{X})$ and $\pi(a_1\bar{w}_1\bar{x}_1\dots \bar{x}_n\bar{w}_2a_2)=aw_1x_1\dots x_nw_2a$. 
Since $\vert \pi^{-1}(a)\vert<\infty$, we have $\pi^{-1}(a)=\{a_1,\dots, a_t\} $ for some $t\in \N$. 
For each pair $a_i, a_j, 1\leq i,j\leq t$, define $k_{i,j}=\min\{\vert w\vert: a_iwa_j\in B_{2+\vert w\vert}(\bar{ X}), \vert w \vert \geq 1\}$. Then for each $i,j$, there exist a word at which the minimum is attained and we call it  $\bar{w}_{i,j}\in B_{k_{i,j}}(\bar X)$. Let $\pi(\bar{w}_{i,j})=w_{i,j}$.

Now let $\bar{x}=(a_1\bar{w}_1\bar{x}_1\dots \bar{x}_n\bar{w_2}a_2\bar{w}_{2,1})^{\infty}\in \bar{X}$ and $x=\pi(\bar{x})$. Then $x$ has a period $(n+2+k_1+k_2+k_{2,1})$ in $X$. 
We first consider the case when $k_1, k_2$ are both nonzero. Since $\G$ is almost-additive and has tempered variation, letting $N_a=\sup \{f_{1}(x): x \in [a]\}$, we obtain 
\begin{equation}\label{aabelow}
\begin{split}
&g_{n+k_1+ k_2+k_{2,1}+2}(x)\\
&\geq \frac{e^{-5C}}{M_n (M_p)^2 (M_1)^2 M_k}\sup \{g_{n}(x): x\in [x_1\dots x_n]\}
(N_a)^2 \sup \{g_{k_1}(x): x\in [w_1]\} \\
&\cdot\sup \{g_{k_2}(x): x\in [w_2]\}
\sup \{g_{k_{2,1}}(x): x\in [w_{2,1}]\}.
\end{split}
\end{equation}

Since $g$ has tempered variation, for each $1\leq i,j\leq t$, there exists constant $C_{w_{i,j}}>0$ such that $ \sup \{g_{k_{i,j}}(x): x\in [w_{i,j}]\}>C_{w_{i,j}}$.
Since we have finitely many $i,j$, let $B=\min_{i,j}C_{w_{i,j}}$. 
and $K=\max_{i,j}k_{i,j}$ 

Now we consider the case when at least one of $k_1, k_2$ is $0$. 
Observe that if $k_1$ is $0$, then we replace $\sup \{g_{k_1}(x): x\in [w_1]\}$ in (\ref{aabelow})
 by $1$. This applies also to $k_2$.
Clearly  there exists $\bar{D}>0$ such that $\min_{w \in W_1, \vert w\vert\geq 1}\sup \{g_{l}(x): x\in [w]\}>\bar{D}$. Let$ \bar {D}'=\min\{1, \bar D\}$. Then, 
(\ref{aabelow}) implies that 
\begin{equation}\label{fis}
\begin{split}
\sum_{0\leq i \leq 2p+K}Z_{n+i+2}(\mathcal{G}, a)
&\geq \frac{e^{-5C}}{M_n(M_p)^2 (M_1)^2 M_K}Z_n(\G)(N_a)^2 B\bar{D'}^2.
\end{split}
\end{equation}
Thus similar arguments as in the proof of Proposition \ref{gdefi} yield
\begin{equation*}
\limsup_{n\rightarrow\infty}\frac{1}{n}\log Z_{n}(\mathcal{G}, a)
\geq \limsup_{n\rightarrow\infty}\frac{1}{n}\log Z_{n}(\mathcal{G}).
\end{equation*}
Since $a$ is arbitrary, we obtain the result.  

Next assume that $Z_1(\G)=\infty$. Then $P(\G)=\infty$. Let $\G'=C+\G$. Given $L>0$,  there exists $X_{l_{s}}, s\geq q$ such that $Z_1(\G\vert_{X_{l_s}})>L$. Let $Y:=X_{l_s}$.
Then (\ref{keyforinfty}) holds if we replace $\F'$ by $\G'$.
Since $P(\mathcal{G'}\vert_Y)=P_{G}(\mathcal{G'}\vert_Y)$,  similar arguments as in the proof of Lemma \ref{gdefinf} imply $P_G(\G)=\infty$
To show the second statement, we use the  similar arguments as in the proof of Proposition \ref{gdefi}. 
If $\bar X$ is a finitely primitive countable Markov shift, let $p$ be a strong specification number for $\bar X$ and set $k_1=k_2=K=p$. 
 
\end{proof}

Note that Theorem \ref{thmsofic} generalizes the thermodynamic formalism on non-compact shifts, including now irreducible countable sofic shifts. Indeed, 

\begin{coro}\label{special1}
Let  $(X, \sigma)$ be a finitely irreducible countable sofic shift. If $\F$ is an almost-additive sequence on $X$ with tempered variation,
then Theorem \ref{thmsofic}
holds for $\F$ and 
$P(\F)=P_{G}(\mathcal{F})$. In particular, Theorem \ref{thmsofic} holds for a continuous function $f$ on $X$  with tempered variation 
by setting $f_n(x)=e^{(S_nf)(x)}$ for all $x\in X$.
\end{coro}
\begin{proof}
By Lemma \ref{d}, $\F$ satisfies \ref{a0}, \ref{a4} and\ref{a5}. 
For the last statement, we also apply Example \ref{basic1}. 
\end{proof}
\begin{rem}\label{add}
The variational principle is proved in \cite[Theorem 1.5] {mu2} for acceptable functions (uniformly continuous functions with an additional property) on finitely irreducible countable Markov shifts. Applying \cite[Proposition 6.2]{ffy}, it is easy to see that acceptable functions belong to the class of continuous functions with tempered variation.
In \cite[Theorem 2.4]{ffy}, the variational principle is studied for continuous functions with tempered variation on irreducible countable Markov shifts, without the finiteness condition on each $M_n$.
We also note that Corollary \ref{special1} generalizes the variational principle \cite[Theorem 3.1]{iy1} to that for almost-additive sequences with tempered variation on finitely irreducible countable sofic shifts. 
\end{rem}



Next we consider examples of Theorem \ref {thmsofic}.

\begin{eje}\label{newexample}
Let $\G$ be defined as in Theorem \ref{ftemperedv}. Then $\G$ is a Bowen sequence defined on a finitely irreducible countable sofic shift satisfying \ref{a0},  \ref{a4} and \ref{a5}.  Note that $\G$ does not satisfy \ref{a1}. Theorem  \ref {thmsofic} is applied in Theorem \ref{ftemperedv}. See Section \ref{hiddeng} for more details.
\end{eje}

\begin{eje}\label{finite2}
In Example  \ref{e11}, the sequence 
$\Psi =\{\log \psi_n\}_{n=1}^{\infty}$ defined on an irreducible countable sofic shift $Y$ satisfies \ref{a0}, \ref {a1} and \ref {a3}. Hence 
Theorem \ref {thmsofic}  holds. Since $Z_1(\Psi)\leq C_2\sum_{i\in\N}(1/{i^{2}})<\infty$ , we obtain $P(\Psi)<\infty$ and equation (\ref{vp:sofic}) holds.
\end{eje}

\begin{eje}\label{finite3}
In Example  \ref{exgeneral}, 
define for $i\in\N$
\begin{equation*}
 L_i:=\frac{\vert \pi^{-1}(i+1)\vert }{\vert \pi^{-1}(i)\vert K}.
 \end{equation*} 
Choose $K>0$  and define a factor map $\pi$ such that $\lim_{i\rightarrow\infty} L_i $ exists and $L:=\lim_{i\rightarrow\infty} L_i<1$.
Then the sequence $\Psi =\{\log \psi_n\}_{n=1}^{\infty}$
defined on a finitely irreducible countable sofic shift $Y$ satisfies \ref{a0}, \ref {a1} and \ref {a3}. Hence 
Theorem \ref {thmsofic}  holds. Since $Z_1(\Psi)<\infty$ by using the ratio test, we obtain $P(\Psi)<\infty$ and  equation (\ref{vp:sofic}) holds.
If there exists $l\in \N$ such that $\vert \pi^{-1}(i)\vert \leq l$ for all $i\in \N$ and $K>1$, then the same results hold.
If we define a constant $K>0$ and a factor map $\pi$ so that $L>1$, then $P(\Psi)=\infty$ and  equation (\ref{vp:soficinf}) holds.
\end{eje}



\section{Invariant Gibbs measures and uniqueness of Gibbs equilibrium measures}\label{sectiongibbs}

The variational principle provides a criteria to choose relevant invariant measures for the (very large) set  $M(X, \sigma)$ of invariant Borel probability measures. Indeed, measures that maximize the supremum have interesting ergodic properties. 
Major difficulties to prove the existence of these measures are the fact that the space $M(X, \sigma)$ is not compact (when endowed with the weak* topology) and that the entropy map $\mu \mapsto h_{\mu}(\sigma)$ is not necessarily upper-semi continuous. Despite this we prove that under certain assumptions on the system and the class of sequence of functions such measures do exist. Moreover, they satisfy the so called \emph{Gibbs} property which relates the measure of a cylinder of length $n$ with the function $f_n$. This property turns out to be very useful in a wide range of applications, for example in dimension theory of dynamical systems. The goal of this section is to prove under some conditions the existence and uniqueness of ergodic Gibbs measures for the Bowen sequences on finitely irreducible countable sofic shifts and the uniqueness of 
equilibrium states. The results are presented in Section \ref{statements} and the proofs of some technical lemmas are to be  found in Section \ref{lemmas}.

\subsection{Invariant Gibbs measures and uniqueness of Gibbs equilibrium measures} \label{statements}

Throughout this section, we assume that $\F=\{\log f_n\}_{n=1}^{\infty}$ is  a sequence defined on a finitely irreducible countable sofic shift $(X, \sigma)$ satisfying \ref{a0}, \ref{a1}, \ref{a3} and \ref{a2}.

\begin{defi}\label{dgibbs}
Let  $(X ,\sigma)$ be a subshift on a countable alphabet and 
$\F=\{\log f_n\}_{n=1}^{\infty}$ a sequence on $X$ satisfying \ref{a0}, \ref{a1}, \ref{a3} and \ref{a2}. A measure $\mu \in M(X, \sigma)$ is said to be an \emph{equilibrium measure} for $\F$ if
\begin{equation*}
P(\F)= h_{\mu}(\sigma) + \lim_{n \to \infty} \frac{1}{n} \int \log f_n \ d \mu.
\end{equation*}
\end{defi}

\begin{defi} \label{def-gibbs}
Let  $(X, \sigma)$ be a subshift on a countable alphabet and  
$\F=\{\log f_n\}_{n=1}^{\infty}$ a sequence on $X$ satisfying \ref{a0}, \ref{a1}, \ref{a3} and \ref{a2}. A measure $\mu \in M(X, \sigma)$ is said to 
be \emph{Gibbs} for $\F$ if there exist constants $C_{0}>0$ and $P \in \R$ such that for every $n \in \N$ and every $x \in [i_1 \dots i_{n}]$ we have
\begin{equation*}
\frac{1}{C_{0}} \leq	\frac{\mu([i_1 \dots i_{n}])}{\exp(-nP)f_n(x)}	\leq C_{0}.
\end{equation*}
 \end{defi}
A Gibbs measure $\mu$ for a continuous function $\phi$ could satisfy 
$h_{\mu}(\sigma)= \infty$ and $\int \phi \ d \mu= - \infty$.  In such a situation, the measure $\mu$ is not an equilibrium measure for $\phi$ (see \cite{s3} for comments and examples). 

Existence of Gibbs measure was studied in \cite{iy1,iy2} for an almost-additive sequence on a topologically mixing countable Markov shift with BIP property and in  \cite[Theorem 3.7]{KR} for a class of  sub-additive Bowen sequences on the full shift on a countable alphabet satisfying
\ref{a1}, \ref{a3} and \ref{a2}. Here we will generalize these results by considering a finitely irreducible countable sofic shift. The main result of this section is the following. 

%

\begin{teo}\label{main2}
Let $(X, \sigma)$ be a finitely irreducible countable sofic shift. If $\F=\{\log f_n\}_{n=1}^{\infty}$ is a Bowen 
sequence on $X$ satisfying \ref{a0}, \ref{a1}, \ref{a3} and \ref{a2}, then there is a unique invariant ergodic Gibbs measure $\mu$ for $\F$. Moreover, if in addition 
\begin{equation*}
\sum_{i\in \N}\sup \{\log f_1(x):x\in [i]\} \sup \{ f_1(x):x\in [i]\}>-\infty,
\end{equation*}
 then $\mu$  is the unique equilibrium measure for $\F$ on $X$. 
\end{teo}

\begin {rem}
By  Proposition \ref{cha1} \ref{a2} is equivalent to  $P(\F)<\infty$.  
\end{rem}


\begin{coro}
Let $(X, \sigma)$ be a finitely irreducible countable Markov shift and
$\mathcal{G}=\{\log g_n\}_{n=1}^{\infty}$ an almost-additive Bowen 
sequence on $X$. 
If $\mathcal{G}$ satisfies \ref{a2},  then there is a unique Gibbs measure $\mu$ for $\G$ 
and it is ergodic. Moreover, if in addition 
\begin{equation*}
\sum_{i\in \N}\sup \{\log g_1(x):x\in [i]\} \sup \{ g_1(x):x\in [i]\}>-\infty,
\end{equation*}
 then $\mu$ is the unique equilibrium measure for $\mathcal{G}$. 
\end{coro}

\begin{proof}
Lemma \ref{d} implies that $\mathcal{G}$ satisfies  \ref{a1} and \ref{a3}. Now apply Theorem \ref{main2}.
\end{proof}

\begin{rem}
Theorem \ref{main2} generalizes \cite[Theorem 4.1] {iy1} in which almost-additive Bowen sequences on finitely primitive countable Markov shifts are considered. If $\mathcal{G}=\{\log g_n\}_{n=1}^{\infty}$ is an almost-additive Bowen sequence, then 
$\sum_{i\in \N}\sup \{\log g_1(x):x\in [i]\} \sup \{ g_1(x):x\in [i]\}>-\infty$ is equivalent to $h_{\mu}(\sigma)<\infty$ where $\mu$ is the Gibbs measure (see \cite[Proposition 3.1]{iy2}).
\end{rem}

In Theorem \ref{main2}, we study the case when $W\neq \{\varepsilon\}$ (see Remark \ref{simplification}). 
Hence, throughout  the rest of the section, without loss of generality we assume


 \begin{enumerate}[label=(A\arabic*)]
\item \label{u} $\F=\{\log f_n\}_{n=1}^{\infty}$ satisfies \ref{a0}, \ref{a1} with some $p\in \N$ and \ref{a3} with a finite set $W$ containing a nonempty word $w^*$ of length $p$, \label{u}
\end{enumerate}

and 
 \begin{enumerate} [label=(A\arabic*)]
  \setcounter{enumi}{1}
\item  In Lemma \ref{forsofic}, for all $k\geq q$,
$w^*\in W$ appears in (\ref{keyforsofic}) for a pair of allowable words $u, v$ of  $X_{l_k}$.\label{spfy}
\end{enumerate}
To see \ref{spfy}, note that  since $W$ from \ref{a3} contains $w^*$
there exist $N_1, N_2$ and a pair $\bar u\in B_{N_1}(X), \bar v\in B_{N_2}(X)$ such that $\bar uw^*{\bar v}$ is an allowable word of $(N_1+N_2+p)$ satisfying \ref{a1}.  In the proof of Lemma \ref{forsofic}, we take $S_{l_k}$ large enough so that it contains all the preimages of symbols that appear in $\bar u$ and $\bar v$. 



The idea of the proof  of Theorem \ref{main2} is similar to that of  \cite[Theorem 4.1]{iy1}, which in turn was proved using  techniques of \cite[Lemma 2.8]{mu2} and \cite[Lemmas 1, 2 and Theorem 5]{b2}. The modification of the proof has to be adapted to the fact that condition \ref{a1} replaces the lower bound condition (\ref{A1}) of an almost-additive sequence.  
We continue to use the notation from Lemma  \ref{forsofic}. 

\begin{teo}\cite{Fe4}  \label{feng} 
Let $(X, \sigma)$ be an irreducible subshift on a finite alphabet. If $\F=\{\log f_n\}_{n=1}^{\infty}$ is a Bowen sequence on $X$ satisfying \ref{a0} and \ref{a1}, then there exists a unique Gibbs measure for $\F$. Moreover, it is the unique equilibrium measure for  $\F$.
\end{teo}



\begin{prop}\label{claim}  
For $n\geq  q$, there is a unique equilibrium measure for $\F\vert_{X_{l_n}}$ 
and it is Gibbs for $\F\vert_{X_{l_n}}$. Moreover, the Gibbs constant $C_{0}$ (see Definition \ref{def-gibbs}) can be chosen  independently of $X_{l_n}$.
\end{prop}

\begin{proof} 
The first part of Proposition \ref{claim} follows from Theorem \ref{feng}. Indeed, note that since $\F$ is a Bowen sequence satisfying 
\ref{a0}, \ref{a1}, \ref{a3} and \ref{a2} and $X_{l_n}$ contains all allowable words in $W$ for $n\geq q$, we have that
$\F\vert_{X_{l_n}}$ is  a sequence on $(X_{l_n}, \sigma_{X_{l_n}})$ satisfying (\ref{keyforsofic})  replacing $D_{n,m}/M_{n+m+p}$ by $D/M$.

In order to prove the second claim in Proposition \ref{claim} we will modify the proof  of  \cite[Claim 4]{iy1}  considering equation (\ref{keyforsofic}). 
By the assumptions, any allowable word in $W$ is an allowable word of $X_{l_n}$ for all  
$n\geq q$. Fix $X_{l_n}$, $n\geq q$, and call it $Z$.  
Define $\alpha_n^{Z}=\sum_{i_1\dots i_{n}\in B_{n}(Z)}\sup\{f_n\vert _Z(z): z\in [i_1\dots i_{n}]\}$. By  the sub-additive property of 
$\{\log e^Cf_n\}_{n=1}^{\infty}$, we have for $l,n\in\N$ that
\begin{equation}\label{p3}
\alpha_{n+l}^{Z}\leq e^C\alpha_n^{Z}\alpha_l^{Z}.
\end{equation}
Hence  $\{\log (e^{C}\alpha_n^{Z})\}_{n=1}^{\infty}$ is sub-additive.
We claim that  for some $C_1>0$ the sequence $\{\log (C_1\alpha_n^{Z})\}_{n=1}^{\infty}$ is super-additive. 
In order to show this, we adapt the arguments of the proof  of \cite[Claim 4]{iy1} to our setting. For $l \in \N$, let $\nu_{l}$ be the Borel probability measure on $Z$ defined by  
\begin{equation*}
\nu_{l}([i\dots i_{l}])=\frac{\sup\{f_{l}\vert_{Z}(z):z\in [i_{1}\dots i_{l}]\}}{\alpha_{l}^{Z}}.
\end{equation*}
By Lemma \ref{forsofic}, for any allowable words 
$u=u_{1}\dots u_{n}$ and $v=v_1\dots v_{l}$ of $Z$, $n, l\in \N$, there exists $w\in B_{\vert w\vert}(Z)\in W$, 
$0\leq \vert w\vert\leq p$
such that $uwv$ 
is an allowable word of $Z$ and that 
\begin{equation}\label{sofic1}
\sup \{f_{n+\vert w \vert +l}\vert_Z(z): z \in [ uwv]\} \geq \frac{D}{M}\sup\{f_n\vert_Z(z): z\in [u]\}\sup \{f_{l}\vert_Z(z): z\in [v]\}.
\end{equation}

For a fixed $\bar{u}\in B_n(Z)$, considering all possible $v\in B_{l}(Z)$ with $w$ satisfying (\ref{sofic1}) and then considering  all possible  $\bar{u}\in B_n(Z)$, we obtain
 
\begin{equation*}
\sum_{i=0}^{p}\alpha_{n+l+i}^{Z}\geq \frac{D}{M}  \alpha_{n}^{Z} \alpha_{l}^{Z}.
\end{equation*}
Let $D/M:=D_1$. Then for each $n, l \in\N$, there exists $0\leq i_{n,l}\leq p$ such 
that $\alpha_{n+l+i_{n,l}}^{Z}\geq (D_1 \alpha_{n}^{Z} \alpha_{l}^{Z})/(p+1).$
By sub-additivity of $\{\log (e^C\alpha_{n}^Z)\}_{n=1}^{\infty}$, we obtain 
\begin{equation*}
\alpha_{n+l+i_{n,l}}^{Z}\leq e^C\alpha_{n+l} ^{Z}\alpha_{i_{n,l}}^{Z}
\leq e^{Cp}\alpha_{n+l}^{Z} (\alpha_{1}^{Z})^{i_{n,l}}.
\end{equation*}
Letting $K=\max_{0\leq i\leq p}Z(\F)^i$, for any $n,l\in \N$ we have
\begin{equation}\label{k0}
\alpha_{n+l}^{Z}\geq D_1  \alpha_{n}^{Z} \alpha_{l}^{Z}/(e^{Cp}K(p+1)).
\end{equation}
Let $C_1=D_1/(e^{Cp}K(p+1))$. Since $P(\F\vert _{Z})=\lim_{n \rightarrow\infty}(1/n)(\log \alpha_n^{Z})$, we use the argument in \cite [Claim 4.1]{iy1}.  The sub-additivity of $\{\log (e^C \alpha_n^{Z})\}_{n=1}^{\infty}$, the super-additivity of $\{\log (C_1\alpha_n^{Z})\}_{n=1}^{\infty}$ and 
$Z_1(\mathcal{F})<\infty$ imply that
\begin{equation}\label{key0}
C_1 \alpha_n^{Z}\leq e^{P(\F\vert_Z)n} \leq e^{C}\alpha_n^{Z}.
\end{equation}

We now construct a Gibbs measure using  similar arguments as those in the proof of \cite[Theorem 5]{b2}. For  fixed $u\in B_n(Z)$, $m\in \N$, we define  $\alpha^{Z,u}_{n+m}=\sum_{ua_1\dots a_{m}\in B_{n+m}(Z)} \sup\{f_{n+m}\vert_Z (z): z\in [ua_1\dots a_{m}]\}.$ 

\begin{lema}\label{key7}
There exists $C_2>0$ such that for each fixed $u\in B_{n}(Z)$,  for $l>n+2p$, we have
$$\alpha^{Z,u}_l \geq C_2 \alpha^Z_{l-n-2p}\sup\{f_{n}\vert_Z (z): z\in [u]\}.$$ 
Note that $C_2$  is independent of $Z$.
\end{lema}
\begin{proof} For the  proof, see Section \ref{lemmas}.
\end{proof}

By the definition of the measure $\nu_{l}$ and \ref{a0}, for a fixed $u=u_1\dots u_n\in B_n(Z)$, $n<l$, we have that,
\begin{equation*}
\nu_l([u])\leq \frac{e^C\sup\{f_n\vert_Z(z): z\in [u]\}\alpha^Z_{l-n}}{\alpha^Z_l}.
\end{equation*}
Therefore, using (\ref{key0}), we obtain that for each $z\in [u]$
\begin{align*}
\frac{\nu_l([u])}{e^{-nP(\F\vert_Z)}f_n\vert_Z(z)}&\leq \frac{M\nu_l([u])}{e^{-nP(\F\vert_Z)}\sup\{f_n\vert_Z(z): z\in [u]\}}
\leq \frac{Me^{2C}\alpha^Z_{l-n}\alpha^Z_{n}}{\alpha^Z_{l}}\leq \frac{Me^{3C}}{C^2_1}.
\end{align*}
On the other hand, by Lemma \ref{key7} and (\ref{key0}), for each $z\in [u]$, for $l>n+2p$,
\begin{equation*}
\frac{\nu_l([u])}{e^{-nP(\F\vert_Z)}f_n\vert_Z(z)}\geq  \frac{\alpha^{Z,u}_{l}}{\alpha^Z_{l}e^{-nP(\F\vert_Z)}\sup\{f_n\vert_Z(z):z\in [u]\}}
\geq C_1C_2e^{-2pP(\F\vert_Z)-C}.
\end{equation*}
Noting that $e^{-2pP(\F\vert_Z)}\geq  e^{-2pP(\F)}$ if $P(\F)\geq 0$ and $e^{-2pP(\F\vert_Z)}> 1 $ if $P(\F)<0$,
there exist $C_3>0, C_4>0$, both independent of $Z$, such that for all $l>n+2p$,
\begin{equation}
C_3\leq \frac{\nu_l([u])}{e^{-nP(\F\vert_Z)}f_n\vert_Z(z)}\leq C_4 \text { for all  } z\in [u].
\end{equation}
Since the set $Z$ is compact, there exists a subsequence $\{\nu_{n_k}\}_{k=1}^{\infty}$ of $\{\nu_{n}\}_{n=1}^{\infty}$ that converges to a measure $\nu$  and for all  $z\in [u]$ \begin{equation}
C_3\leq \frac{\nu([u])}{e^{-nP(\F\vert_Z)}f_n\vert_Z(z)}\leq C_4.
\end{equation}
Now let $\mu_n=(1/{n})\sum_{i=1}^{n}\sigma_Z^{i}\nu$. We claim that any  weak limit point $\mu$ of $\{\mu_n\}_{n=1}^{\infty}$ is a $\sigma_Z$-invariant Gibbs measure on $Z$.

For each fixed $u\in B_{n}(Z)$,  define $\alpha^Z_{l+n} (u) =\sum_{a_1\dots a_{l}u\in B_{l+n}(X)}\sup\{f_{l+n}\vert_Z(z): z\in [a_1\dots a_{l}u]\}.$ Then setting $l=m+i$, for $m\in\N, 0\leq i\leq p$,  we obtain that 
$\sum_{0\leq i\leq p} \alpha^Z_{n+m+i} (u)\geq D_1\alpha^Z_m\sup\{f_{n}\vert_Z(z): z\in [u]\}$. 
Therefore, there exists $0\leq i_{n,m,u}\leq p$ such that 
\begin{equation*}
\alpha^Z_{n+m+i_{n,m,u}} (u)\geq (D_1/(p+1))\alpha^Z_m\sup\{f_{n}\vert_Z(z): z\in [u]\}.
\end{equation*}
Note that $i_{n,m, u}$ depends on $n, m$ and $u$. In the next lemma, we continue to use the above notation.

\begin{lema}\label{key5}
There exists $C_5 >0$ such that for any $0\leq i\leq p$, any $n, m\in\N$ and $u\in B_n(Z)$ 
we have
\begin{equation*}
\alpha^Z_{n+m+i} (u)\geq C_5 \alpha^Z_m\sup\{f_{n}\vert_Z(z): z\in [u]\}.
\end{equation*}
Note that $C_5$  is independent of $Z$.
\end{lema}

\begin{proof} The proof can be found in  Section \ref{lemmas}
\end{proof}

Now we apply Lemma \ref{key5} to show that $\mu$ is $\sigma_Z$-invariant. Let $u \in B_{n}(Z)$ be fixed and
set $M_2=\max\{0, P(\F)\}$. Letting $l=m+i$ for $m \in \N$ and $0\leq i \leq p$, 
\begin{align*}
\nu(\sigma_{Z}^{-l}[u])=&\sum_{v\in B_{l}(Z), vu \in B_{l+n}(Z)} \nu([vu])
\geq \sum_{vu\in B_{l+n}(Z)}\frac{C_3}{M}e^{-(l+n)P(\F\vert_Z)}\sup\{f_{n+l}\vert_Z(z): z\in [vu]\}\\
&\geq \frac{C_3C_5}{M}e^{-(m+i+n)P(\F\vert_Z)}\alpha^Z_{m}\sup\{f_{n}\vert_Z(z): z\in [u]\} \geq \frac{C_3C_5}{MC_4e^C}e^{-pM_2}\nu([u]),
\end{align*}
where in the last inequality we use (\ref{key0}). Using \ref{a0}, similarly, we obtain 
\begin{align*}
\nu(\sigma_{Z}^{-l}[u])\leq \frac{C_4e^CM}{C_1C_3}\nu([u]).
\end{align*}
Therefore, using the similar arguments as in the proof of  \cite[Theorem 5]{b2}, there exist
$\bar{C_3}, \bar{C_4}>0$ such that for $u\in B_{n}(Z)$ and $x \in [u]$ we have
\begin{equation}
\bar{C_3}\leq \frac{\mu([u])}{e^{-nP(\F\vert_{Z})}f_n\vert_{Z}(x)}\leq \bar{C_4}. 
\end{equation}
Thus $\mu$ is a Gibbs measure on $Z$. It is $\sigma_Z$- invariant because it is a weak limit of invariant measures.
By Theorem \ref{feng}, $\mu$ is the unique invariant ergodic Gibbs measure and the unique equilibrium measure for $\F\vert_{Z}$.
Hence, for $n \geq q$,  if we let $\mu_{l_n}$ be the $\sigma \vert_{Z_{l_n}}$- invariant Gibbs measure on $Z_{l_n}$,  then it satisfies for each $k\in\N$,  $u\in B_{k}(Z_{l_n})$ and every $z \in [u]$, 
 \begin{equation}\label{gibbsp}
\bar{C_3}\leq \frac{\mu_{l_n}([u])}{e^{-kP(\F \vert_{Z_{l_n}})}f_k\vert_{Z_{l_n}}(z)}\leq \bar{C_4}.
 \end{equation} 
 Clearly $\bar{C_3}$ and $\bar{C_4}$ are independent of $Z_{l_n}$.
 \end{proof}
 
In the following proof, we continue to use the notation  of the $\sigma \vert_{Z_{l_n}}$-invariant Gibbs measure  $\mu_{l_n}$  on $Z_{l_n}$ satisfying  (\ref{gibbsp}). 
 The idea in the rest of the proof  is basically the same as in \cite[Theorem 4.1]{iy1}.  However, techniques used here are slightly different, taking into account of \ref{a1}. We include some details for completeness.

\begin{proof}[Proof of Theorem \ref{main2}]
We show that the sequence $\{\mu_{l_n}\}_{n=q}^{\infty}$ of $\sigma$-invariant Borel probability measures on $X$ is tight. For this purpose, we apply Prohorov's theorem to the sequence $\{\mu_{l_n}\}_{n=q}^{\infty}$. We note that the same proof of \cite[Theorem 4.1] {iy1} holds (see also the proof of \cite[Lemma 2.7]{mu2}). Here we only state how we modify using the notation of \cite[Theorem 4.1]{iy1}.

We first note that in the proof the Gibbs property of $\mu_{l_n}$ and the property \ref {a0} of $\F\vert_{Z_{l_n}}$ are applied. Secondly the fact that, for an irreducible Markov shift $X$, $X \cap \prod_{k\geq 1} [1, n_k]$ is  a compact subset of $X$ is used  (see proof of \cite[Theorem 4.1]{iy1} for details).  Since we consider a finitely irreducible countable sofic shift $X$, there exist an irreducible countable Markov shift  $\bar{X}$ and one-block factor map  $\pi:\bar{X}\rightarrow X$ such that $\vert \pi^{-1}(i)\vert <\infty$ for each $i\in \N$.  
For a fixed $k$, we first  consider preimages of $[1,n_k]$ and call it $P_{n_k}$. Note that $P_{n_k}$ is a finite set. Then $ \bar{X} \cap \prod_{k\geq 1} P_{n_k}$ is a compact subset of $\bar{X}$. Thus $X \cap \prod_{k\geq 1} [1, n_k]$ is a compact subset of $X$.

Therefore, we conclude that  there exists a convergent subsequence $\{\mu_{{l}_{n_k}}\}_{k=1}^{\infty}$ of $\{\mu_{l_n}\}_{n=q}^{\infty}$. We denote by $\mu$ a limit point of this subsequence. Then $\mu$ is $\sigma$-invariant on $X$. By 
(\ref{gibbsp}), letting $l_{n_k}\rightarrow \infty$, we obtain
for $n\in \N$ , $u\in B_n(X)$  and each $ x\in [u]$ that,
\begin{equation}\label{gg}
\bar {C_3} \leq \frac{\mu\left([u]\right)}{e^{-nP(\F)}f_{n}(x)} \leq \bar{C_4}.
\end{equation}
Therefore, $\mu$ is a Gibbs measure for $\F$ on $X$. Next we show that $\mu$ is ergodic.  In oder to show this we apply the following lemma.

\begin{lema}\label{key1}

For fixed allowable words $u \in B_n(X), v \in B_l(X)$ and $t\in\N$, 
\begin{equation*}
\begin{split}
&\sum_{ua_1\dots a_{i+t}v\in B_{n +l+t+i}(X), 0\leq i\leq 2p} \sup \{f_{n+l+t+i}(x): x \in [ ua_1\dots a_{t+i}v]\}\\
& \geq D^2\sup\{f_n(x): x\in [u]\}\sup\{f_l(x): x\in [v]\}Z_{t}(\F).
\end{split}
\end{equation*}
\end{lema}
\begin{proof} The proof can be found in Section \ref{lemmas}.
\end{proof}

Now we show that any invariant Gibbs measure for $\F$ is ergodic.
In particular, in the following, we show that $\mu$ is ergodic by proving that there exists $C_6>0$ such that given $u\in B_{n}(X)$, $v\in B_{l}(X)$ and $t\in\N$, there exists $0\leq i_{u,v,t}\leq 2p$ such that $\mu([u]\cap \sigma^{-(n+t+i_{u,v,t})}([v]))\geq ({C_6}/({2p+1}))\mu([u])
\mu([v])$. Note that the same proof holds for any invariant Gibbs measure for $\F$.

Define $\alpha_n=\sum_{i_1\dots i_{n}\in B_{n}(X)}\sup\{f_n(x): x\in [i_1\dots i_{n}]\}$. 
Let $M_2=\max\{0, P(\F)\}$. By applying Lemma \ref{key1}, 
\begin{align*}
&\sum_{i=0}^{2p}\mu([u]\cap\sigma^{-(n+t+i)}([v]))
=\sum_{i=0}^{2p} \sum_{ua_1\dots a_{t+i}v \in B_{n+l+t+i}(X)} \mu([ua_1\dots a_{t+i}v])\\
&\geq  \frac{\bar{C_3} e^{-(n+l+t)P(\F)-2pM_2}}{M} \sum_{i=0}^{2p} \sum_{ua_1\dots a_{t+i}v\in B_{n+l+t+i}(X)} 
 \sup\{f_{n+l+t+i}(x): x\in [ua_1\dots a_{t+i}v]\}\\
&\geq  \frac{\bar{C_3} D^2e^{-(n+l+t)P(\F)-2pM_2}}{M} \alpha_t \sup\{f_{n}(x): x\in [u]\}\sup\{f_{l}(x): x\in [v]\}\\
&\geq  \frac{\bar{C_3} D^2e^{-2pM_2}}{M\bar{C_4}^2e^C} \mu([u])\mu([v]) ,
\end{align*}
where in the third inequality we use Lemma \ref{key1} and in the last inequality we use (\ref{gg}).
Now letting $C_6=({\bar{C_3} e^{-2pM_2}}D^2)/({M\bar{C_4}^2e^C})$, there exists $0\leq i_{u,v,t}\leq 2p$ such that 
\begin{equation*}
\mu([u]\cap \sigma^{-(n+t+i_{u,v,t})}([v]))\geq ({C_6}/({2p+1}))\mu([u])
\mu([v]). 
\end{equation*}
The Gibbs property with ergodicity implies that $\mu$ is the unique invariant ergodic measure on $X$ that satisfies the Gibbs property for $\F$. Finally we show that, if in addition, 
$$\sum_{i\in \N}\sup \{\log f_1(x):x\in [i]\} \sup \{ f_1(x):x\in [i]\}>-\infty,$$ then  
 the unique invariant ergodic Gibbs measure $\mu$ for $\F$ is the unique equilibrium measure for $\F$.
We claim that 
$$\sum_{i\in \N}\sup \{\log f_1(x):x\in [i]\} \sup \{ f_1(x): x\in [i]\}>-\infty \textnormal { if and only if } 
-\sum_{i\in \N} \mu([i])\log \mu([i])<\infty.$$ 
To see this,  by (\ref{gg}), 
\begin{align*}
&\sum_{i\in \N} \mu([i])\log \mu([i])\leq \sum_{i\in \N} \bar{C}_4e^{-P(\F)}\sup\{f_1(x):x\in [i]\} \log (\bar{C}_4e^{-P(\F)}\sup\{f_1(x):x\in [i]\})\\ &\leq  \bar{C}_4e^{-P(\F)}(-P(\F)+\log \bar{C}_4)Z_1(\F)
+\bar{C}_4e^{-P(\F)} \sum_{i\in \N} \sup\{f_1(x):x\in [i]\} \log (\sup\{f_1(x):x\in [i]\}).
\end{align*}
Similarly, we can prove the other direction. 
Since for all $n\in \N$
$$h_{\mu}(\sigma)=-\lim_{n\rightarrow\infty}\frac{1}{n}\sum_{u_n\in B_n(X) } 
\mu([u_n])\log \mu([u_n]) \leq -\frac{1}{n}\sum_{u_n\in B_n(X) } \mu([u_n])\log \mu([u_n]),$$
we obtain that $h_{\mu}(\sigma)<\infty$. We note that for $n\in \N$, 

$$ \frac {1}{n}\int \log f_n d\mu \leq \frac {1}{n}\sum_{u_n\in B_n(X) }\sup\{\log f_n(x):x\in [u_n]\} \mu([u_n])\leq\frac {M}{n}\int \log f_n d\mu.$$  
Using (\ref{gg}), a simple calculation shows that
$$h_{\mu}(\sigma)+\lim_{n\rightarrow\infty }  \frac {1}{n}\int \log f_n d\mu=P(\F).$$
Thus  $\lim_{n\rightarrow\infty } (1/{n})\int \log f_n d\mu>-\infty$. Hence $\mu$ is an equilibrium measure.

To show that $\mu$ is the unique equilibrium measure, we use the same arguments as in \cite{KR} and only mention modified parts for our setting. 
As in \cite[Lemma 3.9]{KR}, we first claim that if $\nu\neq \mu$ is an equilibrium measure for $\F$ then $\nu$ is absolutely continuous with respect $\mu$. 
Observe that given a sequence $\{C_n\}_{n=1}^{\infty}$, where each $C_n$ is a union of cylinder sets of length $n$ of $X$, by using the concavity of $h(x)=-x\log x$ and the Gibbs property of $\mu$, we obtain
\begin{align*}\label{modify}
 0=&n(h_{\nu}(\sigma)+\lim_{n\rightarrow\infty }  \frac {1}{n}\int \log f_n d\nu-P(\F)) 
 \leq \int  \log (f_ne^C) d\nu-nP(\F)
-\sum_{w\in B_n(X)} \nu([w])\log \nu([w]) \\ 
&\leq \log 2 +\nu ([C_n])\log (\frac{\mu([C_n]}{e^C\bar C_3}) +\nu ([X\setminus C_n])\log (\frac{\mu([X\setminus C_n])}{e^C\bar C_3}).
\end{align*}
Applying the proof of  \cite[Lemma 3.9]{KR} by using the above inequalities, we obtain the claim. Then we follow the same proof found in  \cite{KR}  to show the uniqueness.



 \end{proof}

\subsection{Proofs of Lemmas \ref{key7}, \ref{key5}, and \ref{key1}} \label{lemmas}
\begin{proof}[Proof of Lemma \ref{key7}]  

Fix $n\in\N$. It is direct from   (\ref{sofic1})    that for any $m\in \N$, $u\in B_n(Z)$,
\begin{equation*}
\sum_{0\leq i\leq p} \alpha^{Z,u}_{n+m+i} \geq D_1\alpha^Z_m\sup\{f_{n}\vert_Z(z): z\in [u]\}, 
\end{equation*}
where $D_1:=D/M$.
Thus, there exists $0\leq i_{n,m,u}\leq p$ such that 
\begin{equation*}
\alpha^{Z,u}_{n+m+i_{n,m,u}} \geq \frac{D_1}{p+1}\alpha^Z_m\sup\{f_{n}\vert_Z(z): z\in [u]\}.
\end{equation*}
Fix $l> n+2p$ and set $m=l-n-2p$. 
Then there exists $i_{n,m,u}$ such that 
\begin{equation}\label{key6}
\alpha^{Z,u}_{l-2p+i_{n,m,u}} \geq \frac{D_1}{p+1}\alpha^Z_{l-2p-n}\sup\{f_{n}\vert_Z(z): z\in [u]\}.
\end{equation}
Now take $w^*\in W$ such that $\vert w^*\vert=p$. Take $ua_1\dots a_{l-n-2p+i_{n,m,u}}\in B_{l-2p+i_{n,m,u}}(Z)$ and call it $v$. 
Then by Lemma \ref{forsofic} there exists $w\in W$ such that $vww^*$ is an allowable word of $Z$ and 
\begin{equation*}
\sup\{f_{l-2p+i_{n,m,u}+\vert w\vert+p}\vert_Z (z): z\in [vww^*]\} 
\geq D_1 \sup\{f_{l-2p+i_{n,m,u}}\vert_Z (z): z\in [v]\} 
\sup\{f_{p}\vert_Z (z): z\in [w^*]\}.
\end{equation*}
In the similar manner, we can take $\bar w\in W$ such that 
\begin{align*}
&\sup\{f_{l+i_{n,m,u}+\vert w\vert+\vert \bar{w}\vert }\vert_Z (z): z\in [vww^*\bar{w}w^*]\} \\
&\geq {D_1}^2 \sup\{f_{l-2p+i_{n,m,u}}\vert_Z (z): z\in [v]\} 
(\sup\{f_{p}\vert_Z (x): x\in [w^*]\})^2.
\end{align*}
Let  $\vert w\vert=q_1$,  $\vert \bar{w}\vert=q_2$ and write $ww^* \bar{w}w^*=w_1\dots w_{2p+q_1+q_2}.$
Then using \ref{a0},
\begin{align*}
&\sup\{f_{l+i_{n,m,u}+q_1+q_2}\vert_Z (z): z\in [vww^*\bar{w}w^*]\} \\
&\leq e^{C} \sup\{f_{l}\vert_Z (z): z\in [vw_1\dots w_{2p-i_{n,m,u}}]\} \sup\{f_{i_{n,m,u}+q_1+q_2}\vert_Z (z): z\in [w_{2p-i_{n,m,u}+1}\dots w_{2p+q_1+q_2}]\} \\
&\leq e^{3pC} \sup\{f_{l}\vert_Z (z): z\in [vw_1\dots w_{2p-i_{n,m,u}}]\} \max_{0\leq i\leq 3p}Z_1(\F)^i,
\end{align*}
if $i_{n,m,u}+q_1+q_2 \geq 1$.
If $i_{n,m,u}=q_1=q_2=0$, then the second line in the above inequalities is simplified. 
If we let $M'=\max_{0\leq i\leq 3p}Z_1(\F)^i$, then 
\begin{align*}
&\sup\{f_{l}\vert_Z (z): z\in [vw_1\dots w_{2p-i_{n,m,u}}]\} \geq \frac{{D_1}^2}{e^{3pC}M'}
\sup\{f_{l-2p+i_{n,m,u}}\vert_Z (z): z\in [v]\} 
(\sup\{f_{p}\vert_Z (z): z\in [w^*]\})^2\\
&\geq \frac{{D_1}^2}{e^{3pC}M'M^2}
\sup\{f_{l-2p+i_{n,m,u}}\vert_Z (z): z\in [v]\} 
(\sup\{f_{p} (y): y\in [w^*]\})^2,
\end{align*}
where in the last inequality we use the fact that $\F$ is a Bowen sequence.
Let $\bar{m}=\min_{w\in W}(\sup\{f_{p} (y): y\in [w]\})^2$. Then 
summing over all allowable words $a_1\dots a_{l-n-2p+i_{n,m,u}}$ such that $ua_1\dots a_{l-n-2p+i_{n,m,u}}\in B_{l-2p+i_{n,m,u}}(Z)$, we obtain that 
\begin{equation*}
\alpha^{Z,u}_{l} \geq \frac{(\sup\{f_{p} (y): y\in [w^*]\})^2{D_1}^2}{e^{3pC}M'M^2(p+1)}\alpha^{Z,u}_{l-2p+i_{n,m,u}} \geq \frac{\bar{m}{D_1}^2}{e^{3pC}M'M^2(p+1)}\alpha^{Z,u}_{l-2p+i_{n,m,u}},\end{equation*}
and combining with (\ref{key6}) the result follows. 
\end{proof}

\noindent

\begin{proof}[Proof of Lemma \ref{key5}]
Fix $n, m\in \N$ and $u\in B_n(Z)$.  There exists $0\leq i_{n,m,u}\leq p$ such that 
$\alpha^Z_{n+m+i_{n,m,u}} (u)\geq (D_1/(p+1))\alpha^Z_m\sup\{f_{n}\vert_Z(z): z\in [u]\}$.
We first consider the case when $p\geq 2$.
Let $i_{n,m, u}=i_0$ and assume $i_0\geq 1$.
Let $a_1\dots a_{m+i_0}u\in B_{n+m+i_{0}}(Z)$ and call it $v$.
Let $w^*=w^*_1\dots w^*_p\in W$ such that  $\vert w^*\vert=p$.
Take $\bar{C}=\max_{0\leq i\leq 2p} Z_1(\F)^{i}$ Also, take  $D_{W}=(1/M)\min_{w \in W}\sup\{f_{\vert w\vert}(x):x\in [w]\}.$
Then by Lemma \ref{forsofic} there exists $w\in W$ such that 
\begin{equation}\label{useful}
\sup\{f_{n+m+i_{0}+p+\vert w\vert}\vert_Z(z): z\in [w^*wv]\}
\geq \frac{D}{M}\sup\{f_{p}\vert_Z (z): z\in [w^*]\}\sup\{f_{n+m+i_{0}}\vert_Z(z): z\in [v]\}.
\end{equation}
First we show that there exists $C_1>0$  such that for any $j\in \N$  such that $i_0+j\leq p$, 
\begin{equation}\label{goal}
\alpha^Z_{n+m+i_0+j} (u)\geq   C_1 \alpha^Z_{n+m+i_0} (u).  
\end{equation}
Fix $j$ and we consider two cases depending on $\vert w\vert$, $\vert w\vert> j$ and $\vert w\vert \leq j$.
Let $w=w_1\dots w_k$ and suppose $k> j$. 
Since 
\begin{align*}
&\sup\{f_{n+m+i_0+p+k}\vert_Z(z): z\in [w^*wv]\}\\
&\leq e^C\sup\{f_{p+k-j}\vert_Z(z): z\in [w^*w_1\dots w_{k-j}]\}
\sup\{f_{n+m+i_{0}+j}\vert_Z(z): z\in [w_{k-j+1}\dots w_{k}v]\}\\
&\leq  e^{2pC} \bar C  \sup\{f_{n+m+i_{0}+j}\vert_Z(z): z\in [w_{k-j+1}\dots w_{k}v]\},
\end{align*}
applying (\ref{useful}), we obtain
\begin{align}
&\sup\{f_{n+m+i_{0}+j}\vert_Z(z): z\in [w_{k-j+1}\dots w_{k}v]\}
\\
&\label{goal1} \geq 
 \frac{ D}{e^{2pC} \bar CM} \sup\{f_{p}\vert_Z (z): x\in [w^*]\}
\sup\{f_{n+m+i_0}\vert_Z(z): z\in [v]\}.
\end{align}
 
Next suppose $k\leq j\leq p-i_0$. Then 
\begin{align}
&\sup\{f_{n+m+i_0+p+k}\vert_Z(z): z\in [w^*wv]\}\\\label{care}
&\leq e^C\sup\{f_{p-(j-k)}\vert_Z(z): z\in [w^*_1\dots w^*_{p-(j-k)}]\}
\sup\{f_{n+m+i_{0}+j}\vert_Z(z): z\in [w^*_{p-(j-k)+1}\dots w^*_{p}wv]\}.
\end{align}
Hence 
\begin{align}
&\sup\{f_{n+m+i_{0}+j}\vert_Z(z): z\in [w^*_{p-(j-k)+1}\dots w^*_{p}wv]\}\\
&\label{goal2} \geq 
 \frac{ D}{e^{pC} \bar CM} \sup\{f_{p}\vert_Z (z): x\in [w^*]\}\sup\{f_{n+m+i_0}\vert_Z(z): z\in [v]\}.
\end{align}

For each $a_1\dots a_{m+i_0}u\in B_{n+m+i_{0}}(Z)$,  finding $w$ satisfying (\ref{useful}) and applying (\ref{goal1}) or (\ref{goal2}), we obtain
\begin{equation}\label{goal3}
\alpha^Z_{n+m+i_0+j} (u)\geq   \frac{ DD_W}{e^{2pC} \bar C}    \alpha^Z_{n+m+i_0} (u).  
\end{equation}
 
Next we show that there exists $C'_1>0$ such that for each $j\in\N$,  $0\leq j\leq i_0\leq p$, we have 
$\alpha^Z_{n+m+i_0-j} (u)\geq   C'_1 \alpha^Z_{n+m+i_0} (u)$. Fix $j$.
For each $v=a_1\dots a_{m+i_0}u\in B_{n+m+i_0}(Z)$, 
\begin{align*}
&\sup\{f_{j}\vert_Z(z): z\in [a_1\dots a_j]\}\sup\{f_{n+m+i_0-j}\vert_Z(z): z\in [a_{j+1}\dots a_{m+i_0}u]\}\\
&
\geq e^{-C} \sup\{f_{n+m+i_0}\vert_Z(z) z \in [v]\}.
\end{align*}
Noting that 
$\sup\{f_{j}\vert _Z(z): z\in [a_1\dots a_j]\}\leq e^{(p-1)C} \bar{C}$, we obtain
\begin{equation}\label{below}
\alpha^Z_{n+m+i_0-j} (u)\geq  \frac{1}{\bar {C}e^{pC}} \alpha^Z_{n+m+i_0} (u).
\end{equation}
For the case when $i_0=0$, we make similar arguments. We note that  (\ref{care}) is not used (calculation is simplified) when $i_0=0, j=p$ and $k=0$.
For the case when $p=1$, we consider the case when $i_0=0,1$ in a similar manner.
Hence we obtain the results.
 \end{proof}
\noindent
\begin{proof}[Proof of Lemma \ref{key1}]
For a fixed $t\in \N$, fix $c\in B_{t}(X)$. Then given $v$ and $c$,  there exists $w_1\in B_{\vert w_1\vert }(X), $ $0\leq \vert w_1\vert \leq p$ such that 
\begin{equation}\label{k0}
\sup \{f_{t+\vert w_1\vert +l}(x): x \in [ cw_1v]\} \geq D\sup\{f_t(x): x\in [c]\}\sup\{f_l(x): x\in [v]\}.
\end{equation}
Therefore, for fixed $u$ and $cw_1v$ above, there  exists $w_2\in B_{\vert w_2\vert }(X), $ $0\leq \vert w_2\vert \leq p$ such that 
\begin{align}
&\sup \{f_{n+\vert w_2 \vert +t+\vert w_1\vert +l}(y): y \in [uw_2cw_1v]\} \label{k2}\\ 
& \geq D\sup\{f_n(x): x\in [u]\}\sup \{f_{t+\vert w_1\vert +l}(x): x \in [ cw_1v]\} \label{k3}\\
 & \label{k4}  \geq D^2\sup\{f_n(x): x\in [u]\}\sup\{f_t(y): x\in [c]\}\sup\{f_l(x): x\in [v]\}. 
\end{align}
Summing over all allowable words $c\in B_{t}(X)$, each of which satisfies (\ref{k0}) and (\ref{k2})-(\ref{k4}) with some $w_1, w_2$, 
we obtain the result.
\end{proof}

\section{Application to Hidden Gibbs measures on shift spaces over countable alphabets} \label{hiddeng}
In this section, we apply the results in the previous sections to problems on factors of invariant Gibbs measures. Let $\pi: X \rightarrow Y$ be a one-block factor map between countable sofic shifts such that $\vert \pi^{-1}(i) \vert<\infty$ for each $i\in \N$. For every measure $\mu \in M(X, \sigma)$ 
the map $\pi$ induces a measure $\nu \in M(Y, \sigma)$ defined by
\begin{equation*}
\nu(B)= \pi \mu (B) := \mu (\pi^{-1} B),
\end{equation*}
where $B \subset Y$ is any Borel set. If the original measure $\mu$ is a Gibbs measure then the measure $\nu$, which is a factor of a Gibbs measure, is sometimes called \emph{hidden Gibbs measure}. Determining the properties of $\pi \mu$ is a problem that has been addressed in different settings. 
In statistical mechanics, it has been found that non-Gibbs measures can occur as images of Gibbs measures under Renormalization Group transformations and generalizations of Gibbs measures have been studied (see for  example \cite{E, EFS}).

The study of this type of measure also has attracted a great deal of attention in dynamical systems. 
For an overview of the subject, see the survey article by Boyle and Petersen \cite{BP}. 
The factor of the Gibbs measure for a continuous function need not be Gibbs for a continuous function but may be for a sequence of continuous functions. 

The main goal of this section is to study factors of Gibbs measures on finitely irreducible countable sofic shifts.
Technically, we make use of the thermodynamic formalism developed in the article, in particular the results in Section \ref{sectiongibbs} and apply a similar approach as in \cite{Y}. 
Let $(X, \sigma_X)$ and $(Y,\sigma_Y)$ be finitely irreducible countable sofic shifts.
For a one-block factor map $\pi: X\rightarrow Y$, $n\in \N, y=(y_1,\dots y_n,\dots)\in Y$, let $E_n(y)$ be a set consisting of exactly one point from 
each cylinder $[x_1\dots x_n]$ such that $\pi(x_1\dots x_n)=y_1\dots y_n$.
Given a sequence $\F=\{\log f_n\}_{n=1}^{\infty}$ on $X$,
define 
\begin{equation*}
g_n(y)= \sup_{E_n(y)} \left\{\sum_{x \in E_n(y)} f_n(x) \right\}.
\end{equation*}
We continue to use the notation in this section. Recall that we identify the set of a countable alphabet with $\N$.

\begin{teo}\label{hgibbs}
Let $(X, \sigma_X)$ be a finitely irreducible countable sofic shifts,  $(Y, \sigma_Y)$ a subshift on a countable alphabet and $\pi:X\rightarrow Y$  a one-block factor map such that for each $i \in \N$, $\vert \pi^{-1}(i)\vert < \infty$.
Let $\F=\{\log f_n\}_{n=1}^{\infty}$ be an almost-additive Bowen sequence on $X$. 
If $Z_1(\F)<\infty$, then there exists a unique invariant ergodic Gibbs measure $\mu$ for $
\F$ and the projection $\pi\mu$ of the measure $\mu$ is the unique invariant ergodic Gibbs measure for $\mathcal{G}=\{\log g_n\}_{n=1}^{\infty}$. Moreover,
\begin{align}
P_{G}(\F)=P(\F)&=\sup _{\mu\in M(X, \sigma_X)}\left\{h_{\mu}(\sigma_X)+\lim_{n\rightarrow\infty}\frac{1}{n}\int\log f_n d\mu:  \lim_{n\rightarrow\infty}\frac{1}{n}\int\log f_n d\mu>-\infty \right\} \label{x}\\
&=\sup  _{\nu\in M(Y, \sigma_Y)}\left\{h_{\nu}(\sigma_Y)+\lim_{n\rightarrow\infty}\frac{1}{n}\int\log g_n d\nu:  \lim_{n\rightarrow\infty}\frac{1}{n}\int\log g_n d\nu>-\infty \right\} \\
&=P(\mathcal{G})
<\infty.\label{y}
\end{align}

In addition, if $\sum_{i\in \N}\sup \{\log f_1(x):x\in [i]\} \sup \{ f_1(x):x\in [i]\}>-\infty$, then  
 $\mu$ is the unique equilibrium measure for $\F$ and $\pi\mu$ is the  unique equilibrium measure for $\G$.  In particular, if $(X, \sigma_X)$ is a factor of a finitely primitive countable Markov shift,  then  $\limsup$ in the definition (\ref{gurev}) of $P_{G}(\F)$ can be replaced by 
$\lim$.

\end{teo}

\begin{rem}
In \cite[Theorem 3.1] {Y},  almost-additive Bowen sequences on finitely primitive subshifts are considered and the proof of Theorem \ref{hgibbs} generalizes it for those on finitely irreducible subshifts. 
\end{rem}

\begin{rem}
Another approach to show \cite[Theorem 3.1]{Y} is to apply \cite[Proposition 3.7]{Fe4} concerning relative variational principle. However, in \cite[Proposition 3.7]{Fe4}, shift spaces are assumed to be compact (subshifts on  finite alphabets) and so we cannot apply the proposition directly to show Theorem \ref{hgibbs}. 
\end{rem}

\begin{proof}[Proof of Theorem \ref{hgibbs}]
We first note that $Y$ is an irreducible countable sofic shift because $X$ is an irreducible countable sofic shift. Since $X$ is finitely irreducible, there exist $p\in \N$ and a finite set $W_1$ defined in Definition \ref{fi}.

 By Lemma \ref{d} the sequence $\F=\{\log f_n\}_{n=1}^{\infty}$ satisfies \ref{a0}, \ref{a1} with $p$, \ref{a3} with $W_1$ and \ref{a2}. Hence, by Theorem \ref{main2},
there exists a unique invariant ergodic Gibbs measure $\mu$ for $
\F=\{\log f_n\}_{n=1}^{\infty}$. Clearly  $\mathcal{G}=\{\log g_n\}_{n=1}^{\infty}$ is a Bowen sequence.
We show that $\mathcal{G}$ satisfies \ref{a0}, \ref{a1}, \ref{a3} and \ref{a2}.  
By  \cite[Lemma 3.4]{Y}, the sequence $\mathcal{G}$ satisfies \ref{a0}. To verify that condition \ref{a2} is fulfilled, note that for each symbol $i\in\N$ in $Y$ we have that
\begin{equation*}
\sup\{g_{1}(y): y\in [i]\}\leq \sum_{j \in \N, \pi (j)=i} \sup\{f_1(x): x \in [j]\}.
\end{equation*}
Then
$Z_1(\G)\leq \sum_{i\in\N}\sum_{j \in \N, \pi (j)=i} \sup\{f_1(x): x \in [j]\}=Z_1(\F)< \infty.$
Next we show that  $\mathcal{G}$ satisfies \ref{a1}.
For $y=(y_1,\dots, y_n, \dots)\in Y$, by the Bowen property,
\begin{align}
&\frac{1}{M}\sum_{x_1\dots x_n \in B_n(X), \pi (x_1\dots x_n)=y_1\dots y_n} \sup\{f_n(x): x\in [x_1\dots x_n]\}\label{g1} \leq g_{n}(y)\\\label{equivalent}
&\leq \sum_{x_1\dots x_n \in B_n(X), \pi (x_1\dots x_n)=y_1\dots y_n} \sup\{f_n(x): x \in [x_1\dots x_n]\}. 
\end{align}
We note that if $X$ is an irreducible subshift on a finite alphabet (compact case), then \cite[Lemma 5.7]{Fe4}  and  (\ref{equivalent}) imply that $\mathcal{G}$ satisfies \ref{a0} and \ref{a1}. For completeness, we  present a proof in  this non-compact setting. Since $p$ is a weak specification number of $X$, $Y$ also satisfies the weak specification
property with the specification number $p$. In particular,
for given $u\in B_n(Y)$ and  $v \in B_m(Y)$, $n, m\in \N$, there exists ${w_1}\in \pi(W_1)$ (see Example \ref{exfactor} for the notation), $0\leq \vert {w_1}\vert \leq p$ such that $u{w_1}v$ is an allowable word of $Y$. 
For $w\in \pi(W_1)$ such that $uwv$ is allowable in $Y$, 
pick a $y_w\in[uwv]$.  
Note that  given any $x_1\dots x_n\in \pi^{-1}(u)$ and $x'_1\dots x'_m\in \pi^{-1}(v)$, there exists $w_0\in W_1$ such that $x_1\dots x_nw_0x'_1\dots x'_m$ is allowable with the property \ref{a1} and 
$\pi(x_1\dots x_nw_0x'_1\dots x'_m)=u\pi (w_0)v$. Then
\begin{align*}
&\sum_{w\in \pi(W_1)}\sup\{ g_{n+m+ \vert w\vert}(y): y\in [uwv]\}\geq \sum_{w\in \pi(W_1)} g_{n+m+ \vert w\vert}(y_w)\\
&\geq  \sum_{w\in \pi(W_1)} \frac{1}{M}\sum_{\substack {x_1\dots x_n\bar{w}x'_1\dots x'_m\in B_{n+m+\vert \bar{w}\vert}(X)\\ \pi (x_1\dots x_n\bar{w}x'_1\dots x'_m)=uwv}} \sup\{f_{n+ \vert w\vert+m}(x): x\in [x_1\dots x_n\bar{w}x'_1\dots x'_m]\}\\
&\geq \frac{1}{M}\sum_{\substack{x_1\dots x_n\bar{w}x'_1\dots x'_m\in B_{n+m+\vert \bar{w}\vert}(X)\\ \pi (x_1\dots x_n\bar{w}x'_1\dots x'_m)=uwv}}D\sup\{f_n(x):  x\in [x_1\dots x_n]\}\sup\{f_m(x):  x\in [x'_1\dots x'_m]\}\\
&\geq  \frac{D}{M} \left(\sum_{\substack{x_1\dots x_n\in B_{n}(X)\\\pi(x_1\dots x_n)=u}}\sup\{f_n(x):  x\in [x_1\dots x_n]\} \right)
\left(\sum_{\substack{x'_1\dots x'_m \in B_{m}(X)\\\pi(x'_1\dots x'_m)=v}}\sup\{f_m(x):  x\in [x'_1\dots x'_m]\} \right)\\
&\geq
\frac{D}{M}\sup\{g_{n}(y): y \in [u]\} 
\sup\{g_{m}(y): y\in [v]\},
\end{align*}
where in the third inequality we take $\bar w\in W_1$ such that \ref{a1} holds with $x_1\dots x_n \bar wx'_1\dots x'_m$.
Therefore, there exists $w_1\in \pi(W_1)$ such that $uw_1v$ is allowable in $Y$ and 
\begin{equation}\label{gnbowen}
\sup\{g_{n+\vert w_1\vert+m}(y): y \in [uw_1v]\}
\geq \frac{D}{M\vert \pi(W_1)\vert }\sup\{g_{n}(y): y \in [u]\}\sup\{g_{m}(y): y \in [v]\}.
\end{equation}
Hence $\G$ satisfies \ref{a1} with the same value of $p$ that appears in the weak specification and \ref{a3} with $W=\pi (W_1)$.
By Theorem \ref{main2} the sequence  $\G$ has a unique invariant Gibbs measure $\nu$. The second and fourth equalities in Theorem \ref{hgibbs} hold because of the variational principle.

To complete the proof of the theorem, we apply ideas  found in the proof of \cite[Theorem 3.1]{Y}. Let $\mu$ be the equilibrium measure for $\F$. To show that that $\pi\mu=\nu$, 
observe that the proof of \cite[Theorem 3.7]{Y}  holds in our setting because of the definition of the Gibbs measure.  Hence,  if we define $\tilde{g}_n(y)=g_n(y)e^{-nP(\F)}$ and $\tilde{\G}=\{\log \tilde{g}_n\}_{n=1}^{\infty}$, then there is a unique invariant Gibbs measure $\tilde\nu$  for $\tilde{\G}$ such that $\pi\mu=\tilde\nu$. Hence $\pi\mu=\nu$ and it is the unique Gibbs measure for $\G$. By the definition of topological pressure, it is easy to see that  $Z_n(\G) \leq Z_n(\F)$ and $Z_n(\F)\leq {M} Z_n(\G)$. Thus $P(\F)=P(\G)$. 
Finally, we show that $\nu$ is a unique equilibrium measure by showing that  $\sum_{i\in \N}\sup \{\log g_1(y):y\in [i]\} \sup \{ g_1(y):y\in [i]\}>-\infty$.
Assume that $ \sum_{x_1\in \N} \sup\{f_{1}(x): x\in [x_1]\}\sup\{\log f_{1}(x): x\in [x_1]\}>-\infty.$

Using the definition of $g_1$  and the fact that $\F$ is a Bowen sequence we obtain that 
\begin{align*}
&\sup \{ g_1(y):y\in [y_1]\}\sup \{\log g_1(y):y\in [y_1]\} \\
&\geq \frac{1}{M} \left(\sum_{\substack{x_1\in\N\\ \pi (x_1)=y_1}} \sup\{f_{1}(x): x\in [x_1]\} \right) \log \left(\frac{1}{M}\sum_{\substack{x_1\in \N\\\pi (x_1)=y_1}} \sup\{f_{1}(x): x\in [x_1]\} \right) \\
&\geq \frac{1}{M}\cdot \left(\log \frac{1}{M} \right)\dot \sum_{\substack{x_1\in \N\\ \pi (x_1)=y_1}}\sup\{f_{1}(x): x\in [x_1]\}\\
&+ \frac{1}{M}\sum_{\substack{x_1\in \N\\\pi (x_1)=y_1}} \sup\{f_{1}(x): x\in [x_1]\}\sup\{\log f_{1}(x): x\in [x_1]\}.
\end{align*}

Therefore, summing over all allowable $y_1\in\N$, we obtain the result.
Applying Theorem \ref{main2} we have that $\nu$ is the unique equilibrium measure for $\G$.  For the last statement, we apply Proposition \ref{defaasofic}.
\end{proof}
\begin{teo}\label{ftemperedv}
Let  $(X, \sigma_X)$ be a finitely irreducible countable sofic shift, $(Y, \sigma_Y)$ a subshift on a countable  alphabet 
and $\pi:X\rightarrow Y$ a one-block factor map such that for each $i \in \N$, $\vert \pi^{-1}(i)\vert < \infty$.
Let $\F=\{\log f_n\}_{n=1}^{\infty}$ be an almost-additive sequence on $X$ with tempered variation. Then
\begin{align}
P_{G}(\F)=P(\F)&=\sup _{\mu\in M(X, \sigma_X)}\left\{h_{\mu}(\sigma_X)+\limsup_{n\rightarrow\infty}\frac{1}{n}\int\log f_n d\mu:  \limsup_{n\rightarrow\infty}\frac{1}{n}\int\log f_n d\mu>-\infty \right\} \label{x1}\\
&=\sup _{\nu\in M(Y, \sigma_Y)}\left\{h_{\nu}(\sigma_Y)+\limsup_{n\rightarrow\infty}\frac{1}{n}\int\log g_n d\nu:  \limsup_{n\rightarrow\infty}\frac{1}{n}\int\log g_n d\nu>-\infty \right\} \\
&=P(\mathcal{G}).\label{y1}
\end{align}
If $\sup f_1<\infty$, then $\limsup$ in the above equations can be replaced  by $\lim$.
\end{teo}

\begin{proof}
If $\F$ has tempered variation,  (\ref{gnbowen}) is replaced by
\begin{equation*}\label{gnbowen_new}
\begin{split}
&\sup\{g_{n+\vert w_1\vert+m}(y): y \in [uwv]\}
\\ & \geq \frac{e^{-C}Q}{M_{n+m+p}M_nM_mM_{p}\vert \pi(W_1)\vert }\sup\{g_{n}(y): y \in [u]\}\sup\{g_{m}(y): y \in [v]\},
\end{split}
\end{equation*}
where $Q$ is defined for $\F$ as in Lemma \ref{d}. 
Applying Corollary \ref{special1} and Theorem \ref{thmsofic},  we obtain (\ref{x1}) and (\ref{y1}).
To show $P(\F)=P(\mathcal{G})$, we make similar arguments as in the proof of Theorem \ref{hgibbs} .
\end{proof}
\begin{rem}
We do not know the existence of equilibrium measures for $\F$ and  $\G$  in Theorem \ref{ftemperedv}.
\end{rem}

Next we consider the images of factors of Gibbs measures for single functions. 
Recall the definition of functions in the Bowen class from Section \ref{back}.
\begin{coro}  \label{hgibbsc}
Let  $(X, \sigma_X)$ be a finitely irreducible countable sofic shift, $(Y, \sigma_Y)$ a subshift on a countable alphabet 
and $\pi:X\rightarrow Y$ a one-block factor map such that for each 
$i \in \N$, $\vert \pi^{-1}(i)\vert < \infty$.
Let $f\in C(X)$ be in the Bowen class and suppose $Z_1(f)<\infty$. Then there exists a unique invariant ergodic Gibbs measure $\mu$ for $
f$. Setting $f_n=e^{S_n(f)}$ in $\G$, the projection $\pi\mu$ of the measure $\mu$ is the unique invariant ergodic Gibbs measure for $\mathcal{G}=\{\log g_n\}_{n=1}^{\infty}$. Then
\begin{align}
P_{G}(f)=P(f)&=\sup _{\mu\in M(X, \sigma_X)}\left\{h_{\mu}(\sigma_X)+\int f d\mu:  \int f  d\mu>-\infty \right\} \label{x} \\
&=\sup _{\nu\in M(Y, \sigma_Y)}\left\{h_{\nu}(\sigma_Y)+\lim_{n\rightarrow\infty}\frac{1}{n}\int\log g_n d\nu:  \lim_{n\rightarrow\infty}\frac{1}{n}\int\log g_n d\nu>-\infty \right\} \\
&=P(\mathcal{G})
<\infty.\label{y}
\end{align}
In addition, if $\sum_{i\in \N}\sup \{\log f(x):x\in [i]\} \sup \{ f(x):x\in [i]\}>-\infty$, then  
$\mu$ is the unique equilibrium measure for $f$ and $\pi\mu$ is the  unique equilibrium measure for $\G$. 
 \end{coro}
\begin{proof}
The result is clear by applying Theorem \ref{hgibbs}. 
\end{proof}
\begin{rem}
 This is a generalization of \cite[Corollary 3.2]{Y}. 
 \end{rem}




\section{Other applications}\label{applications}

\subsection{Product of matrices and maximal Lyapunov exponents}

A natural and interesting application of the non-additive version of thermodynamic formalism is the study of the norm of products of matrices. Indeed, let $M_d(\R)$ be the set of real valued $d \times d$ matrices and $\| \cdot \|$ be a sub-multiplicative norm. Let  $\{A_1, A_2, \dots \}$ be a countable set in $M_d(\R)$. Let  $(X, \sigma)$ be a finitely irreducible countable sofic shift. If $w=(i_1, i_2, \dots) \in X$,  define the sequence of functions $\Phi=\{\log \phi_n\}_{n=1}^{\infty}$ by
\[\phi_n(w)= \Vert A_{i_{n}}  \cdots A_{i_2} A_{i_1}\Vert.\]
Since
\begin{equation*}
\Vert AB \Vert \leq \Vert A  \Vert \Vert B \Vert,
\end{equation*}
the sequence $\Phi$ is sub-additive. It is a direct consequence of the sub-additive ergodic theorem \cite{ki} that if $\mu \in M(X, \sigma)$ is an ergodic measure, then for $\mu$-almost every $w\in X$ 
 \begin{equation*}
\lim_{n\to \infty} \frac{1}{n} \int \log \phi_n \ d\mu =  \lim_{n\to \infty} \frac{1}{n} \log \phi_n(w).
\end{equation*}
The number
\[\lambda(w) := \lim_{n\to \infty} \frac{1}{n} \log \phi_n(w),\]
is called \emph{Maximal Lyapunov exponent of $w$}, whenever the limit exists. This number was originally studied in the context in which $X$ is the full shift on a finite alphabet with  a finite collection matrices with strictly positive entries (see the work by 
Furstenberg and Kesten from 1960 \cite{fk}). Ever since, the assumptions on the space and on the matrices has been generalized in wide ranges. The techniques developed in this article allow for another generalization that can be thought of as a non-compact version of the results obtained by Feng in \cite{Fe3}.

\begin{prop} \label{lyap}
Let $(X, \sigma)$ be a finitely irreducible countable sofic shift. Let $\{A_1, A_2, \dots \}$ be a countable set of matrices  in $M_d(\R)$ having non-negative entries. Let $\Phi=\{\log \phi_n\}_{n=1}^{\infty}$  be a the sequence of functions such that $\phi_n:X \to \R$ is defined by $\phi_n(w)= \Vert A_{i_{n}}  \cdots A_{i_2} A_{i_1}\Vert$. If $\Phi$ satisfies \ref{a1}, \ref{a3} and \ref{a2}, then there exists  a unique invariant ergodic Gibbs measure $\mu$ for $\Phi$. Moreover, if in addition 
\begin{equation*}
\sum_{i=1}^{\infty} \|A_i\|  \log \| A_i \| > - \infty
\end{equation*}
 then $\mu$  is the unique equilibrium measure for $\Phi$ on $X$, that is
 \begin{equation*}
 P(\Phi)= h_{\mu}(\sigma) + \lim_{n \to \infty} \frac{1}{n} \int \log \phi_n d \mu.
 \end{equation*}
\end{prop}

Note that $\phi_n$ is constant in cylinders of length $n$, therefore the Bowen condition is satisfied. Proposition \ref{lyap} is an extension of \cite[Proposition 7.1]{iy1} in which the same conclusion was obtained under the assumption that $X$ is a countable Markov shift satisfying the BIP condition and $\Phi$ is almost-additive.

\subsection{The singular value function} Thermodynamic formalism has been used, at least since the mid 1970«s, to study the (Hausdorff) dimension of certain dynamically defined sets. This approach has been rather successful when the dynamical system is conformal. However, in dimension two (or higher) where a typical dynamical system is non-conformal the results obtained are fairly weak. With the purpose of obtaining better estimates on the dimension of non-conformal repellers, Falconer \cite{f} introduced the singular value function. The singular values $s_1(A), s_2(A)$ of a $2\times 2$ matrix $A$ are the eigenvalues, counted with multiplicities, of the matrix $(A^*A)^{1/2}$, where $A^*$ denotes the transpose of $A$. The singular values can be interpreted as the length of the semi-axes of the ellipse which  is the image of the unit ball under $A$.

Let $f : \R^2 \mapsto  \R^2$ be a $C^1$ map and let  $\Lambda \subset \R^2$ be a repeller of $f$.  That is, the set $\Lambda$ is a (not necessarily  compact), $f$-invariant, and the map $f$ is expanding on $\Lambda$, i.e., there exist $c > 0$ and $\beta > 1$ such that 
\[ \|d_xf^n(v) \| \geq c\beta^n \|v\|, \]
for every $x \in \Lambda$, $n \in \N$ and $v \in T_x \R^2$. We will also assume  that there exists  an open set $U \subset \R^2$  such that $\Lambda \subset U$ and $\Lambda = \cap_{n \in \N} f^n(U)$ and that $f$ restricted to $\Lambda$ can be coded by an irreducible countable sofic shift. For each $x \in  \R^2$ and $ v \in  T_xR^2$, we define the \emph{Lyapunov exponent} of $(x,v)$ by
\begin{equation*}
\lambda(x,v):=\limsup_{n \to \infty} \frac{1}{n} \log \| d_xf^n v  \|.
\end{equation*}
For each $x \in \R^2$, there exists a positive integer $s(x)  \leq 2$, numbers $\lambda_1(x) \geq \lambda_2(x)$, and linear subspaces
\[ \{0\} =E_{s(x)+1}(x) \subset  E_{s(x)}(x) \subset E_{1}(x)=T_xR^2,\]
such that
\[E_i(x)=\left\{v \in T_x\R^2 : \lambda(x,v)=\lambda_i(x) \right\}\]
and $ \lambda(x,v)=\lambda_i(x)$ if $v \in E_i(x) \setminus E_{i+1}(x)$.
The functions, $\phi_{i,n}: \Lambda \to \R$  are defined by
\begin{equation*}
 \phi_{i,n}(x)= \log s_i(d_xf^n)  
\end{equation*}
and called \emph{singular value functions}.   It  follows from Oseledets' multiplicative ergodic theorem  that for each finite $f-$invariant measure $\mu$ there exists a set $X \subset \R^2$ of full $\mu$ measure such that
\begin{equation} \label{eq:lya}
 \lim_{n \to \infty} \frac{\phi_{i,n}(x)}{n}= \lim_{n \to \infty} \frac{1}{n} \log s_i(d_xf^n)= \lambda_i(x).
 \end{equation}
It was proved by Barreira and Gelfert  \cite[Proposition 4]{bg} that
if the dynamical system $f$ has dominated splitting  (see \cite[p.234]{b3} for a precise definition) and $\Lambda$ is compact then the sequences 
$\{\phi_{i,n}\}_{n=1}^{\infty}$ are almost-additive. The methods developed in this article allow us to study the singular value function in a broader context. In particular, it is a consequence of the variational principle that
\begin{prop}
Let $(f, \Lambda)$ be a non-conformal repeller that can be coded by an irreducible countable sofic shift. If the singular value functions $\Phi$ satisfy   
 \ref{a1}, \ref{a3} and \ref{a2}, then there exists  a unique invariant ergodic Gibbs measure $\mu$ for $\Phi$. 
\end{prop}
We stress that Gibbs measures are of particular importance in the dimension theory of dynamical systems.

\end{document}